\documentclass[11pt]{article}
\usepackage[T1]{fontenc} % IMPORTANT (to be able to correctly copy and paste from the output pdf)
\usepackage[utf8]{inputenc} % IMPORTANT (ex. for German accents in bibliography)
\usepackage[top=2cm, bottom=2cm, left=2.5cm, right=2.5cm]{geometry} % margins

\usepackage{authblk} %author affiliation
\providecommand{\keywords}[1]
{
	\footnotesize
	\textbf{{Keywords.}} #1
}
\providecommand{\MSC}[1]
{
	\footnotesize
	\textbf{{MSC codes.}} #1
}

\usepackage{amsmath, amsthm, mathrsfs, amssymb, mathtools}
\usepackage{physics} % type dx
\usepackage{textcomp} % type section symbol
\usepackage{bbm} % \mathbbm{1} works well

\usepackage{hyperref}

\usepackage{graphicx} %\includegraphics[width=\textwidth]{240228.png}
\usepackage{epstopdf}
\usepackage[skip=1ex]{caption}
\usepackage{units, wrapfig}
\usepackage{subcaption, cleveref} % all for subfigures, cleveref after hyperref
\usepackage{tikz}
\usetikzlibrary{calc}
\usetikzlibrary{patterns} % fill by patterns
\usetikzlibrary{fpu}

\usepackage{multirow, longtable}
\usepackage{tabvar}
\usepackage{enumitem}
\usepackage{makecell}
\usepackage{xcolor}
\definecolor{niceblue}{rgb}{0.25, 0.4, 0.96}
\hypersetup{colorlinks,urlcolor=niceblue,linkcolor=niceblue,citecolor=niceblue,breaklinks}

\theoremstyle{plain} 
\newtheorem{theorem}{Theorem}[section]
\newtheorem{proposition}[theorem]{Proposition} % continue the numbering as theorem (to avoid confusion)
 
\newtheorem{lemma}[theorem]{Lemma}
\theoremstyle{definition}

\theoremstyle{remark}
\newtheorem{remark}[theorem]{Remark} 

\newcommand{\C}{\mathbb{C}}

\newcommand{\R}{\mathbb{R}} 

\newcommand{\Hso}{\mathrm{H}}
\DeclareMathOperator{\dive}{\mathrm{div}}
\DeclareMathOperator{\pa}{\partial}

\renewcommand{\norm}[1]{\lVert#1\rVert}
\newcommand{\scalar}[1]{\langle#1\rangle}
\newcommand{\ic}{\mathrm{i}} %Complex number typing

\newcommand{\sigmar}{\sigma_{\mathrm{r}}}
\newcommand{\dx}{\,\mathrm{dx}}

\title{On the convergence analysis of one-shot inversion methods}
\author[1]{Marcella Bonazzoli}
\author[1]{Houssem Haddar}
\author[1,2]{Tuan-Anh Vu}
\affil[1]{Inria, UMA, ENSTA Paris, Institut Polytechnique de Paris, 91120 Palaiseau, France}
\affil[2]{Institute of Mathematics, Vietnam Academy of Science and Technology, Vietnam}
\date{Revised and resubmitted on April 10, 2024}

\begin{document}
	\maketitle
	
	\begin{abstract}
		When an inverse problem is solved by a gradient-based optimization algorithm, the corresponding forward and adjoint problems, which are introduced to compute the gradient, can be also solved iteratively. The idea of iterating at the same time on the inverse problem unknown and on the forward and adjoint problem solutions yields the concept of one-shot inversion methods. We are especially interested in the case where the inner iterations for the direct and adjoint problems are incomplete, that is, stopped before achieving a high accuracy on their solutions. Here, we focus on general linear inverse problems and generic fixed-point iterations for the associated forward problem. We analyze variants of the so-called multi-step one-shot methods, in particular semi-implicit schemes with a regularization parameter. 
		We establish sufficient conditions on the descent step for  convergence, by studying the eigenvalues of the block matrix of the coupled iterations. Several numerical experiments are provided to illustrate the convergence of these methods in comparison with the classical gradient descent, where the forward and adjoint problems are solved exactly by a direct solver instead. We observe that very few inner iterations are enough to guarantee good convergence of the inversion algorithm, even in the presence of noisy data.    
	\end{abstract}
	
	\noindent\keywords{inverse problems, one-shot methods, convergence analysis, parameter identification}
	
	\medskip
	\noindent\MSC{35R30, 65F10, 65N21}
	
	\normalsize
	\section{Introduction}
	
	For large-scale inverse problems, which often arise in real life applications, the solution of the corresponding forward and adjoint problems is generally computed using an iterative solver, such as preconditioned fixed point or Krylov subspace methods, rather than exactly by a direct solver, such as LU-type solvers (see e.g.~\cite{TOURNIER2019, audibert2023}). Indeed, the corresponding linear systems could be too large to be handled with direct solvers because of their high memory requirement. In addition, iterative solvers are easier to parallelize on many cores for time speed-up. By coupling the iterative solver with a gradient-based optimization iteration, the idea of \emph{one-step one-shot methods} is to iterate at the same time on the forward problem solution (the state variable), the adjoint problem solution (the adjoint state) and on the inverse problem unknown (the parameter or design variable).  
	If two or more inner iterations are performed on the state and adjoint state before updating the parameter (by starting from the previous iterates as initial guess for the state and adjoint state), we speak of \emph{multi-step one-shot methods}. 
	Our goal is to rigorously analyze the convergence of such inversion methods. In particular, we are interested in those schemes where the inner iterations on the direct and adjoint problems are incomplete, i.e.~stopped before achieving convergence. Indeed, solving the forward and adjoint problems exactly by direct solvers or very accurately by iterative solvers could be very time-consuming with little improvement in the accuracy of the inverse problem solution.   
	
	The concept of one-shot methods was first introduced by Ta'asan \cite{taasan91} for optimal control problems. Based on this idea, a variety of related methods, such as the all-at-once methods, where the state equation is included in the misfit functional, were developed for aerodynamic shape optimization, see for instance \cite{taasan92, shenoy97, hazra05, schulz09, gauger09} and the literature review in the introduction of \cite{schulz09}. All-at-once approaches to inverse problems for parameter identification were studied in, e.g., \cite{haber01, burger02, kaltenbacher14, kaltenbacher16, kaltenbacher17, tram2019, tram2024}. An alternative method, called Wavefield Reconstruction Inversion (WRI), was introduced for seismic imaging in \cite{leeuwen13}, as an improvement of the classical Full Waveform Inversion (FWI) \cite{tarantola82}. WRI is a penalty method which combines the advantages of the all-at-once approach with those of the reduced approach (where the state equation represents a constraint and is enforced at each iteration, as in FWI), and was extended to more general inverse problems in \cite{leeuwen15}. 
	%Recently, Kaltenbacher \cite{kaltenbacher16} has studied regularization based on all-at-once methods for inverse problems under reasonable assumptions.
	
	Few convergence proofs, especially for the multi-step one-shot methods, are available in the literature. In particular, for non-linear design optimization problems, Griewank \cite{griewank06} proposed a version of one-step one-shot methods where a Hessian-based preconditioner is used in the design variable iteration. 
	%(see Section ``Inexact Block Gauss-Seidel", page 162 \cite{griewank06}) 
	The author proved conditions to ensure that the real eigenvalues of the Jacobian of the coupled iterations are smaller than $1$, but these are just necessary and not sufficient conditions to exclude real eigenvalues smaller than $-1$. In addition, no condition to also bound complex eigenvalues below $1$ in modulus was found, and multi-step methods were not investigated. 
	% (these comments are given in [gauger12] page 103) 
	% see Proposition 2 and 3 in Section 4 \cite{griewank06}. 
	In \cite{hamdi09, hamdi10, gauger12} an exact penalty function of doubly augmented Lagrangian type was introduced to coordinate the coupled iterations, and global convergence of the proposed optimization approach was proved under some assumptions. This particular one-step one-shot approach was later extended to time-dependent problems in \cite{guenther16}.
	
	In this work, we consider variants of multi-step one-shot methods where the forward and adjoint problems are solved using fixed point iterations and the inverse problem is solved using gradient descent methods. In particular, we analyze semi-implicit schemes with a regularization parameter. This is a preliminary work where we focus on (discretized) linear inverse problems. The only basic assumptions we require are the uniqueness of the inverse problem solution and the convergence of the fixed point iteration for the forward problem. To analyze the convergence of the coupled iterations we study the real and complex eigenvalues of the block iteration matrix. We prove that if the descent step is small enough, then the considered one-shot methods converge.  Moreover, the upper bounds for the descent step in this sufficient condition are explicit in the number of inner iterations, in the norms of the operators involved in the problem, and in the regularization parameter.  
	Note that previously in our research report \cite{bon22} we studied (shifted) explicit schemes with no regularization, while here we include a regularized cost functional and analyze semi-implicit schemes. Moreover, for the particular scalar case, in \cite{bon22} we established sufficient and also necessary convergence conditions on the descent step. In the literature \cite{griewank06,hamdi09, hamdi10, gauger12} mentioned above, only the one-step one-shot algorithm has been considered and the analysis for that scheme was far from complete as it was only known that the real eigenvalues of the iteration system has modulus less than one. In the current work, we study the eigenvalues in the whole complex plane and consider also the case of multiple inner iterations on the state variable. In addition, we include the possibility of using some regularization terms for the inverse variable.
	
	This paper is structured as follows. In Section~\ref{sec:intro-k-shot}, we introduce the principle of multi-step one-shot methods and define two variants of these algorithms. 
	Since their analysis is similar, we shall focus on one of them, namely the semi-implicit scheme.
	Then, in Section~\ref{sec:seim-1-shot}, respectively Section~\ref{sec:seim-k-shot}, we analyze the convergence of one-step one-shot methods, respectively multi-step one-shot methods: first, we establish eigenvalue equations for the block matrices of the coupled iterations, then we derive sufficient convergence conditions on the descent step by studying the location of the eigenvalues in the complex plane. 
	Finally, in Section~\ref{sec:num-exp} we do several numerical tests on the performance of the different algorithms on a toy 2D Helmholtz inverse problem. These tests are carried out in four cases. In the first case, the measurements are noise-free and the parameter that we desire to reconstruct is discretized in a low-dimensional space, while in the second case, the measurements are affected by noise and the reconstructed parameter is discretized in a higher dimensional space. 
	In the third and fourth cases, we test the robustness with respect to the size of the discretized problem, and the dependence on the norm of the iteration matrix of the forward problem. In particular, we observe that very few inner iterations are enough to guarantee good convergence of the inversion algorithms, even in the presence of noisy data. 
	
	Throughout this work, $\scalar{\cdot,\cdot}$ indicates the usual Hermitian scalar product in $\C^n$, that is $\scalar{x,y} \coloneqq \overline{y}^\intercal x, \forall x,y\in\C^n$, and $\norm{\cdot}$ the vector/matrix norms induced by $\scalar{\cdot,\cdot}$. We denote by $A^*=\overline{A}^\intercal$ the adjoint operator of a matrix $A\in\C^{m\times n}$, and likewise by $z^*=\overline{z}$ the conjugate of a complex number $z$. 
	The identity matrix is always denoted by $I$, whose size is understood from context. 
	Finally, for a matrix $T\in\C^{n\times n}$, we denote by $\rho(T)$ the spectral radius of $T$, and when $\rho(T)<1$, we define
	\[
	s(T) \coloneqq \sup_{z\in\C, |z|\ge 1}\norm{\left(I-T/z\right)^{-1}},
	\]
	which is further studied in Appendix \ref{app:lems}.
	
	\section{Multi-step one-shot inversion methods}
	\label{sec:intro-k-shot}
	
	We focus on (discretized) linear inverse problems, which correspond to a \emph{direct (or forward) problem} of the form: seek $u \equiv u(\sigma)$ such that 
	% The discretized version of some direct problems (which can be obtained by finite element discretization for example) is of the form
	\begin{equation}
		\label{prb-direct}
		u=Bu+M\sigma+F
	\end{equation}
	where $u\in\R^{n_u}$, $\sigma\in\R^{n_\sigma}$, $B\in\R^{n_u \times n_u}$, $M\in\R^{n_u \times n_\sigma}$ and $F\in\R^{n_u}$.  
	Here $I-B$ is the invertible matrix associated with the direct problem (e.g.~obtained after discretization of a PDE model and applying a preconditioner to the linear system), with parameter $\sigma$. Equation~\eqref{prb-direct} is also referred to as the \emph{state equation} and $u$ is the \emph{state variable}. 
	%	The spaces $\R^{n_u}$ and $\R^{n_\sigma}$ are equipped with arbitrary scalar products. 
	Given $\sigma$, we assume that one solves for $u$ by a fixed point iteration
	\begin{equation}
		\label{iter-prb-direct}
		u_{\ell+1}=Bu_\ell+M\sigma+F, \quad \ell=0,1,\dots
	\end{equation}
	%	which converges for any initial guess $u_0$ if and only if the spectral radius $\rho(B)$ is less than $1$ 
	We indeed assume  $\rho(B)<1$ so that the fixed point iteration \eqref{iter-prb-direct} converges for any initial guess $u_0$ (see e.g. \cite[Theorem 2.1.1]{greenbaum97}). 
	% , thanks to the following result:
	% \begin{lemma}\label{specradi-norm}
		% 	Let $\K^{n\times n}$ be equipped with a norm where $\K$ is either $\R$ or $\C$. For any $A\in\K^{n\times n}$, these following statements are equivalent:
		% 	\begin{enumerate}[label=(\roman*)]
			% 		\item $A^k\overset{k\to\infty}{\longrightarrow}0$.
			% 		\item The iteration $x^{k+1}=Ax^k$, $k=0,1,...$ converges for any initial $x^0\in\K^n$.
			% 		\item $\rho(A)<1$.
			% 		\item There exists a subordinate norm $|.|$ on $\K^{n\times n}$ such that $|A|<1$.
			% 	\end{enumerate}
		% \end{lemma}	
	% Fact: thanks to \cite{allaire}!
	%\noindent This lemma is a corollary of \cite[Theorem 1.3.1]{greenbaum97}. 
	%Then $I-B$ is invertible and $$u=u(\sigma)=(I-B)^{-1}(M\sigma+F).$$ 
	Measuring $g=Hu(\sigma)$, where $H\in\R^{n_g\times n_u}$, we consider the \emph{linear inverse problem} of retrieving $\sigma$ from the knowledge of $g$. 
	%$n_g\in\N^*$, $\R^{n_g}$ is subspace of the Hilbert space $\C^{n_g}$ equipped with freely-chosen scalar product; and we pose the inverse problem: Find $\sigma$ from $f$. 
	Let us set $A\coloneqq H(I-B)^{-1}M$. The inverse problem can be synthetically written as $A\sigma=g-H(I-B)^{-1}F$, which amounts to inverting the ill-conditioned matrix $A$. We shall assume in the following the uniqueness of the solution for this inverse problem, which is equivalent to the injectivity of $A$.
	In summary, we set
	\begin{equation}
		\label{direct-and-inv-prb}
		\begin{array}{lc}
			\mbox{direct problem:} & u=Bu+M\sigma+F,\\
			\mbox{inverse problem:} & \mbox{measure }g=Hu(\sigma),\mbox{ retrieve }\sigma
		\end{array}
	\end{equation}
	with the assumptions:
	\begin{equation}
		\label{hypo}
		\rho(B)<1, \quad H(I-B)^{-1}M \mbox{ is injective}.
	\end{equation}	
	\begin{remark}
		\label{rk:complexB}
		Considering real-valued matrices $B$ and $M$ is not a restrictive assumption. Indeed, the case of complex-valued matrices can be rewritten as a system of real-valued equations by doubling the size of the linear system (see \cite[Section 5]{bon22}).
	\end{remark}
	To solve the inverse problem we write its regularized least squares formulation: given some measurements $g$, we seek the \emph{regularized solution} $\sigma_\alpha^\mathrm{ex}$ defined by
	%$g\coloneqq Hu(\sigma^\mathrm{ex})$ ($g$ can also be a noisy version of $Hu(\sigma^\mathrm{ex})$) where $\sigma^\mathrm{ex}$ the exact solution of the inverse problem and , 
	\begin{equation}
		\label{eq:J}
		\sigma_\alpha^\mathrm{ex}= \mathrm{argmin}_{\sigma\in\R^{n_\sigma}} J(\sigma) \quad \mbox{ where } J(\sigma) \coloneqq \frac{1}{2}\norm{Hu(\sigma)-g}^2+\frac{\alpha}{2}\norm{\sigma}^2, \; \alpha\ge0.
	\end{equation}
	In case we have exact measurements, i.e.~$g=Hu(\sigma^\mathrm{ex})$, one can take $\alpha=0$, so that $\sigma_\alpha^\mathrm{ex}=\sigma^\mathrm{ex}$. However, in practice, the measurements are often affected by noise, hence in general we cannot reconstruct $\sigma^\mathrm{ex}$ but its approximation $\sigma_\alpha^\mathrm{ex}$. 
	Now we solve the minimization problem in \eqref{eq:J}.
	Using the classical Lagrangian technique with real scalar products, we introduce the \emph{adjoint state} $p \equiv p(\sigma)$, which is the solution of 
	$$p=B^*p+H^*(Hu-g)$$
	%where $B^*$ and $H^*$ are respectively the adjoint operator of $B$ and $H$
	and allows us to compute the gradient of the cost functional as
	$$\nabla J(\sigma)=M^*p(\sigma)+\alpha\sigma.$$  
	%We also notice that $p(\sigma^\mathrm{ex})=0$. 
	The classical gradient descent algorithm then reads 
	\begin{equation}
		\label{alg:usualgd}
		\mbox{\textbf{usual gradient descent:}}\quad
		\begin{cases}
			\sigma^{n+1}=\sigma^n-\tau M^*p^n-\tau\alpha\sigma^n, \\
			u^n=Bu^n+M\sigma^n+F,\\
			p^n=B^*p^n+H^*(Hu^n-g),
		\end{cases}
	\end{equation}
	where $\tau>0$ is the descent step size, and the state and adjoint state equations are solved exactly at each iteration step for $\sigma$. Notice that when $F=0$, \eqref{alg:usualgd} is equivalent to $\sigma^{n+1}=\sigma^n-\tau A^*(A\sigma^n-g)-\tau\alpha\sigma^n$. One can also consider a slightly different version where an implicit scheme is applied to the regularization term leading to the following semi-implicit gradient scheme
	%	In \eqref{alg:usualgd} if instead we  update $\sigma^{n+1}=\sigma^n-\tau M^*p^n-\tau\alpha\sigma^{n+1}$, we obtain the
	\begin{equation}
		\label{alg:seimgd}
		\mbox{\textbf{semi-implicit gradient descent:}}\quad
		\begin{cases}
			\sigma^{n+1}=\sigma^n-\tau M^*p^n-\tau\alpha\sigma^{n+1}, \\
			u^n=Bu^n+M\sigma^n+F,\\
			p^n=B^*p^n+H^*(Hu^n-g).
		\end{cases}
	\end{equation}
	%It is well-known that both algorithms converge for sufficiently small $\tau$ 
	%(see e.g. Appendix B of \cite{bon22})
	It can be shown that both algorithms converge for sufficiently small $\tau>0$: for any initial guess, \eqref{alg:usualgd} converges if and only if
	$
	\tau<\frac{2}{\rho(A^*A)+\alpha}
	$
	and \eqref{alg:seimgd} converges if and only if
	$
	(\rho(A^*A)-\alpha)\tau<2.
	$
	%	One can also estimate $\rho(A^*A)\le\norm{A}^2$ to obtain the bounds for $\tau$ depending on $H,M,B$. 
	These results indicate in particular that we gain more stability with the semi-implicit scheme.
	%Hereafter	
	
	We are interested in methods where the direct and adjoint problems are rather solved iteratively as in \eqref{iter-prb-direct}, and where we iterate at the same time on the forward problem solution and the inverse problem unknown: such methods are called \emph{one-shot methods}. More precisely, we are interested in two variants of \emph{multi-step one-shot methods}, defined as follows. Let $n$ be the index of the (outer) iteration on $\sigma$.
	We update $\sigma^{n+1}=\sigma^n-\tau M^*p^n-\tau\alpha\sigma^n$ as in gradient descent methods (or respectively, $\sigma^{n+1}=\sigma^n-\tau M^*p^n-\tau\alpha\sigma^{n+1}$ as in semi-implicit gradient descent methods), but the state and adjoint state equations are now solved by a fixed point iteration method, using just \emph{$k$ inner iterations}
	%and \emph{coupled}:
	%	$$\begin{cases}
		%		u^{n+1}_{\ell+1}=Bu^{n+1}_\ell+M\sigma^{n+1}+F,\\
		%		p^{n+1}_{\ell+1}=B^*p^{n+1}_\ell+H^*(Hu^{n+1}_\ell-g),
		%	\end{cases}
	%	\quad\ell=0,1,\dots,k-1,
	%	\quad\begin{cases}
		%		u^{n+1}\coloneqq u^{n+1}_k,\\ p^{n+1}\coloneqq p^{n+1}_k.
		%	\end{cases}$$
	and as initial guess we naturally choose 
	%$u^{n+1}_0=u^n$ and $p^{n+1}_0=p^n$, 
	the information from the previous (outer) step. We then get the two following variants of multi-step one-shot algorithms
	\begin{equation}\label{alg:k-shot}
		k\mbox{\textbf{-step one-shot:}}\quad
		\begin{cases}
			\sigma^{n+1}=\sigma^n-\tau M^*p^n-\tau\alpha\sigma^n, \\
			u^{n+1}_0 = u^n, p^{n+1}_{0} = p^n, \\
			\mbox{for }\ell=0,1,\dots,k-1: \\
			\quad\left| \begin{array}{l}
				u^{n+1}_{\ell+1}=Bu^{n+1}_\ell+M {\sigma^{n+1}}+F,\\
				p^{n+1}_{\ell+1}=B^*p^{n+1}_\ell+H^*(Hu^{n+1}_\ell -g),
			\end{array}\right. \\
			u^{n+1}\coloneqq  u^{n+1}_k, p^{n+1}\coloneqq  p^{n+1}_k
		\end{cases}
	\end{equation}
	and 
	\begin{equation}
		\label{alg:seim-k-shot}
		\text{
				\textbf{semi-implicit}
				$k$\textbf{-step one-shot:}
		}\quad
		\begin{cases}
			\sigma^{n+1}=\sigma^n-\tau M^*p^n-\tau\alpha\sigma^{n+1}, \\
			u^{n+1}_0 = u^n, p^{n+1}_{0} = p^n, \\
			\mbox{for } \ell=0,1,\dots,k-1: \\
			\quad\left| \begin{array}{l}
				u^{n+1}_{\ell+1}=Bu^{n+1}_\ell+M {\sigma^{n+1}}+F,\\
				p^{n+1}_{\ell+1}=B^*p^{n+1}_\ell+H^*(Hu^{n+1}_\ell -g),
			\end{array}\right. \\
			u^{n+1} \coloneqq  u^{n+1}_k, p^{n+1} \coloneqq  p^{n+1}_k.
		\end{cases}
	\end{equation}
	In particular, when $k=1$, we obtain the two following algorithms
	\begin{equation}
		\label{alg:1-shot}
		\mbox{\textbf{one-step one-shot:}}\quad\begin{cases}
			\sigma^{n+1}=\sigma^n-\tau M^*p^n-\tau\alpha\sigma^n, \\
			u^{n+1} = Bu^n+M\sigma^{n+1}+F,\\
			p^{n+1} = B^*p^n+H^*(Hu^n-g)
		\end{cases}
	\end{equation}
	and 
	\begin{equation}
		\label{alg:seim-1-shot}
		\text{
				\textbf{semi-implicit one-step one-shot:}
		}\quad
		\begin{cases}
			\sigma^{n+1}=\sigma^n-\tau M^*p^n-\tau\alpha\sigma^{n+1},\\
			u^{n+1} = Bu^n+M\sigma^{n+1}+F,\\
			p^{n+1} = B^*p^n+H^*(Hu^n-g).
		\end{cases}
	\end{equation}
	Note that when $k\rightarrow\infty$, the $k$-step one-shot method \eqref{alg:k-shot} formally converges to the usual gradient descent \eqref{alg:usualgd}, while the semi-implicit $k$-step one-shot method \eqref{alg:seim-k-shot} formally converges to the semi-implicit gradient descent \eqref{alg:seimgd}. Since the analysis of the two schemes \eqref{alg:k-shot} and \eqref{alg:seim-k-shot} can be done following similar arguments, we choose to concentrate on only one of them, namely the semi-implicit scheme. We refer to \cite{bon22} for the analysis 
	%of both schemes 
	in the case $\alpha=0$ (for which the two schemes coincide).
	
	We first analyze the one-step one-shot method ($k=1$) in Section~\ref{sec:seim-1-shot} and then the multi-step one-shot method ($k\ge 1$) in Section~\ref{sec:seim-k-shot}. 
	%The convergence for multi-step one-shot can be deduced using the same arguments.
	
	%	\begin{remark}
		%		When $B=0$,
		%		the usual gradient descent \eqref{alg:usualgd} becomes
		%		$$\begin{cases}
			%			\sigma^{n+1}=\sigma^n-\tau M^*p^n-\tau\alpha\sigma^n, \\
			%			u^n=M\sigma^n+F,\\
			%			p^n=H^*(Hu^n-f),
			%		\end{cases}$$
		%		while $k$-step one-shot \eqref{alg:k-shot} becomes 
		%			$$\begin{cases}
			%			\sigma^{n+1}=\sigma^n-\tau M^*p^n-\tau\alpha\sigma^n, \\
			%			u^{n+1}=M\sigma^{n+1}+F,\\
			%			p^{n+1}=H^*(Hu^n-f)
			%		\end{cases}
		%		\mbox{ if }k=1$$
		%		and 
		%		$$\begin{cases}
			%			\sigma^{n+1}=\sigma^n-\tau M^*p^n-\tau\alpha\sigma^n, \\
			%			u^{n+1}=M\sigma^{n+1}+F,\\
			%			p^{n+1}=H^*(Hu^{n+1}-f).
			%		\end{cases}
		%		\mbox{ if }k\ge 2.$$
		%	\end{remark}
	
	\section{Convergence of the one-step one-shot method ($k=1$)}
	\label{sec:seim-1-shot}
	
	\subsection{Block iteration matrix and eigenvalue equation}
	\label{subsec:seim-1-shot:eigen-eq}
	
	To analyze the convergence of the semi-implicit one-step one-shot method \eqref{alg:seim-1-shot}, we first express \\
	$(\sigma^{n+1},u^{n+1},p^{n+1})$ in terms of $(\sigma^n,u^n,p^n)$, by inserting the expression for $\sigma^{n+1}$ into the iteration for $u^{n+1}$ in \eqref{alg:seim-1-shot}, so that system \eqref{alg:seim-1-shot} is rewritten as
	\begin{equation}
		\label{seim-1-shot:rewrite}
		\begin{cases}
			\sigma^{n+1}=\frac{1}{1+\tau\alpha}\sigma^n-\frac{\tau }{1+\tau\alpha}M^*p^n,\\
			u^{n+1}=Bu^n+\frac{1}{1+\tau\alpha}M\sigma^n-\frac{\tau }{1+\tau\alpha}MM^*p^n+F,\\
			p^{n+1}=B^*p^n+H^*Hu^n-H^*g.\\
		\end{cases}
	\end{equation}
	Now, we consider the errors $(\sigma^n-\sigma^\mathrm{ex}_\alpha,u^n-u(\sigma^\mathrm{ex}_\alpha),p^n-p(\sigma^\mathrm{ex}_\alpha))$ with respect to the regularized solution at the $n$-th iteration, and, by abuse of notation, we designate them by $(\sigma^n,u^n,p^n)$. We obtain that these errors satisfy 
	\begin{equation}
		\label{seim-1-shot:sys-err}
		\begin{cases}
			\sigma^{n+1}=\frac{1}{1+\tau\alpha}\sigma^n-\frac{\tau}{1+\tau\alpha}M^*p^n,\\
			u^{n+1}=Bu^n+\frac{1}{1+\tau\alpha}M\sigma^n-\frac{\tau}{1+\tau\alpha}MM^*p^n,\\
			p^{n+1}=B^*p^n+H^*Hu^n,\\
		\end{cases}
	\end{equation}
	%for algorithm \eqref{seim-1-shot:rewrite}, 
	\noindent or equivalently, by putting in evidence the block iteration matrix
	\begin{equation}
		\label{seim-1-shot:itermat}
		\begin{bmatrix}
			p^{n+1}\\ u^{n+1}\\ \sigma^{n+1}
		\end{bmatrix}=\begin{bmatrix}
			B^* & H^*H & 0\\
			-\frac{\tau}{1+\tau\alpha} MM^* & B & \frac{1}{1+\tau\alpha}M \\
			-\frac{\tau}{1+\tau\alpha}M^* & 0 & \frac{1}{1+\tau\alpha}I
		\end{bmatrix}
		\begin{bmatrix}
			p^{n}\\ u^{n}\\ \sigma^{n}
		\end{bmatrix}.
	\end{equation}
	Now recall that a fixed point iteration converges if and only if the spectral radius of its iteration matrix is less than $1$. Therefore in the following proposition we establish an eigenvalue equation for the iteration matrix of the semi-implicit one-step one-shot method. 
	
	\begin{proposition}
		%[Eigenvalue equation for the semi-implicit $1$-step one-shot method]
		\label{prop:seim-1-shot:eigen-eq} 
		Assume that $\lambda\in\C, |\lambda|\ge 1$ is an eigenvalue of the iteration matrix in \eqref{seim-1-shot:itermat}. If $\lambda\in\C$, $\lambda\notin\mathrm{Spec}(B)$ then $\exists \, y\in\C^{n_\sigma}, \norm{y}=1$ such that:
		\begin{equation}\label{eq:seim-1-shot:ori-eq-eigen}
			(1+\tau\alpha)\lambda-1+\tau\lambda\scalar{M^*(\lambda I-B^*)^{-1}H^*H(\lambda I-B)^{-1}My,y}=0.
		\end{equation}
		In particular, $\lambda=1$ is not an eigenvalue of the iteration matrix.
	\end{proposition}
	%	\begin{remark}
		%		Since $\rho(B)$ is less than $1$, so is $\rho(B^*)$.
		%	\end{remark}		
	\begin{proof}
		Since $\lambda\in\C$ is an eigenvalue of the iteration matrix, there exists a non-zero vector $(\tilde{p},\tilde{u},y)\in\C^{n_u+n_u+n_\sigma}$ such that 
			\begin{equation*}
				\begin{cases}
					\lambda y = \frac{1}{1+\tau\alpha}y-\frac{\tau}{1+\tau\alpha}M^*\tilde{p},\\
					\lambda\tilde{u}= B\tilde{u}+\frac{1}{1+\tau\alpha}My-\frac{\tau}{1+\tau\alpha} MM^*\tilde{p},\\
					\lambda\tilde{p}=B^*\tilde{p}+H^*H\tilde{u}.
				\end{cases}
			\end{equation*}
			By inserting the first equation into the second equation, we simplify this system of equations as
			\begin{equation}
				\label{eq:seim-1-shot:eigenvec}
				\begin{cases}
					\lambda y = \frac{1}{1+\tau\alpha}y-\frac{\tau}{1+\tau\alpha}M^*\tilde{p},\\
					\lambda\tilde{u}= B\tilde{u}+\lambda My,\\
					\lambda\tilde{p}=B^*\tilde{p}+H^*H\tilde{u}.
				\end{cases}
			\end{equation}
			The second equation in \eqref{eq:seim-1-shot:eigenvec} gives us directly $\tilde{u}$ in terms of $y$:
			\begin{equation}
				\label{eq:u(y)_1-shot}
				\tilde{u}=\lambda(\lambda I-B)^{-1}My.
			\end{equation}
			The third equation in \eqref{eq:seim-1-shot:eigenvec} at first gives us $\tilde{p}$ in terms of $\tilde{u}$:
			$$\lambda\tilde{p}=(\lambda I-B^*)^{-1}H^*H\tilde{u},$$
			then by combing with equation \eqref{eq:u(y)_1-shot}, we obtain $\tilde{p}$ in terms of $y$:
			\begin{equation}
				\label{eq:p(y)_1-shot}
				\tilde{p}=\lambda(\lambda I-B^*)^{-1}H^*H(\lambda I-B)^{-1}My.
			\end{equation}
			We also see that $y\neq 0$; indeed if $y=0$ then equations \eqref{eq:u(y)_1-shot} and \eqref{eq:p(y)_1-shot} yield $u=0$ and $p=0$, that is a contradiction. Next, inserting the expression of $\tilde{p}$ in equation \eqref{eq:p(y)_1-shot} back into the first equation in \eqref{eq:seim-1-shot:eigenvec}, we get
			$$\lambda y=\frac{1}{1+\tau\alpha}y-\frac{\tau}{1+\tau\alpha}\lambda M^*(\lambda I-B^*)^{-1}H^*H(\lambda I-B)^{-1}My,$$
			which leads to
			\begin{equation}\label{eq:seim-1-shot:pre-eq-eigen}
				\left[(1+\alpha\tau)\lambda-1\right]y+\tau\lambda M^*(\lambda I-B^*)^{-1}H^*H(\lambda I-B)^{-1}My=0.
			\end{equation}
			Finally, by taking scalar product of \eqref{eq:seim-1-shot:pre-eq-eigen} with $y$, then dividing by $\norm{y}^2$, we obtain \eqref{eq:seim-1-shot:ori-eq-eigen}.
		
		\color{black}{\noindent Now assume that $\lambda=1$ is an eigenvalue of the iteration matrix, then \eqref{eq:seim-1-shot:ori-eq-eigen} yields
			$$\alpha+\norm{H(I-B)^{-1}My}^2=0,$$
			which cannot be true due to the injectivity of $H(I-B)^{-1}M$.}
	\end{proof}
	
	In the following sections we will show that, for sufficiently small $\tau$, equation \eqref{eq:seim-1-shot:ori-eq-eigen}
	%admits no solution 
	cannot hold if $|\lambda|\ge 1$, thus algorithm \eqref{alg:seim-1-shot} converges. 
	%When $\lambda\neq 0$, 
	It is convenient to rewrite \eqref{eq:seim-1-shot:ori-eq-eigen} as
	\begin{equation}
		\label{eq:seim-1-shot:eq-eigen}
		(1+\tau\alpha)\lambda^2-\lambda+\tau\scalar{M^*\left(I-B^*/\lambda\right)^{-1}H^*H\left(I-B/\lambda\right)^{-1}My,y}=0.
	\end{equation}
	For the analysis we use auxiliary technical results proved in Appendix~\ref{app:lems}.
	
	\subsection{Location of the eigenvalues in the complex plane}
	\label{subsec:seim-1-shot:loca-eigen}
	
	We now turn our attention to the eigenvalues $\lambda$ for which \eqref{eq:seim-1-shot:eq-eigen} holds. We would like to derive conditions on the descent step $\tau$ such that all the eigenvalues lie inside the unit circle which would ensure the convergence for the scheme \eqref{alg:seim-1-shot}. We start with the simple case of real eigenvalues. 
	
	%	find conditions on the descent step $\tau$ such that the real eigenvalues stay inside the unit disk. Recall that we have already proved that $\lambda=1$ is not an eigenvalue for $1$-step one-shot method.
	
	\begin{proposition}[Real eigenvalues]\label{prop:seim-1-shot:tau:real-lam}
		Equation \eqref{eq:seim-1-shot:eq-eigen} admits no solution $\lambda\in\R, \lambda\neq 1, |\lambda|\ge 1$ for all $\tau>0$.
	\end{proposition}
	\begin{proof}
		When $\lambda\in\R\backslash\{0\}$ equation \eqref{eq:seim-1-shot:eq-eigen} becomes $$(1+\tau\alpha)\lambda^2-\lambda+\tau\norm{H(I-B/\lambda)^{-1}My}^2=0.$$
		If $\lambda\in\R,\lambda\neq 1,|\lambda|\ge 1$ then $(1+\tau\alpha)\lambda^2-\lambda\ge \lambda^2-\lambda>0$, thus the left-hand side of the above equation is positive for any $\tau>0$.
	\end{proof}
	
	For the general case of complex eigenvalues, the study is much more complicated and technical. First, we study separately the very particular and simple case where $B=0$. 
	
	\begin{proposition}
		\label{prop:seim-1-shot:tau:B=0}
		When $B=0$, the eigenvalue equation \eqref{eq:seim-1-shot:eq-eigen} cannot hold for $\lambda\in\C, |\lambda|\ge 1$ if $\tau>0$ and
		$$
		(\norm{H}^2\norm{M}^2-\alpha)\tau<1.
		$$
	\end{proposition}
	\begin{proof}
		When $B=0$, $HM$ is injective by \eqref{hypo}, and equation \eqref{eq:seim-1-shot:eq-eigen} becomes
		$(1+\tau\alpha)\lambda^2-\lambda+\tau\norm{HMy}^2=0$. One can prove the proposition by explicitly computing the roots of the second order polynomial with respect to $\lambda$. One can also apply the following elementary Lemma \ref{lem:marden-ord2} below, which can be deduced from Marden's work \cite{marden66} or Appendix C of \cite{bon22}. Indeed, the previous polynomial can be written as $P(\lambda)=a_0+a_1\lambda+\lambda^2$ with $a_0=\frac{\tau\norm{HMy}^2}{1+\tau\alpha}$ and $a_1=\frac{-1}{1+\tau\alpha}$. We have
			$a_0-a_1+1=\frac{\tau(\norm{HMy}^2+\alpha)+2}{1+\tau\alpha}>0$
			and $a_0+a_1+1=\frac{\tau(\norm{HMy}^2+\alpha)}{1+\tau\alpha}>0$ since $\norm{HMy}^2>0$ thanks to the injectivity of $HM$. Therefore, according to Lemma \ref{lem:marden-ord2}, the roots of $P$ stay strictly inside the unit circle of the complex plane if and only if $|a_0|=\frac{\tau\norm{HMy}^2}{1+\tau\alpha}<1$, which is equivalent to 
			$(\norm{HMy}^2-\alpha)\tau<1$. Since $\norm{y}=1$, the later inequality is verified if $
			(\norm{H}^2\norm{M}^2-\alpha)\tau<1
			$, which is the statement of the proposition.
	\end{proof}
	% $\tau(\norm{HMy}^2-\alpha)<1$
	% $\tau(\norm{HMy}^2-\alpha)\le\tau(\norm{H}^2\norm{M}^2-\alpha)$
	
	\begin{lemma}\label{lem:marden-ord2}
		Let $a_0,a_1\in\R$, then all roots of $\mathcal{P}(z)=a_0+a_1z+z^2$ stay strictly inside the unit circle of the complex plane if and only if
		$$
		|a_0|<1\quad\text{and}\quad
		(a_0-a_1+1)(a_0+a_1+1)>0.
		$$
	\end{lemma}
	%\noindent The proof of this lemma can be deduced from Appendix C of \cite{bon22}.
	
	Now we consider the more complicated case where $0\neq \rho(B)<1$. The following proposition summarizes the results we obtained. 
	
	%In order to state the result, let us introduce the following functions: 
	%$$\frac{\norm{H}^2\norm{M}^2}{(1-\norm{B})^4}4\norm{B}^2+(\sqrt{2}-1)\alpha, $$
	%
	%$$\frac{\norm{H}^2\norm{M}^2}{(1-\norm{B})^4}\cfrac{1}{2\sin\frac{\theta_0}{2}}(1+\norm{B})^2(1-\norm{B})^2+\left(\sqrt{2}+\frac{1}{2\sin\frac{\theta_0}{2}}-1\right)\alpha,
	%$$
	%
	%$$
	%\frac{\norm{H}^2\norm{M}^2}{(1-\norm{B})^4}\cfrac{2c}{\delta_0}\norm{B}^2+\left(\frac{\sqrt{c}}{\delta_0}-1\right)\alpha$$
	%where $c=c(\theta_0,\delta_0)\coloneqq \left(1+2\delta_0\sin\frac{3\theta_0}{2}+\delta_0^2\right)/\cos^2\frac{3\theta_0}{2}$.
	
	%We now look for conditions on the descent step $\tau$ such that also the complex eigenvalues stay inside the unit disk. We first deal with the shifted $1$-step one-shot method. 
	
	\begin{proposition}[Complex eigenvalues]
		\label{prop:seim-1-shot:tau:complex-lam}
		If $0\neq \rho(B)<1$, there exists $\tau>0$ sufficiently small such that equation \eqref{eq:seim-1-shot:eq-eigen} admits no solution $\lambda\in\C\backslash\R, |\lambda|\ge 1$. 
		In particular, if $\norm{B}<1$, given any $\delta_0>0$ and $0<\theta_0\le\frac{\pi}{4}$, one can choose
		\[
		\tau <\min_{1\le i\le 3}\left(
		\frac{\norm{H}^2\norm{M}^2}{(1-\norm{B})^4}\varphi_i(\norm{B})+C_i\alpha\right)^{-1},
		\]
		where the functions $\varphi_i$, $1\le i\le 3$, are defined by
		\begin{equation*}
			\varphi_1(b)\coloneqq4b^2,
			\quad\varphi_2(b)\coloneqq\cfrac{1}{2\sin\frac{\theta_0}{2}}(1+b)^2(1-b)^2
			\quad\text{and}\quad
			\varphi_3(b)\coloneqq\cfrac{2c}{\delta_0}b^2,\\
		\end{equation*}
		and the pure constants $C_i$, $1\le i\le 3$, are defined by
		\begin{equation*}
			C_1\coloneqq\sqrt{2}-1,\quad 
			C_2\coloneqq \sqrt{2}+\frac{1}{2\sin\frac{\theta_0}{2}}-1
			\quad\text{and}\quad
			C_3\coloneqq\frac{\sqrt{c}}{\delta_0}-1
		\end{equation*}
		with
		\begin{equation*}
			c\coloneqq\frac{1+2\delta_0\sin\frac{3\theta_0}{2}+\delta_0^2}{\cos^2\frac{3\theta_0}{2}}.
		\end{equation*}
		%	\begin{align*}
			%		\varphi_1(b)&\coloneqq4b^2,
			%		\quad\varphi_2(b)\coloneqq\cfrac{1}{2\sin\frac{\theta_0}{2}}(1+b)^2(1-b)^2,
			%		\quad\varphi_3(b)\coloneqq\cfrac{2c}{\delta_0}b^2,\\
			%		C_1&\coloneqq\sqrt{2}-1,\quad 
			%		C_2\coloneqq \sqrt{2}+\frac{1}{2\sin\frac{\theta_0}{2}}-1,\quad
			%		C_3\coloneqq\frac{\sqrt{c}}{\delta_0}-1\quad\mbox{and}
			%		\quad c\coloneqq\frac{1+2\delta_0\sin\frac{3\theta_0}{2}+\delta_0^2}{\cos^2\frac{3\theta_0}{2}}.
			%	\end{align*}
	\end{proposition}
	
	\begin{proof} %We shall consider first the case $\norm{B}<1$. We indicate at the end of the proof the changes how to treat the case $0\neq \rho(B)<1$.
		
		\textbf{Step 1. Rewrite equation \eqref{eq:seim-1-shot:eq-eigen} by separating real and imaginary parts.}
		
		Let $\lambda=R(\cos\theta+\ic\sin\theta)$ in polar form where $R=|\lambda| \ge 1$ and $\theta\in(-\pi,\pi)$, $\theta\neq 0$. Write ${1}/{\lambda}=r(\cos\phi+\ic\sin\phi)$ in polar form where $r={1}/{|\lambda|}={1}/{R}\le 1$ and $\phi=-\theta\in(-\pi,\pi)$. By Lemma \ref{lem:decomPQ}, we have
		$$\left(I-\frac{B}{\lambda}\right)^{-1}=P(\lambda)+\ic Q(\lambda),\quad \left(I-\frac{B^*}{\lambda}\right)^{-1}=P(\lambda)^*+\ic Q(\lambda)^*$$
		where $P(\lambda)$ and $Q(\lambda)$ are $\C^{n_u\times n_u}$ matrices that satisfy the following bounds in the case $\norm{B}<1$ for all $|\lambda|\ge 1$: %by omitting the dependence on $\lambda$, 
		\begin{gather}
			\norm{P(\lambda)}\le p \coloneqq 
			\cfrac{1}{1-\norm{B}}, \label{eq:pB} \\ 
			\norm{Q(\lambda)}\le q_1\coloneqq
			\cfrac{\norm{B}}{1-\norm{B}}
			\quad\mbox{and}\quad
			%\label{eq:q2B}
			\norm{Q(\lambda)}\le |\sin\theta|q_2 \quad\mbox{with}\quad q_2\coloneqq
			\cfrac{\norm{B}}{(1-\norm{B})^2}. \label{eq:qB}
		\end{gather}
		These bounds still hold in the case $0\neq\rho(B)<1$ with  
		\begin{equation}
			\label{eq:pqB:rho(B)<1}
			p=(1+\norm{B})s(B)^2, \quad q_1=\norm{B}s(B)^2 \quad\mbox{and}\quad q_2=\norm{B}s(B)^2.
		\end{equation}

		%	\begin{equation}\label{eq:pB}
			%		\norm{P}\le p \coloneqq \left\{\begin{array}{cl}
				%			(1+\norm{B})s(B)^2 &\text{ for }0\neq \rho(B)<1,\\
				%			\cfrac{1}{1-\norm{B}} &\text{ when }\norm{B}<1;
				%		\end{array}\right.
			%	\end{equation}
		%	\begin{equation}\label{eq:q1B}
			%		\norm{Q}\le q_1 \coloneqq \left\{\begin{array}{cl}
				%			\norm{B}s(B)^2 &\text{ for general }B\neq 0,\\
				%			\cfrac{\norm{B}}{1-\norm{B}} &\text{ when }0<\norm{B}<1;
				%		\end{array}\right.
			%	\end{equation}
		%	\begin{equation}\label{eq:q2B}
			%		\norm{Q}\le |\sin\theta|q_2,\quad q_2 \coloneqq \left\{\begin{array}{cl}
				%			\norm{B}s(B)^2 &\text{ for general }B\neq 0,\\
				%			\cfrac{\norm{B}}{(1-\norm{B})^2} &\text{ when }0<\norm{B}<1.
				%		\end{array}\right.
			%	\end{equation}

		% We have the following table:
		% \begin{center}
			% 	\begin{tabular}{|c|c|c|c|c|}
				% 		\hline
				% 		& $p$ & $q_1$ & $q_2$ & $p+q_1$ \\ 
				% 		\hline
				% 		General $B\neq 0$ & $(1+\norm{B})s(B)^2$
				% 		& $\norm{B}s(B)^2$ & $\norm{B}s(B)^2$
				% 		& $(1+2\norm{B})s(B)^2$\\
				% 		\hline
				% 		$0<\norm{B}<1$ & $\cfrac{1}{1-\norm{B}}$
				% 		& $\cfrac{\norm{B}}{1-\norm{B}}$ & $\cfrac{\norm{B}}{(1-\norm{B})^2}$
				% 		& $\cfrac{1+\norm{B}}{1-\norm{B}}$\\
				% 		\hline
				% 	\end{tabular}
			% \end{center}
		% \noindent Note that $p$, $p+q_1$ and $q_2$ are  positive. 
		\noindent To simplify the notation, we will not explicitly write the dependence of $P$ and $Q$ on $\lambda$. Now we rewrite \eqref{eq:seim-1-shot:eq-eigen} as
		\begin{equation}
			\label{eq:ready-split:k=1}
			(1+\tau\alpha)(\lambda^2-\lambda)+\tau\alpha(\lambda-1)+\tau\alpha+\tau G(P^*+\ic Q^*,P+\ic Q)=0
		\end{equation}
		where 
		$$G(X,Y)\coloneqq\scalar{M^*XH^*HYMy,y}\in\C, \quad X,Y\in\C^{n_u\times n_u}.$$
		Notice that $G$ is a bilinear form and $G(X,Y)=G(Y^*,X^*)^*$ so that $G(X,Y)+G(X^*,Y^*)$ is real.
		%\begin{itemize}
		%		\item $\forall X,Y_1,Y_2\in\C^{n_u\times n_u}, \forall z_1,z_2\in\C$: \quad
		%		$G(X, z_1Y_1+z_2Y_2)=z_1G(X,Y_1)+z_2G(X,Y_2).$
		%		\item $\forall X_1,X_2,Y\in\C^{n_u\times n_u}, \forall z_1,z_2\in\C$: \quad
		%		$G(z_1X_1+z_2X_2,Y)=z_1G(X_1,Y)+z_2G(X_2,Y).$ 
		%		\item $\forall X\in\C^{n_u\times n_u}$: \quad
		%		$0\le G(X^*,X)=\norm{HXMy}^2\le (\norm{H}\norm{M}\norm{X})^2 \norm{y}^2.$
		%		\item $\forall X,Y\in\C^{n_u\times n_u}$: \quad $G(X,Y)+G(Y^*,X^*)\in\R$, indeed 
		%		$$\begin{array}{ll}
			%			G(X,Y)&=\scalar{M^*XH^*HYMy,y}=\scalar{y,M^*Y^*H^*HX^*My}\\
			%			&=\scalar{M^*Y^*H^*HX^*My,y}^*=G(Y^*,X^*)^*.
			%	\end{array}$$
		%\end{itemize}
		With these properties of $G$, we expand \eqref{eq:ready-split:k=1} and take its real and imaginary parts, which yields
		\begin{equation}\label{eq:realpart:k=1}
			(1+\tau\alpha)\Re(\lambda^2-\lambda)+\tau\alpha\Re(\lambda-1)+\tau\alpha+\tau [G(P^*,P)- G(Q^*,Q)]=0,
		\end{equation}
		and
		\begin{equation}\label{eq:impart:k=1}
			(1+\tau\alpha)\Im(\lambda^2-\lambda)+\tau\alpha\Im(\lambda-1)+\tau [G(P^*,Q)+ G(Q^*,P)]=0.
		\end{equation}
		
		\noindent\textbf{Step 2. Use a suitable combination of equations \eqref{eq:realpart:k=1} and \eqref{eq:impart:k=1}.}
		
		%choose $\tau$ so that we obtain a new equation with a left-hand side which is strictly positive/negative.}
	
	Let $\gamma\in\R$.
	%defined by cases as in Lemma~\ref{lem:gamma123}. 
	Multiplying equation \eqref{eq:impart:k=1} with $\gamma$ then summing it with equation \eqref{eq:realpart:k=1}, we obtain:
	\begin{multline*}
		(1+\tau\alpha)[\Re(\lambda^2-\lambda)+\gamma\Im(\lambda^2-\lambda)]+\tau\alpha[\Re(\lambda-1)+\gamma\Im(\lambda-1)]\\
		+\tau\alpha+\tau[G(P^*,P)-G(Q^*,Q)+\gamma G(P^*,Q)+\gamma G(Q^*,P)]=0,
	\end{multline*}
	or equivalently,
	\begin{multline}\label{eq:im+gamma*re:k=1}
		(1+\tau\alpha)[\Re(\lambda^2-\lambda)+\gamma\Im(\lambda^2-\lambda)]+\tau\alpha[\Re(\lambda-1)+\gamma\Im(\lambda-1)]\\
		+\tau\alpha+\tau G(P^*+\gamma Q^*,P+\gamma Q)-(1+\gamma^2)\tau G(Q^*,Q)=0.
	\end{multline}
	
	\begin{figure}[htb]
		\centering
		\includegraphics[width=0.35\textwidth]{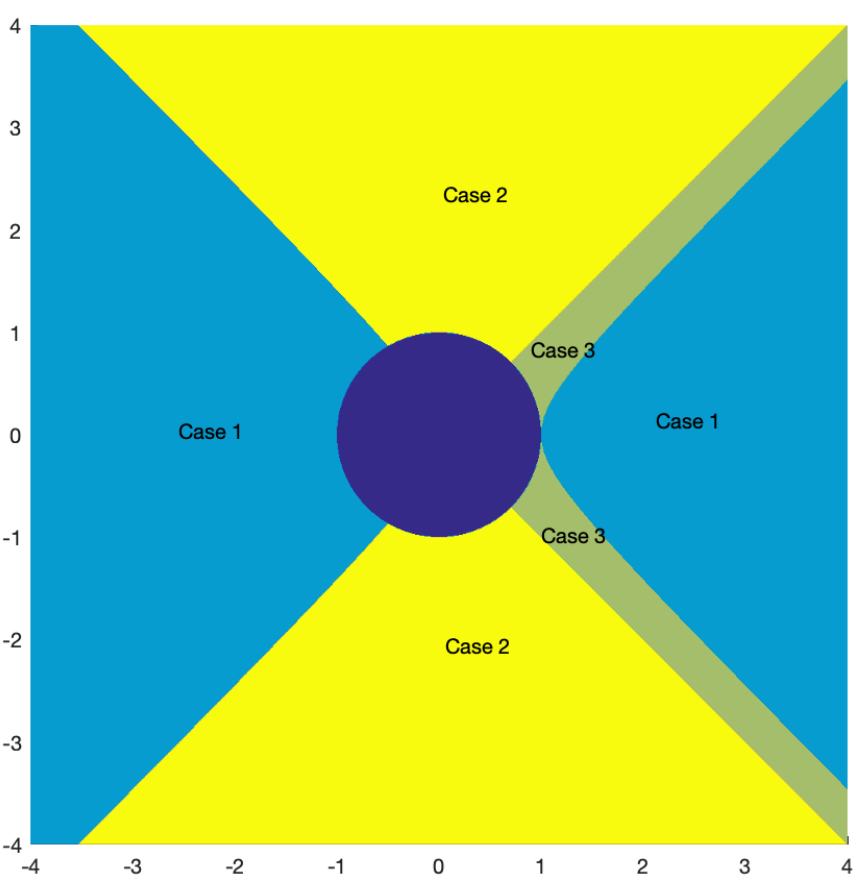}
		\caption{Illustration of the regions in the complex plane associated with the cases 1, 2 and 3 indicated in the proof of Proposition \ref{prop:seim-1-shot:tau:complex-lam} for $\theta_0=\frac{\pi}{4}$. The center circle is the unit circle.}
		\label{fig:regions}
	\end{figure}
	Now we consider four cases for $\lambda$ as in Lemma~\ref{lem:gamma123} (see Figure \ref{fig:regions}):
	\begin{itemize}
		\item\textit{Case 1.} $\Re(\lambda^2-\lambda)\ge 0$;
		\item\textit{Case 2.} $\Re(\lambda^2-\lambda)<0$ and $\theta\in [\theta_0,\pi-\theta_0]\cup[-\pi+\theta_0,-\theta_0]$ for fixed $0<\theta_0\le\frac{\pi}{4}$;
		\item\textit{Case 3.} $\Re(\lambda^2-\lambda)<0$ and $\theta \in (-\theta_0,\theta_0)$ for fixed $0<\theta_0\le \frac{\pi}{4}$;
		\item\textit{Case 4.} $\Re(\lambda^2-\lambda)<0$ and $\theta \in (\pi-\theta_0,\pi)\cup(-\pi,-\pi+\theta_0)$ for fixed $0<\theta_0\le \frac{\pi}{4}$. 
	\end{itemize}
	%	Notice that here just three cases 1, 2 and 3 of $\lambda$ need to be considered. 	 
	Cases 1, 2 and 3 are respectively treated in Lemmas \ref{lem:seim-1-shot:case1}, \ref{lem:seim-1-shot:case2} and \ref{lem:seim-1-shot:case3} below.
	Notice that case 4 corresponds to an empty set, according to Lemma \ref{lem:gamma123} (iv). The statement of the proposition easily follows from the combination of Lemmas \ref{lem:seim-1-shot:case1}, \ref{lem:seim-1-shot:case2} and \ref{lem:seim-1-shot:case3}.
	In particular, the parameter $\theta_0$ is mainly used for case 2 to obtain a lower bound for $|\lambda-1|$. This is why we require $\theta_0>0$. We need $\theta_0\le\frac{\pi}{4}$ to design a lower bound for a suitable combination of $\Re(\lambda^2-\lambda)$ and $\Im(\lambda^2-\lambda)$ in case 3.
\end{proof}
%because the analogue of the fourth one is excluded by Lemma \ref{lem:gamma123} (iv). 

\noindent In the lemmas below we shall make use of the following obvious property where $\norm{y}=1$:
\begin{equation}
	\label{ppty:G}
	0\le G(X^*,X)=\norm{HXMy}^2\le (\norm{H}\norm{M}\norm{X})^2,\quad\forall X\in\C^{n_u\times n_u}.
\end{equation}
%note: \norm{y}=1

\begin{lemma}[Case 1]
	\label{lem:seim-1-shot:case1} Let $0\neq\rho(B)<1$ and let $|\lambda|\ge 1$ with $\Re(\lambda^2-\lambda)\ge 0$. Then equation \eqref{eq:seim-1-shot:eq-eigen} cannot hold if one chooses
	$$\tau<\left(\norm{H}^2\norm{M}^2s(B)^4 \cdot4\norm{B}^2+(\sqrt{2}-1)\alpha\right)^{-1}.$$
	%$$\tau<\frac{1}{4\norm{H}^2\norm{M}^2\norm{B}^2s(B)^4+(\sqrt{2}-1)\alpha}.$$
	Moreover, if $\norm{B}<1$, the result is also true if 
	$$\tau<\left(\frac{\norm{H}^2\norm{M}^2}{(1-\norm{B})^4}\cdot4\norm{B}^2+(\sqrt{2}-1)\alpha\right)^{-1}.$$
	%	$$\tau<\frac{1}{4\norm{H}^2\norm{M}^2\norm{B}^2(1-\norm{B})^{-4}+(\sqrt{2}-1)\alpha}.$$
\end{lemma}
\begin{proof}
	Define $$\gamma_1=\gamma_1(\lambda)\coloneqq\left\{\begin{array}{cc}
		1 &\text{if } \Im (\lambda^2-\lambda)\ge 0,\\
		-1 &\text{if } \Im (\lambda^2-\lambda)<0
	\end{array}\right.$$
	as in Lemma \ref{lem:gamma123}. Writing \eqref{eq:im+gamma*re:k=1} for $\gamma=\gamma_1$ and using $\gamma_1^2=1$ we obtain
	\begin{multline}\label{eq:seim-1-shot:case1}
		(1+\tau\alpha)[\Re(\lambda^2-\lambda)+\gamma_1\Im(\lambda^2-\lambda)]+\tau\alpha[\Re(\lambda-1)+\gamma_1\Im(\lambda-1)]\\
		+\tau\alpha+\tau G(P^*+\gamma_1 Q^*,P+\gamma_1 Q)-2\tau G(Q^*,Q)=0.
	\end{multline}
	Since $G(P^*+\gamma_1 Q^*,P+\gamma_1 Q)\ge 0$ by \eqref{ppty:G} and $\tau\alpha\ge 0$, then the left-hand side of \eqref{eq:seim-1-shot:case1} is positive if $\tau$ satisfies
	$$(1+\tau\alpha)[\Re(\lambda^2-\lambda)+\gamma_1\Im(\lambda^2-\lambda)]-\tau\alpha|\Re(\lambda-1)+\gamma_1\Im(\lambda-1)|
	-2\tau G(Q^*,Q)>0,$$
	or equivalently, 
	\begin{equation}
		\label{ineq:seim-1-shot:case1}
		1+\tau\alpha-\tau\alpha\frac{|\Re(\lambda-1)+\gamma_1\Im(\lambda-1)|}{\Re(\lambda^2-\lambda)+\gamma_1\Im(\lambda^2-\lambda)}
		-2\tau \frac{G(Q^*,Q)}{\Re(\lambda^2-\lambda)+\gamma_1\Im(\lambda^2-\lambda)}>0.
	\end{equation}
	Notice that the choice of $\gamma_1$ ensures $\Re(\lambda^2-\lambda)+\gamma_1\Im(\lambda^2-\lambda)>0$. By Lemma \ref{lem:gamma123}  we have
	$$\frac{|\Re(\lambda-1)+\gamma_1\Im(\lambda-1)|}{\Re(\lambda^2-\lambda)+\gamma_1\Im(\lambda^2-\lambda)}
	\le\frac{\sqrt{1+\gamma_1^2}|\lambda-1|}{|\lambda(\lambda-1)|}
	=\frac{\sqrt{2}}{|\lambda|}\le\sqrt{2}.$$
	Using again Lemma \ref{lem:gamma123}  and  \eqref{ppty:G}, we have
	\begin{multline*}
		\frac{G(Q^*,Q)}{\Re(\lambda^2-\lambda)+\gamma_1\Im(\lambda^2-\lambda)}
		\le\frac{(\norm{H}\norm{M}|\sin\theta|q_2)^2}{2|\sin(\theta/2)|}
		=2\left|\sin\frac{\theta}{2}\right|\cos^2\frac{\theta}{2}\norm{H}^2\norm{M}^2q_2^2\\
		\le2\norm{H}^2\norm{M}^2q_2^2.
	\end{multline*}
	Inserting the two previous inequalities in \eqref{ineq:seim-1-shot:case1} gives the desired results using definitions \eqref{eq:qB} and \eqref{eq:pqB:rho(B)<1} of $q_2$.
\end{proof}

%By the properties of $G$ we have
%$$G(P^*+\gamma_1 Q^*,P+\gamma_1 Q) \ge 0$$
%and
%$$G(Q^*,Q)\le(\norm{H}\norm{M}\norm{Q})^2\norm{y}^2\le (\norm{H}\norm{M}|\sin\theta|q_2)^2\norm{y}^2,$$
%therefore the left-hand side of \eqref{eq:seim-1-shot:case1} will be  positive if $\tau$ satisfies
%$$\tau<\frac{\Re(\lambda^2-\lambda)+\gamma_1\Im(\lambda^2-\lambda)}{2\left(\norm{H}\norm{M}|\sin\theta|q_2\right)^2}.$$
%Since $\Re(\lambda^2-\lambda)+\gamma_1\Im(\lambda^2-\lambda)\ge 2|\sin(\theta/2)|$ by Lemma \ref{lem:gamma123} , it is enough to choose
%$$\tau<\cfrac{1}{4\left|\sin\frac{\theta}{2}\right|\cos^2\frac{\theta}{2}\norm{H}^2\norm{M}^2q_2^2}.$$
%Since $\left|\sin\frac{\theta}{2}\right|\cos^2\frac{\theta}{2}\le 1$, it is sufficient to choose 
%$\tau<\frac{1}{4\norm{H}^2\norm{M}^2q_2^2}$ and we use definition \eqref{eq:q2B} of $q_2$.

\begin{lemma}[Case 2]
	\label{lem:seim-1-shot:case2} Let $0\neq\rho(B)<1$ and let $|\lambda|\ge1$, $\Re(\lambda^2-\lambda)<0$, $\theta\in [\theta_0,\pi-\theta_0]\cup[-\pi+\theta_0,-\theta_0]$ for given $0<\theta_0\le\frac{\pi}{4}$. Then equation \eqref{eq:seim-1-shot:eq-eigen} cannot hold if one chooses
	$$\tau<\left(\norm{H}^2\norm{M}^2s(B)^4\cdot\frac{(1+2\norm{B})^2}{2\sin\frac{\theta_0}{2}}+\left(\sqrt{2}+\frac{1}{2\sin\frac{\theta_0}{2}}-1\right)\alpha\right)^{-1}.$$
	%$$\tau<\cfrac{1}{\frac{1}{2\sin\frac{\theta_0}{2}}\norm{H}^2\norm{M}^2(1+2\norm{B})^2s(B)^4+\left(\sqrt{2}+\frac{1}{2\sin\frac{\theta_0}{2}}-1\right)\alpha}.$$
	Moreover, if $\norm{B}<1$, the result is also true if
	$$\tau<\left(\frac{\norm{H}^2\norm{M}^2}{(1-\norm{B})^2}\cdot\frac{(1+\norm{B})^2}{2\sin\frac{\theta_0}{2}}+\left(\sqrt{2}+\frac{1}{2\sin\frac{\theta_0}{2}}-1\right)\alpha\right)^{-1}.$$
	%$$\tau<\cfrac{1}{\frac{1}{2\sin\frac{\theta_0}{2}}\norm{H}^2\norm{M}^2(1+\norm{B})^2(1-\norm{B})^{-2}+\left(\sqrt{2}+\frac{1}{2\sin\frac{\theta_0}{2}}-1\right)\alpha}.$$
\end{lemma}
\begin{proof}
	Define 
	$$\gamma_2=\gamma_2(\lambda)\coloneqq\left\{\begin{array}{cc}
		-1 &\text{if } \Im (\lambda^2-\lambda)\ge 0,\\
		1 &\text{if } \Im (\lambda^2-\lambda)<0
	\end{array}\right.$$ 
	as in Lemma \ref{lem:gamma123} (ii). Writing \eqref{eq:im+gamma*re:k=1} for $\gamma=\gamma_2$ and using $\gamma_2^2=1$ we obtain
	\begin{multline}\label{eq:seim-1-shot:case2}
		(1+\tau\alpha)[\Re(\lambda^2-\lambda)+\gamma_2\Im(\lambda^2-\lambda)]+\tau\alpha[\Re(\lambda-1)+\gamma_2\Im(\lambda-1)]
		\\
		+\tau\alpha+\tau G(P^*+\gamma_2 Q^*,P+\gamma_2 Q)-2\tau G(Q^*,Q)=0.
	\end{multline}
	Since $G(Q^*,Q)\ge 0$ by \eqref{ppty:G}, the left-hand side of \eqref{eq:seim-1-shot:case2} is negative if $\tau$ satisfies
	\begin{multline*}
		(1+\tau\alpha)[\Re(\lambda^2-\lambda)+\gamma_2\Im(\lambda^2-\lambda)]+\tau\alpha|\Re(\lambda-1)+\gamma_2\Im(\lambda-1)|
		\\
		+\tau\alpha+\tau G(P^*+\gamma_2 Q^*,P+\gamma_2 Q)<0,
	\end{multline*}
	or equivalently, 
	\begin{multline}\label{ineq:seim-1-shot:case2}
		-1-\tau\alpha
		+\tau\alpha\frac{|\Re(\lambda-1)+\gamma_2\Im(\lambda-1)|}{|\Re(\lambda^2-\lambda)+\gamma_2\Im(\lambda^2-\lambda)|}
		+\frac{\tau\alpha}{|\Re(\lambda^2-\lambda)+\gamma_2\Im(\lambda^2-\lambda)|}
		\\
		+\tau\frac{G(P^*+\gamma_2 Q^*,P+\gamma_2 Q)}{|\Re(\lambda^2-\lambda)+\gamma_2\Im(\lambda^2-\lambda)|}<0.
	\end{multline}
	Notice that the choice of $\gamma_2$ ensures $\Re(\lambda^2-\lambda)+\gamma_2\Im(\lambda^2-\lambda)<0$. 
	In the following we derive upper bounds independent of $\lambda$ for the terms appearing with the negative sign in \eqref{ineq:seim-1-shot:case2}.
	By Lemma \ref{lem:gamma123} (ii) we have
	$$\frac{|\Re(\lambda-1)+\gamma_2\Im(\lambda-1)|}{|\Re(\lambda^2-\lambda)+\gamma_2\Im(\lambda^2-\lambda)|}\le\frac{\sqrt{1+\gamma_2^2}|\lambda-1|}{|\lambda(\lambda-1)|}=\frac{\sqrt{2}}{|\lambda|}\le\sqrt{2}$$
	and
	$$\frac{1}{|\Re(\lambda^2-\lambda)+\gamma_2\Im(\lambda^2-\lambda)|}\le\frac{1}{2\sin\frac{\theta_0}{2}}.$$
	Using again Lemma \ref{lem:gamma123} (ii) and  \eqref{ppty:G}, we have
	$$\frac{G(P^*+\gamma_2 Q^*,P+\gamma_2 Q)}{|\Re(\lambda^2-\lambda)+\gamma_2\Im(\lambda^2-\lambda)|}
	\le\frac{\norm{H}^2\norm{M}^2(p+q_1)^2}{2\sin\frac{\theta_0}{2}}.$$
	Inserting these previous inequalities in \eqref{ineq:seim-1-shot:case2} gives the desired results using definitions \eqref{eq:pB}, \eqref{eq:qB} and \eqref{eq:pqB:rho(B)<1} of $p$ and $q_1$.
\end{proof}

%By the properties of $G$
%$$G(Q^*,Q)\ge 0,\quad G(P^*+\gamma_2 Q^*,P+\gamma_2 Q)\le(\norm{H}\norm{M}\norm{P+\gamma_2 Q})^2\norm{y}^2$$
%and the estimate  
%$\norm{P+\gamma_2 Q}\le \norm{P}+|\gamma_2|\norm{Q}=\norm{P}+\norm{Q}\le p+q_1,$
%the left-hand side of \eqref{eq:seim-1-shot:case2} will be strictly negative if $\tau$ satisfies:
%$$\tau<\frac{-\Re(\lambda^2-\lambda)-\gamma_2\Im(\lambda^2-\lambda)}{\left[\norm{H}\norm{M}(p+q_1)\right]^2}.$$
%Thanks to Lemma \ref{lem:gamma123} (ii), it is sufficient to choose
%$$\tau<\cfrac{2\sin\frac{\theta_0}{2}}{\norm{H}^2\norm{M}^2(p+q_1)^2}$$
%and we use definitions \eqref{eq:pB} and \eqref{eq:q1B} of $p$ and $q_1$.

\begin{lemma}[Case 3]
	\label{lem:seim-1-shot:case3} 
	Let $0\neq\rho(B)<1$ and let $|\lambda|\ge 1$, $\Re(\lambda^2-\lambda)<0$, $\theta \in (-\theta_0,\theta_0)$ for given $0<\theta_0\le \frac{\pi}{4}$. For any $\delta_0>0$, equation \eqref{eq:seim-1-shot:eq-eigen} cannot hold if one chooses
	$$\tau<\left(\norm{H}^2\norm{M}^2s(B)^4\cdot\frac{2c}{\delta_0}\norm{B}^2+\left(\frac{\sqrt{c}}{\delta_0}-1\right)\alpha\right)^{-1}$$
	where $c=c(\theta_0,\delta_0)\coloneqq \left(1+2\delta_0\sin\frac{3\theta_0}{2}+\delta_0^2\right)/\cos^2\frac{3\theta_0}{2}$.
	Moreover, if $0<\norm{B}<1$, the result is also true if 
	$$\tau<\left(\frac{\norm{H}^2\norm{M}^2}{(1-\norm{B})^{-4}}\cdot\frac{2c}{\delta_0}\norm{B}^2+\left(\frac{\sqrt{c}}{\delta_0}-1\right)\alpha\right)^{-1}.$$
	%	$$\tau<\frac{1}{\frac{2c}{\delta_0}\norm{H}^2\norm{M}^2\norm{B}^2(1-\norm{B})^{-4}+\left(\frac{\sqrt{c}}{\delta_0}-1\right)\alpha}.$$
\end{lemma}
\begin{proof}
	Define $$\gamma_3=\gamma_3(\mathrm{sign}(\theta))\coloneqq\left\{\begin{array}{cc}
		\left(\delta_0+\sin\frac{3\theta_0}{2}\right)/\cos\frac{3\theta_0}{2} & \text{if }\theta>0,\\
		-\left(\delta_0+\sin\frac{3\theta_0}{2}\right)/\cos\frac{3\theta_0}{2} & \text{if }\theta<0
	\end{array}\right.$$ 
	as in Lemma \ref{lem:gamma123} (iii). 
	Writing \eqref{eq:im+gamma*re:k=1} for $\gamma=\gamma_3$ we obtain
	\begin{multline}\label{eq:seim-1-shot:case3}
		(1+\tau\alpha)[\Re(\lambda^2-\lambda)+\gamma_3\Im(\lambda^2-\lambda)]+\tau\alpha[\Re(\lambda-1)+\gamma_3\Im(\lambda-1)]
		\\
		+\tau\alpha+\tau G(P^*+\gamma_3 Q^*,P+\gamma_3 Q)
		-\tau(1+\gamma_3^2)G(Q^*,Q)=0.
	\end{multline}
	Since $G(P^*+\gamma_3 Q^*,P+\gamma_3 Q)\ge 0$ by \eqref{ppty:G} and $\tau\alpha\ge0$, the left-hand side of \eqref{eq:seim-1-shot:case3} is positive if $\tau$ satisfies
	$$
	(1+\tau\alpha)[\Re(\lambda^2-\lambda)+\gamma_3\Im(\lambda^2-\lambda)]-\tau\alpha|\Re(\lambda-1)+\gamma_3\Im(\lambda-1)|
	-\tau(1+\gamma_3^2)G(Q^*,Q)>0,
	$$
	or equivalently, 
	\begin{equation}\label{ineq:seim-1-shot:case3}
		1+\tau\alpha-\tau\alpha\frac{|\Re(\lambda-1)+\gamma_3\Im(\lambda-1)|}{\Re(\lambda^2-\lambda)+\gamma_3\Im(\lambda^2-\lambda)}
		-\tau(1+\gamma_3^2) \frac{G(Q^*,Q)}{\Re(\lambda^2-\lambda)+\gamma_3\Im(\lambda^2-\lambda)}>0.
	\end{equation}
	Notice that the choice of $\gamma_3$ ensures $\Re(\lambda^2-\lambda)+\gamma_3\Im(\lambda^2-\lambda)>0$, also
	$$
	1+\gamma_3^2=1+\frac{\left(\delta_0+\sin\frac{3\theta_0}{2}\right)^2}{\cos^2\frac{3\theta_0}{2}}=\frac{1+2\delta_0\sin\frac{3\theta_0}{2}+\delta_0^2}{\cos^2\frac{3\theta_0}{2}}=:c
	$$ 
	is a constant greater than $\delta_0^2$. By Lemma \ref{lem:gamma123} (iii) we have
	$$\frac{|\Re(\lambda-1)+\gamma_3\Im(\lambda-1)|}{\Re(\lambda^2-\lambda)+\gamma_3\Im(\lambda^2-\lambda)}
	\le\frac{\sqrt{1+\gamma_3^2}}{\delta_0}=\frac{\sqrt{c}}{\delta_0}.
	$$
	Using again Lemma \ref{lem:gamma123} (iii) and \eqref{ppty:G}, we have
	\begin{multline*}
		\frac{G(Q^*,Q)}{\Re(\lambda^2-\lambda)+\gamma_1\Im(\lambda^2-\lambda)}
		\le\frac{(\norm{H}\norm{M}|\sin\theta|q_2)^2}{2\delta_0|\sin(\theta/2)|}
		=
		\frac{2}{\delta_0}\left|\sin\frac{\theta}{2}\right|\cos^2\frac{\theta}{2}\norm{H}^2\norm{M}^2q_2^2
		\\
		\le\frac{2}{\delta_0}\norm{H}^2\norm{M}^2q_2^2.
	\end{multline*}
	Inserting the two previous inequalities in \eqref{ineq:seim-1-shot:case3} gives the desired results using definitions \eqref{eq:qB} and \eqref{eq:pqB:rho(B)<1} of $q_2$.
\end{proof}

%By the properties of $G$
%$$G(P^*+\gamma_3 Q^*,P+\gamma_3 Q) \ge 0,\quad G(Q^*,Q)\le(\norm{H}\norm{M}\norm{Q})^2\norm{y}^2$$
%and by the estimate $\norm{Q}\le|\sin\theta|q_2$, the left-hand side of \eqref{eq:seim-1-shot:case3} will be strictly positive if $\tau$ satisfies:
%$$\tau<\frac{\Re(\lambda^2-\lambda)+\gamma_3\Im(\lambda^2-\lambda)}{(1+\gamma_3^2)\left(\norm{H}\norm{M}|\sin\theta|q_2\right)^2}.$$
%Since by Lemma \ref{lem:gamma123} (iii) $\Re(\lambda^2-\lambda)+\gamma_3\Im(\lambda^2-\lambda)>2\delta_0\left|\sin\frac{\theta}{2}\right|$, it is sufficient to choose 
%$$\tau<\cfrac{\delta_0}{2(1+\gamma_3^2)\norm{H}^2\norm{M}^2q_2^2}=\cfrac{1}{2\norm{H}^2\norm{M}^2q_2^2}\cdot\cfrac{\delta_0\cos^2\frac{3\theta_0}{2}}{1+2\delta_0\sin\frac{3\theta_0}{2}+\delta_0^2},$$ 
%where we have used the definition of $\gamma_3$. To conclude we use definition \eqref{eq:q2B} of $q_2$.

\subsection{Final result ($k=1$)}

Considering Proposition~\ref{prop:seim-1-shot:tau:B=0} (for $B=0$), Proposition~\ref{prop:seim-1-shot:tau:real-lam} (for real eigenvalues and $B\neq 0$) and taking the bound in Proposition~\ref{prop:seim-1-shot:tau:complex-lam} (for complex eigenvalues and $B\neq 0$), we obtain a sufficient condition on the descent step $\tau$ to ensure convergence of the semi-implicit one-step one-shot method. 

\begin{theorem}[Convergence of semi-implicit one-step one-shot]
	\label{th:seim-1-shot:tau:all}
	Under assumption \eqref{hypo}, the one-step one-shot method \eqref{alg:seim-1-shot} converges for sufficiently small $\tau$. In particular, for $\norm{B}<1$, there exists an explicit piecewise (at most a $4$th order)
	polynomial function $\mathcal{P}_1$  and a pure constant $C>0$ such that $\mathcal{P}_1>0$ on $[0,1)$ and it is enough to take 
	$$\tau<\left(\frac{\norm{H}^2\norm{M}^2}{(1-\norm{B})^4}\mathcal{P}_1(\norm{B})+C\alpha\right)^{-1}.$$
	%$$\tau<\frac{1}{\frac{\norm{H}^2\norm{M}^2}{(1-\norm{B})^4}\mathcal{P}_1(\norm{B})+C\alpha}.$$
\end{theorem}	

We emphasize that the bound on $\tau$ in Theorem \ref{th:seim-1-shot:tau:all} depends only on $\norm{B},\norm{M},\norm{H}$ and the regularization parameter $\alpha$. This bound does not depend on the dimensions of $\sigma$, $u$ and $g$.

\section{Convergence of the multi-step one-shot method ($k\ge 1$)}
\label{sec:seim-k-shot}

We now tackle the general case of semi-implicit multi-step one-shot methods, that is algorithm \eqref{alg:seim-k-shot} with $k\ge 1$. The procedure is quite similar to the case $k=1$ but with more involved technicalities.

\subsection{Block iteration matrix and eigenvalue equation}
\label{sec:seim-k-shot:eigen-eq}

Let $k\ge1$ be the number of inner iterations for $u$ and $p$. 
First we express $(\sigma^{n+1},u^{n+1},p^{n+1})$ in terms of $(\sigma^n,u^n,p^n)$ in a matrix form as for the case $k=1$.
%explicitly rewriting the recursions for $u$ and $p$. 
More precisely, the system \eqref{alg:seim-k-shot} can be equivalently written as
\begin{equation}
	\label{seim-k-shot:rewrite}
	\begin{cases}
		\sigma^{n+1}=\sigma^n-\tau M^*p^n,\\
		u^{n+1}=B^ku^n+T_kM\sigma^n-\tau T_kMM^*p^n+T_kF,\\
		p^{n+1}=[(B^*)^k-\tau X_kMM^*]p^n+U_ku^n+X_kM\sigma^n+X_kF-T_k^*H^*g\\
	\end{cases}
\end{equation}
where
\begin{align}
	T_k\coloneqq& I+B+...+B^{k-1}=(I-B)^{-1}(I-B^k), \label{eq:Tk} \\
	U_k\coloneqq& (B^*)^{k-1}H^*H+(B^*)^{k-2}H^*HB+...+H^*HB^{k-1}, \quad\text{and}, \label{eq:Uk} \\
	X_k\coloneqq& \left\{\begin{array}{cl}
		(B^*)^{k-2}H^*HT_1+(B^*)^{k-3}H^*HT_2+...+H^*HT_{k-1} & \text{if }k\ge 2,\\
		0 & \text{if }k=1. 
	\end{array}\right. \label{eq:Xk}
\end{align}
%\begin{equation}
%	\label{seim-k-shot:rewrite}
%	\begin{cases}
	%		\sigma^{n+1}=\sigma^n-\tau M^*p^n,\\
	%		u^{n+1}=B^ku^n+T_kM\sigma^n-\tau T_kMM^*p^n+T_kF,\\
	%		p^{n+1}=[(B^*)^k-\tau X_kMM^*]p^n+U_ku^n+X_kM\sigma^n+X_kF-T_k^*H^*g\\
	%	\end{cases}
%\end{equation}
%where
%\begin{equation}\label{eq:Tk}
%	T_k\coloneqq I+B+...+B^{k-1}=(I-B)^{-1}(I-B^k),
%\end{equation}
%\begin{equation}
%	\label{eq:Uk}
%	U_k\coloneqq (B^*)^{k-1}H^*H+(B^*)^{k-2}H^*HB+...+H^*HB^{k-1},
%\end{equation}
%\begin{equation}\label{eq:Xk}
%	X_k\coloneqq \left\{\begin{array}{cl}
	%		(B^*)^{k-2}H^*HT_1+(B^*)^{k-3}H^*HT_2+...+H^*HT_{k-1} & \text{if }k\ge 2,\\
	%		0 & \text{if }k=1. 
	%	\end{array}\right.
%\end{equation}
%Note that \eqref{k-shot expl n+1} ($k$-step one-shot) can be obtained from \eqref{k-shot expl n} (shifted $k$-step one-shot) by replacing $\sigma^n$ with $\sigma^{n+1}=\sigma^n-\tau M^*p^n$ in the equations for $u$ and $p$, which yields two extra terms in \eqref{k-shot expl n+1}. 
%In what follows we study the $k$-step one-shot method. 
Before analyzing recursion \eqref{seim-k-shot:rewrite}, we gather in the following lemma some useful properties of $T_k, U_k$ and $X_k$.
\begin{lemma}\label{lem:ppty:TUXk} 
	We have the following properties.
	\begin{enumerate}[label=(\roman*)]
		\item The matrices $U_k$ and $X_k$ can be respectively rewritten as 
		\begin{gather*} 
			U_k=\sum_{i+j=k-1} (B^*)^iH^*HB^j, \; \forall k\ge 1, \displaybreak[0] \\
			X_k=\sum_{l=0}^{k-2}\sum_{i+j=l} (B^*)^iH^*HB^j=\sum_{l=1}^{k-1}U_l, \; \forall k\ge 2.
		\end{gather*}
		%		\begin{align*} 
			%			U_k=\sum_{i+j=k-1} (B^*)^iH^*HB^j,\forall k\ge 1, \quad\mbox{and}\quad
			%			X_k=\sum_{l=0}^{k-2}\sum_{i+j=l} (B^*)^iH^*HB^j=\sum_{l=1}^{k-1}U_l, \forall k\ge 2.
			%	\end{align*}
		\item The matrices $U_k$ and $X_k$ are self-adjoint: $U_k^*=U_k$, $X_k^*=X_k$.
		% \item $U_k\overset{k\to\infty}{\longrightarrow}0$, 	$X_k\overset{k\to\infty}{\longrightarrow} (I-B^*)^{-1}H^*H(I-B)^{-1}$. This is consistent with the usual GD.
		\item We have the relation 
		\begin{equation}\label{iden:TUXk}
			U_kT_k-X_kB^k+X_k=T_k^*H^*HT_k, \quad\forall k\ge 1.
		\end{equation}
	\end{enumerate}
\end{lemma}	
\begin{proof}
	Property  is easy to check from the definitions \eqref{eq:Tk}, \eqref{eq:Uk} and \eqref{eq:Xk}. Property (ii) straightforwardly follows from (i). 
	
	% \noindent (iii) By Lemma \ref{specradi-norm} there exists a subordinate norm $\norm{\cdot}$ such that $\norm{B}<1$. Since all norms on a finite space are equivalent, we can assume that $\norm{B}<1$ without changing the convergent essence. Then, the inequality $\norm{U_k}\le k\norm{H}^2\norm{B}^{k-1}$ and the identity
	% $$(I-B^*)^{-1}H^*H(I-B)^{-1}=\left[\sum_{i=0}^\infty (B^*)^i\right]H^*H \left(\sum_{i=0}^\infty B^i\right)=\sum_{i,j\ge 0}(B^*)^iH^*HB^j$$
	% yield the desired results.
	
	Now we prove (iii). For $k=1$, we have $U_1=H^*H$, $T_1=I$ and $X_1=0$, hence the identity is verified. For $k\ge 2$, we remark that $X_{k+1}=B^*X_k+H^*HT_k$ by definition \eqref{eq:Xk}. Thanks to (ii) we know that $X_{k+1}$ is self-adjoint, hence $X_{k+1}=X_{k+1}^*=X_kB+T_k^*H^*H$. On the other hand, from  we get that $X_{k+1}=X_k+U_k$. Thus,
	$$X_k+U_k=X_kB+T_k^*H^*H,\quad\mbox{ or equivalently, }\quad U_k=X_k(B-I)+T_k^*H^*H.$$
	Finally,
	$$U_kT_k=X_k(B-I)T_k+T_k^*H^*HT_k=X_k(B^k-I)+T_k^*H^*HT_k.$$
\end{proof}

Now, we consider the errors $(\sigma^n-\sigma^\mathrm{ex}_\alpha,u^n-u(\sigma^\mathrm{ex}_\alpha),p^n-p(\sigma^\mathrm{ex}_\alpha))$ with respect to the regularized solution  at the $n$-th iteration, and, by abuse of notation, we denote them by $(\sigma^n,u^n,p^n)$. We obtain that the errors satisfy
\begin{equation}
	\label{seim-k-shot:sys-err}
	\begin{cases}
		\sigma^{n+1}=\frac{1}{1+\tau\alpha}\sigma^n-\frac{\tau }{1+\tau\alpha}M^*p^n,\\
		u^{n+1}=B^ku^n+\frac{1}{1+\tau\alpha}T_kM\sigma^n-\frac{\tau}{1+\tau\alpha}T_kMM^*p^n,\\
		p^{n+1}=\left[(B^*)^k-\frac{\tau}{1+\tau\alpha} X_kMM^*\right]p^n+U_ku^n+\frac{1}{1+\tau\alpha}X_kM\sigma^n,\\
	\end{cases}
\end{equation}
%for algorithm \eqref{seim-k-shot:rewrite}, 
\noindent or equivalently, by putting in evidence the block iteration matrix
\begin{equation}
	\label{seim-k-shot:itermat}
	\begin{bmatrix}
		p^{n+1}\\ u^{n+1}\\ \sigma^{n+1}
	\end{bmatrix}=\begin{bmatrix}
		(B^*)^k-\frac{\tau }{1+\tau\alpha}X_kMM^* & U_k & \frac{1}{1+\tau\alpha}X_kM \\
		-\frac{\tau }{1+\tau\alpha}T_kMM^* & B^k & \frac{1}{1+\tau\alpha}T_kM \\
		-\frac{\tau}{1+\tau\alpha} M^* & 0 & \frac{1}{1+\tau\alpha}I
	\end{bmatrix}
	\begin{bmatrix}
		p^{n}\\ u^{n}\\ \sigma^{n}
	\end{bmatrix}.
\end{equation}
%	Now recall that a fixed point iteration converges if and only if the spectral radius of its iteration matrix is strictly less than $1$. Therefore in the following propositions we establish eigenvalue equations for the iteration matrix of the $k$-step one-shot method. 

\begin{proposition}
	%[Eigenvalue equation for the $k$-step one-shot method]
	\label{prop:seim-k-shot:eq-eigen} 
	Assume that $\lambda\in\C, |\lambda|\ge 1$ is an eigenvalue of the iteration matrix in  \eqref{seim-k-shot:itermat}. If $\lambda\in\C$, $\lambda\notin\mathrm{Spec}(B)$ then $\exists \, y\in\C^{n_\sigma}, \norm{y}=1$ such that:
	\begin{equation}\label{eq:seim-k-shot:ori-eq-eigen}
		(1+\tau\alpha)\lambda-1+\tau\lambda\scalar{M^*[\lambda I-(B^*)^k]^{-1}[(\lambda-1)X_k+T_k^*H^*HT_k](\lambda I-B^k)^{-1}My,y}=0.
	\end{equation}
	In particular,  $\lambda=1$ is not an eigenvalue of the iteration matrix.
\end{proposition}
%	\begin{remark}
	%		Since $\rho(B)$ is strictly less than $1$, so are $\rho(B^*), \rho(B^k)$ and $\rho((B^*)^k)$.
	%	\end{remark}		
\noindent The proof is similar to the proof of Proposition~\ref{prop:seim-1-shot:eigen-eq}. The slight difference is that in the calculation we use \eqref{iden:TUXk} to simplify some terms and exploit the fact that $T_k$ and $(\lambda I-B^k)^{-1}$ commute.

\begin{proof}
	Since $\lambda\in\C$ is an eigenvalue of the iteration matrix, there exists a non-zero vector $(\tilde{p},\tilde{u},y)\in\C^{n_u+n_u+n_\sigma}$ such that 
		\begin{equation*}
			\begin{cases}
				\lambda y = \frac{1}{1+\tau\alpha}y-\frac{\tau}{1+\tau\alpha}M^*\tilde{p},\\
				\lambda\tilde{u}= B^k\tilde{u}+\frac{1}{1+\tau\alpha}T_kMy-\frac{\tau}{1+\tau\alpha} T_kMM^*\tilde{p},\\
				\lambda\tilde{p}=\left[(B^*)^k-\frac{\tau}{1+\tau\alpha}X_kMM^*\right]\tilde{p}+U_k\tilde{u}+\frac{1}{1+\tau\alpha}X_kMy.
			\end{cases}
		\end{equation*}
		By inserting the first equation into the second and the third equations, we simplify this system of equations as
		\begin{equation}
			\label{eq:seim-k-shot:eigenvec}
			\begin{cases}
				\lambda y = \frac{1}{1+\tau\alpha}y-\frac{\tau}{1+\tau\alpha}M^*\tilde{p},\\
				\lambda\tilde{u}= B^k\tilde{u}+\lambda T_kMy,\\
				\lambda\tilde{p}=(B^*)^k\tilde{p}+U_k\tilde{u}+\lambda X_kMy.
			\end{cases}
		\end{equation}
		The second equation in \eqref{eq:seim-k-shot:eigenvec} gives us directly $\tilde{u}$ in terms of $y$:
		\begin{equation}
			\label{eq:u(y)_k-shot}
			\tilde{u}=\lambda(\lambda I-B^k)^{-1}T_kMy=\lambda T_k(\lambda I-B^k)^{-1}My.
		\end{equation}
		The third equation in \eqref{eq:seim-k-shot:eigenvec} gives us $\tilde{p}$ in terms of $\tilde{u}$ and $y$:
		$$\tilde{p}=[\lambda I-(B^*)^k]^{-1}U_k\tilde{u}+\lambda [\lambda I-(B^*)^k]^{-1}X_kMy,$$
		then by combing with the equation \eqref{eq:u(y)_k-shot}, we obtain $\tilde{p}$ in terms of $y$:
		\begin{equation}
			\label{eq:p(y)_k-shot}
			\tilde{p}=\lambda(\lambda I-(B^*)^k)^{-1}V(\lambda I-B^k)^{-1}My,
		\end{equation}
		where 
		$$V\coloneqq U_kT_k+X_k(\lambda I -B^k)=(\lambda-1)X_k+T_k^*H^*HT_k$$ 
		thanks to Lemma \ref{lem:ppty:TUXk}. We also see that $y\neq 0$; indeed if $y=0$ then equations \eqref{eq:u(y)_k-shot} and \eqref{eq:p(y)_k-shot} give $u=0$ and $p=0$, that is a contradiction. Next, inserting the expression of $\tilde{p}$ in equation \eqref{eq:p(y)_k-shot} back into the first equation in \eqref{eq:seim-k-shot:eigenvec}, we get
		$$\lambda y=\frac{1}{1+\tau\alpha}y-\frac{\tau}{1+\tau\alpha}\lambda M^*(\lambda I-(B^*)^k)^{-1}V(\lambda I-B^k)^{-1}My,$$
		which leads to
		\begin{equation}\label{eq:seim-k-shot:pre-eq-eigen}
			\left[(1+\alpha\tau)\lambda-1\right]y+\tau\lambda M^*[\lambda I-(B^*)^k]^{-1}V(\lambda I-B^k)^{-1}My=0.
		\end{equation}
		Finally, by taking scalar product of \eqref{eq:seim-k-shot:pre-eq-eigen} with $y$, then dividing by $\norm{y}^2$, we obtain \eqref{eq:seim-k-shot:ori-eq-eigen}.
	
	\color{black}{\noindent Now assume that $\lambda=1$ is an eigenvalue of the iteration matrix, then \eqref{eq:seim-k-shot:ori-eq-eigen} yields
		$$\alpha+\norm{H(I-B)^{-1}My}^2=0,$$
		which cannot be true due to the injectivity of $H(I-B)^{-1}M$.}
\end{proof}

In the following sections we will show that, for sufficiently small $\tau$, equation \eqref{eq:seim-k-shot:ori-eq-eigen} admits no solution $|\lambda|\ge 1$, thus algorithm \eqref{alg:seim-k-shot} converges. When $\lambda\neq 0$, it is convenient to rewrite \eqref{eq:seim-k-shot:ori-eq-eigen} as
\begin{multline}	
	\label{eq:seim-k-shot:eq-eigen}
	(1+\tau\alpha)\lambda^2-\lambda
	+\tau\scalar{M^*\left[I-(B^*)^k/\lambda\right]^{-1}[(\lambda-1)X_k+T_k^*H^*HT_k]\left(I-B^k/\lambda\right)^{-1}My,y}=0.
\end{multline}

\begin{remark}
	\label{rk:one-dim}
	The simple scalar case where $n_u, n_\sigma, n_g =1$ and $\alpha=0$ is analyzed in \cite{bon22}, for which necessary and sufficient conditions on $\tau$ are derived.
\end{remark}

\begin{remark}\label{rk:B=0,k>2}
	Note that when $B=0$ and $k \ge 2$, the semi-implicit $k$-step one-shot \eqref{alg:seim-k-shot} is equivalent to the semi-implicit gradient descent method \eqref{alg:seimgd}, which converges if and only if $(\rho(A^*A)-\alpha)\tau<2$.
\end{remark}

For the analysis we use some auxiliary results proved in Appendix~\ref{app:lems}, and the following bounds for $s(B^k), T_k, X_k$.
\begin{lemma}\label{lem:sTXk}
	If $\norm{B}<1$ then for every $k\ge 1$:
	$$
	s(B^k)=s((B^*)^k)\le\frac{1}{1-\norm{B}^k}, \quad
	\norm{T_k} \le \frac{1-\norm{B}^k}{1-\norm{B}}
	$$
	and
	$$
	\norm{X_k} \le \frac{\norm{H}^2(1-k\norm{B}^{k-1}+(k-1)\norm{B}^k)}{(1-\norm{B})^2}.
	$$
	%				$$
	%		\begin{array}{c}
		%			s(B^k)=s((B^*)^k)\le\frac{1}{1-\norm{B}^k}, \quad
		%			\norm{T_k} \le \frac{1-\norm{B}^k}{1-\norm{B}},\quad
		%			\norm{X_k} \le \frac{\norm{H}^2(1-k\norm{B}^{k-1}+(k-1)\norm{B}^k)}{(1-\norm{B})^2}.
		%		\end{array}
	%		$$
	%		\[
	%		s(B^k)\le\frac{1}{1-\norm{B}^k}, \quad
	%		\norm{T_k} \le \frac{1-\norm{B}^k}{1-\norm{B}}, \quad
	%		\norm{X_k} \le \frac{\norm{H}^2(1-k\norm{B}^{k-1}+(k-1)\norm{B}^k)}{(1-\norm{B})^2}.
	%		\]
\end{lemma}
\begin{proof}
	The bound for $s(B^k)$ is proved using Lemma~\ref{lem:inv(I-T/z)} and $\norm{B^k} \le \norm{B}^k$. Next, from \eqref{eq:Tk} we have
	$$\norm{T_k}\le 1+\norm{B}+...+\norm{B}^{k-1}=\frac{1-\norm{B}^k}{1-\norm{B}}.$$
	From \eqref{eq:Xk}, %if $k=1$ we have $X_1=0$, and 
	if $k\ge 2$ we have
	%		$$\begin{array}{cl}
		%			\norm{X_k}&\le\norm{H}^2 \bigl(\norm{B}^{k-2}+\norm{B}^{k-3}(1+\norm{B})+...+(1+\norm{B}+...+\norm{B}^{k-2})\bigr)\\
		%			&=\displaystyle\norm{H}^2(1+2\norm{B}+...+(k-1)\norm{B}^{k-2})=\frac{\norm{H}^2(1-k\norm{B}^{k-1}+(k-1)\norm{B}^k)}{(1-\norm{B})^2}.
		%		\end{array}
	%		$$
	\begin{align*}
		\norm{X_k}&\le\norm{H}^2 \bigl(\norm{B}^{k-2}+\norm{B}^{k-3}(1+\norm{B})+...+(1+\norm{B}+...+\norm{B}^{k-2})\bigr)\\
		&=\displaystyle\norm{H}^2(1+2\norm{B}+...+(k-1)\norm{B}^{k-2})\\
		&=\frac{\norm{H}^2(1-k\norm{B}^{k-1}+(k-1)\norm{B}^k)}{(1-\norm{B})^2}.
	\end{align*}
	
\end{proof}

\subsection{Location of eigenvalues in the complex plane}
\label{subsec:seim-k-shot:complex-eigen}
We first establish conditions on the descent step $\tau>0$ such that the real eigenvalues stay inside the unit disk. Recall that we have already proved that $\lambda=1$ is not an eigenvalue for any $k$.

\begin{proposition}[Real eigenvalues]
	\label{prop:seim-k-shot:tau:real-lam}
	Let $0\neq \rho(B)<1$ and $\lambda\in\R, \lambda\neq 1, |\lambda|\ge 1$. Then equation  \eqref{eq:seim-k-shot:eq-eigen} cannot hold if $\tau>0$ and
	$$\left(\norm{M}^2\norm{X_k}s(B^k)^2-\frac{1}{2}\alpha\right)\tau<1.$$
	Moreover, if $\norm{B}<1$, the result is also true if $\tau>0$ and
	$$\left(\frac{\norm{H}^2\norm{M}^2}{(1-\norm{B})^2(1-\norm{B}^k)^2}\left(1-k\norm{B}^{k-1}+(k-1)\norm{B}^k\right)-\frac{1}{2}\alpha\right)\tau<1.$$
	%$$\tau<\frac{1}{\frac{\norm{H}^2\norm{M}^2}{(1-\norm{B})^2(1-\norm{B}^k)^2}\left(1-k\norm{B}^{k-1}+(k-1)\norm{B}^k\right)-\frac{1}{2}\alpha}.$$
\end{proposition}
\begin{proof}
	When $\lambda\in\R$ equation \eqref{eq:seim-k-shot:eq-eigen} can be rewritten as
	\begin{multline}
		\label{eq:seim-k-shot:eq-eigen-real-lam}
		(1+\tau\alpha)\lambda^2-\lambda
		+\tau\norm{HT_k\left(I-\frac{B^k}{\lambda}\right)^{-1}My}^2
		\\
		+\tau(\lambda-1)\scalar{M^*\left[I-\frac{(B^*)^k}{\lambda}\right]^{-1}X_k\left(I-\frac{B^k}{\lambda}\right)^{-1}My,y}=0.
	\end{multline}
	We show that we can choose $\tau$ so that the left-hand side of the above equation is positive. First, we note that
	$$\left|\scalar{M^*\left[I-\frac{(B^*)^k}{\lambda}\right]^{-1}X_k\left(I-\frac{B^k}{\lambda}\right)^{-1}My,y}\right|\le\norm{M}^2\norm{X_k}s(B^k)^2.$$
	If $\lambda>1$, we rewrite equation \eqref{eq:seim-k-shot:eq-eigen-real-lam} again as
	\begin{multline*}
		(1+\tau\alpha)\lambda(\lambda-1)+\tau\alpha+\tau\norm{HT_k\left(I-\frac{B^k}{\lambda}\right)^{-1}My}^2
		\\
		+\tau(\lambda-1)\scalar{M^*\left[I-\frac{(B^*)^k}{\lambda}\right]^{-1}X_k\left(I-\frac{B^k}{\lambda}\right)^{-1}My,y}=0.
	\end{multline*}
	Since $\lambda(\lambda-1)\ge\lambda-1$, $\norm{HT_k\left(I-\frac{B^k}{\lambda}\right)^{-1}My}^2\ge 0$ and $\tau\alpha\ge0$, we choose $\tau$ such that
	\[
	(1+\tau\alpha)(\lambda-1)-\tau(\lambda-1)\norm{M}^2\norm{X_k}s(B^k)^2>0,
	\]
	or equivalently,
	\[
	1+\tau\alpha-\tau\norm{M}^2\norm{X_k}s(B^k)>0.
	\]
	If $\lambda\le-1$, we consider equation \eqref{eq:seim-k-shot:eq-eigen-real-lam}. Since $\norm{HT_k\left(I-\frac{B^k}{\lambda}\right)^{-1}My}^2\ge 0$, we choose $\tau$ such that
	\[
	(1+\tau\alpha)\lambda^2-\lambda-\tau(1-\lambda)\norm{M}^2\norm{X_k}s(B^k)^2>0,
	\]
	or equivalently,
	\[
	\frac{(1+\tau\alpha)\lambda^2-\lambda}{1-\lambda}-\tau\norm{M}^2\norm{X_k}s(B^k)^2>0,
	\]
	Since the function $\lambda\mapsto\frac{(1+\tau\alpha)\lambda^2-\lambda}{\lambda-1}$ is decreasing on $(-\infty,-1]$, it suffices to choose $\tau$ such that
	\[
	1+\frac{\alpha}{2}\tau-\tau\norm{M}^2\norm{X_k}s(B^k)^2>0,
	\]
	which proves the first statement of the proposition. Finally, the case $\norm{B}<1$ can be deduced using Lemma \ref{lem:sTXk}.
\end{proof}

For the general case of complex eigenvalues, the study is much more complicated and technical. The following proposition summarizes the results obtained. 

\begin{proposition}[Complex eigenvalues]
	%[$k$-step one-shot method]
	\label{prop:seim-k-shot:tau:complex-lam}
	If $0\neq \rho(B)<1$, there exists $\tau>0$ sufficiently small such that equation \eqref{eq:seim-k-shot:eq-eigen} admits no solution $\lambda\in\C\backslash\R$, $|\lambda|\ge 1$.  
	In particular, if $\norm{B}<1$, given any $\delta_0>0$ and $0<\theta_0<\frac{\pi}{4}$, one can choose
	\[
	\tau <\min_{1\le i\le 3}
	\left(\frac{\norm{H}^2\norm{M}^2}{(1-\norm{B})^2(1-\norm{B}^k)^2}\psi_i(k,\norm{B})+C_i\alpha\right)^{-1},
	\]
	%		\[
	%	\tau <\min_{1\le i\le 3}
	%	\frac{1}{\frac{\norm{H}^2\norm{M}^2}{(1-\norm{B})^2(1-\norm{B}^k)^2}\psi_i(k,\norm{B})+C_i\alpha},
	%	\]
	where the functions $\psi_i$ are defined by
	%	$$
	%	\psi_1(k,b)\coloneqq 4b^{2k}+\sqrt{2}[1-kb^{k-1}+(k-1)b^k](1+b^k),
	%	$$
	%	
	%	$$\psi_2(k,b)\coloneqq\left[ \frac{1}{2\sin\frac{\theta_0}{2}}(1-b^k)^2+\sqrt{2}(1-kb^{k-1}+(k-1)b^k)\right](1+b^k)^2,$$
	\begin{equation*}
		\begin{array}{rl}
			\psi_1(k,b)\coloneqq& 4b^{2k}+\sqrt{2}[1-kb^{k-1}+(k-1)b^k](1+b^k),
			\\[1ex]
			\psi_2(k,b)\coloneqq& \left( \frac{1}{2\sin\frac{\theta_0}{2}}(1-b^k)^2+\sqrt{2}(1-kb^{k-1}+(k-1)b^k)\right)(1+b^k)^2
			\\[2ex]
			\text{and }\psi_3(k,b)\coloneqq&\displaystyle\frac{2c\sin\frac{\theta_0}{2}}{\delta_0}b^{2k}+\frac{\sqrt{c}}{\delta_0}[1-kb^{k-1}+(k-1)b^k](1+b^{2k})\\[2ex]
			&\displaystyle+2\max\left(\frac{\sqrt{c}}{\delta_0},\frac{\sqrt{c}}{\cos2\theta_0}\right)[1-kb^{k-1}+(k-1)b^k]b^k,\\
		\end{array}
	\end{equation*}
	and the pure constant $C_i$, $1\le i\le 3$, are defined by
	\begin{equation*}
		C_1\coloneqq\sqrt{2}-1,\quad 
		C_2\coloneqq\sqrt{2}+\frac{1}{2\sin\frac{\theta_0}{2}}-1,
		\quad\text{and}\quad
		C_3\coloneqq\frac{\sqrt{c}}{\delta_0}-1
	\end{equation*}
	with
	\begin{equation*}
		c\coloneqq\frac{1+2\delta_0\sin\frac{3\theta_0}{2}+\delta_0^2}{\cos^2\frac{3\theta_0}{2}}.
	\end{equation*}
\end{proposition}	

%match with k=1!
%	$$\psi_1(1,b)=4b^2,\quad\psi_2(1,b)=\cfrac{1}{2\sin\frac{\theta_0}{2}}(1+b)^2(1-b)^2,\quad\psi_3(1,b)=\cfrac{2c}{\delta_0}b^2,$$

\begin{proof}
	\textbf{Step 1. Rewrite equation \eqref{eq:seim-k-shot:eq-eigen} by separating real and imaginary parts.}
	
	Let $\lambda=R(\cos\theta+\ic\sin\theta)$ in polar form where $R=|\lambda| \ge 1$ and $\theta\in(-\pi,\pi)$, $\theta\neq 0$. Write ${1}/{\lambda}=r(\cos\phi+\ic\sin\phi)$ in polar form where $r={1}/{|\lambda|}={1}/{R}\le 1$ and $\phi=-\theta\in(-\pi,\pi)$. 
	By Lemma \ref{lem:decomPQ} applied to $T=B^k$, we have
	$$\left(I-\frac{B^k}{\lambda}\right)^{-1}=P_k(\lambda)+\ic Q_k(\lambda),\quad \left(I-\frac{(B^*)^k}{\lambda}\right)^{-1}=P_k(\lambda)^*+\ic Q_k(\lambda)^*$$
	where $P_k(\lambda)$ and $Q_k(\lambda)$ are $\C^{n_u\times n_u}$ matrices that satisfy the following bounds in the case $\norm{B}<1$ for all $|\lambda|\ge 1$:
	\begin{gather}
		\norm{P_k(\lambda)}\le p \coloneqq\cfrac{1}{1-\norm{B}^k}, \label{eq:pBk}\\
		\norm{Q_k(\lambda)}\le q_1 \coloneqq\cfrac{\norm{B}^k}{1-\norm{B}^k} 
		\quad\mbox{and}\quad
		\norm{Q_k(\lambda)}\le q_2|\sin\theta|
		\quad\mbox{with}\quad
		q_2 \coloneqq
		\cfrac{\norm{B}^k}{(1-\norm{B}^k)^2}. \label{eq:qBk}
	\end{gather}
	These bounds still hold in the case $0\neq\rho(B)<1$ with  
	\begin{equation}
		\label{eq:pqBk:rho(B)<1}
		p\coloneqq (1+\norm{B^k})s(B^k)^2,\quad
		q_1\coloneqq \norm{B^k}s(B^k)^2,\quad
		\mbox{and}\quad
		q_2\coloneqq\cfrac{\norm{B^k}}{1-\norm{B^k}}.
	\end{equation}
	
	%		\begin{equation}\label{eq:pBk}
		%		\norm{P}\le p \coloneqq \left\{\begin{array}{cl}
			%		(1+\norm{B^k})s(B^k)^2 &\text{ for general }B,\\
			%		\cfrac{1}{1-\norm{B}^k} &\text{ when }\norm{B}<1;
			%		\end{array}\right.
		%		\end{equation}
	%		\begin{equation}\label{eq:q1Bk}
		%		\norm{Q}\le q_1 \coloneqq \left\{\begin{array}{cl}
			%		\norm{B^k}s(B^k)^2&\text{ for general }B,\\
			%		\cfrac{\norm{B}^k}{1-\norm{B}^k} &\text{ when }\norm{B}<1;
			%		\end{array}\right.
		%		\end{equation}
	%		\begin{equation}\label{eq:q2Bk}
		%		\norm{Q}\le q_2|\sin\theta|,\quad q_2 \coloneqq \left\{\begin{array}{cl}
			%		\norm{B^k}s(B^k)^2 &\text{ for general }B,\\
			%		\cfrac{\norm{B}^k}{(1-\norm{B}^k)^2} &\text{ when }\norm{B}<1.
			%		\end{array}\right.
		%		\end{equation}
	
	% We have the following table:
	% 	\begin{center}
		% 		\begin{tabular}{|c|c|c|c|c|}
			% 			\hline
			% 			& $p$ & $q_1$ & $q_2$ & $p+q_1$ \\ 
			% 			\hline
			% 			General $T=B^k$ & $(1+\norm{T})s(T)^2$
			% 			& $\norm{T}s(T)^2$ & $\norm{T}s(T)^2$
			% 			& $(1+2\norm{T})s(T)^2$\\
			% 			\hline
			% 			$0<\norm{B}<1$ & $\cfrac{1}{1-\norm{B}^k}$
			% 			& $\cfrac{\norm{B}^k}{1-\norm{B}^k}$ & $\cfrac{\norm{B}^k}{(1-\norm{B}^k)^2}$
			% 			& $\cfrac{1+\norm{B}^k}{1-\norm{B}^k}$\\
			% 			\hline
			% 		\end{tabular}
		% 	\end{center}
	
	\noindent To simplify the notation, we will not explicitly write the dependence of $P_k$ and $Q_k$ on $\lambda$. Now we rewrite \eqref{eq:seim-k-shot:eq-eigen} as
	\label{eq:seim-k-shot:ready-split}
	\begin{multline}
		(1+\tau\alpha)(\lambda^2-\lambda)+\tau\alpha(\lambda-1)+\tau\alpha
		\\
		+\tau G_k(P^*_k+\ic Q^*_k,P_k+\ic Q_k)
		+\tau(\lambda-1)L_k(P^*_k+\ic Q^*_k,P_k+\ic Q_k)
		=0
	\end{multline}
	where 
	\begin{equation*}
		G_k(X,Y)=\scalar{M^*XT_k^*H^*HT_kYMy,y}, \forall 	X,Y\in\C^{n_u\times n_u}
	\end{equation*}
	and
	\begin{equation*}
		L_k(X,Y)=\scalar{M^*XX_kYMy,y}, \forall 	X,Y\in\C^{n_u\times n_u}.
	\end{equation*}
	
	%		\begin{itemize}
		%			\item $\forall X,Y_1,Y_2\in\C^{n_u\times n_u}, \forall z_1,z_2\in\C$: \quad
		%			$G(X, z_1Y_1+z_2Y_2)=z_1G(X,Y_1)+z_2G(X,Y_2).$
		%			\item $\forall X_1,X_2,Y\in\C^{n_u\times n_u}, \forall z_1,z_2\in\C$: \quad
		%			$G(z_1X_1+z_2X_2,Y)=z_1G(X_1,Y)+z_2G(X_2,Y).$
		%			\item $\forall X\in\C^{n_u\times n_u}$: \quad
		%			$G(X^*,X)\in \R$.
		%			\item $\forall X,Y\in\C^{n_u\times n_u}$: \quad $G(X,Y)+G(Y^*,X^*)\in\R$, indeed
		%			$$\begin{array}{ll}
			%			G(X,Y)&=\scalar{M^*XT_k^*H^*HT_kYMy,y}=\scalar{y,M^*Y^*T_k^*H^*HT_kX^*My}\\
			%			&=\scalar{M^*Y^*T_k^*H^*HT_kX^*My,y}^*=G(Y^*,X^*)^*.
			%			\end{array}$$ 
		%		\end{itemize}
	
	\noindent Notice that $G_k$ is a bilinear form and $G_k(X,Y)=G_k(Y^*,X^*)^*$ so that $G_k(X,Y)+G_k(X^*,Y^*)$ is real. Similarly, $L_k$ has the same properties as $G_k$ (note that $X_k^*=X_k$ by Lemma \ref{lem:ppty:TUXk}). With these properties of $G_k$ and $L_k$, we expand \eqref{eq:seim-k-shot:ready-split} and take its real and imaginary parts, which yields
	\begin{multline}\label{eq:seim-k-shot:realpart}
		(1+\tau\alpha)\Re(\lambda^2-\lambda)+\tau\alpha\Re(\lambda-1)+\tau\alpha
		+\tau G_{1,k}+\tau [\Re(\lambda-1)L_{1,k}-\Im(\lambda-1)L_{2,k}]=0
	\end{multline}
	and
	\begin{equation}\label{eq:seim-k-shot:impart}
		(1+\tau\alpha)\Im(\lambda^2-\lambda)+\tau\alpha\Im(\lambda-1)
		+\tau G_{2,k}
		+\tau [\Im(\lambda-1)L_{1,k}+\Re(\lambda-1)L_{2,k}]=0
	\end{equation}
	where
	\begin{equation*}
		\begin{array}{c}
			G_{1,k}\coloneqq G_k(P^*_k,P_k)-G_k(Q^*_k,Q_k),\quad G_{2,k}\coloneqq G_k(P^*_k,Q_k)+G_k(Q^*_k,P_k),\\
			L_{1,k}\coloneqq L_k(P^*_k,P_k)-L_k(Q^*_k,Q_k)
			\quad\mbox{and}\quad L_{2,k}\coloneqq L_k(P^*_k,Q_k)+L_k(Q^*_k,P_k).
		\end{array}
	\end{equation*}
	
	\noindent\textbf{Step 2. Use a suitable combination of equations \eqref{eq:seim-k-shot:realpart} and \eqref{eq:seim-k-shot:impart}.}
	
	%		\noindent\textbf{Step 2. Find a suitable combination of equations \eqref{eq:seim-k-shot:realpart} and \eqref{eq:seim-k-shot:impart}, choose $\tau$ so that we obtain a new equation with a left-hand side which is strictly positive/negative.}
	
	\noindent Let $\gamma\in\R$. Multiplying equation \eqref{eq:seim-k-shot:impart} with $\gamma$ then summing it with equation \eqref{eq:seim-k-shot:realpart}, we obtain:
	\begin{multline}\label{eq:seim-k-shot:re+gamma*im}
		(1+\tau\alpha)[\Re(\lambda^2-\lambda)+\gamma\Im(\lambda^2-\lambda)]+\tau\alpha[\Re(\lambda-1)+\gamma\Im(\lambda-1)]+\tau\alpha
		\\
		+\tau G_k(P^*_k+\gamma Q^*_k,P_k+\gamma Q_k)-\tau(1+\gamma^2)G_k(Q^*_k,Q_k)
		\\
		+\tau[\Re(\lambda-1)+\gamma\Im(\lambda-1)]L_{1,k}+\tau[\gamma\Re(\lambda-1)-\Im(\lambda-1)]L_{2,k} =0.
	\end{multline}
	Now we consider four cases of $\lambda$ as in the proof for $k=1$ namely:
	%Lemma~\ref{lem:gamma123}:
	\begin{itemize}
		\item\textit{Case 1.} $\Re(\lambda^2-\lambda)\ge 0$;
		\item\textit{Case 2.} $\Re(\lambda^2-\lambda)<0$ and $\theta\in [\theta_0,\pi-\theta_0]\cup[-\pi+\theta_0,-\theta_0]$ for fixed $0<\theta_0<\frac{\pi}{4}$;
		\item\textit{Case 3.} $\Re(\lambda^2-\lambda)<0$ and $\theta \in (-\theta_0,\theta_0)$ for fixed $0<\theta_0<\frac{\pi}{4}$;
		\item\textit{Case 4.} $\Re(\lambda^2-\lambda)<0$ and $\theta \in (\pi-\theta_0,\pi)\cup(-\pi,-\pi+\theta_0)$ for fixed $0<\theta_0<\frac{\pi}{4}$. 
	\end{itemize}
	Cases 1, 2 and 3 are respectively treated in Lemmas \ref{lem:seim-k-shot:case1}, \ref{lem:seim-k-shot:case2} and \ref{lem:seim-k-shot:case3} below.
	Notice that case 4 corresponds to an empty set, according to Lemma \ref{lem:gamma123} (iv). The statement of the proposition easily follows from the combination of Lemmas \ref{lem:seim-k-shot:case1}, \ref{lem:seim-k-shot:case2} and \ref{lem:seim-k-shot:case3}.
	%		 Note that here just three cases 1, 2 and 3 of $\lambda$ need to be considered, because the analogue of the fourth one is excluded by Lemma \ref{lem:gamma123} (iv). These three cases will be treated in the following four lemmas (Lemmas \ref{lem:seim-k-shot:case1}--\ref{lem:seim-k-shot:case3}), which together give the statement of this proposition.
\end{proof}

In the lemmas below we shall make use of the following obvious properties where $\norm{y}=1$:
\begin{equation}\label{ppty:Gk}
	0\le G_k(X^*,X)=\norm{HT_kXMy}^2\le (\norm{H}\norm{M}\norm{T_k}\norm{X})^2,\quad\forall X\in\C^{n_u\times n_u}.
\end{equation}
\begin{multline}\label{ppty:L1k}
	|L_{1,k}|=|L_k(P^*_k,P_k)-L_k(Q^*_k,Q_k)|
	\le |L_k(P^*_k,P_k)|+|L_k(Q^*_k,Q_k)|\\
	\le \norm{X_k}\norm{M}^2(\norm{P_k}^2+\norm{Q_k}^2)
	\le \norm{X_k}\norm{M}^2(p^2+q_1^2)
\end{multline}
and
\begin{multline}\label{ppty:L2k}
	|L_{2,k}|=|L_k(P^*_k,Q_k)+L_k(Q^*_k,P_k)|
	\le |L_k(P^*_k,Q_k)|+|L_k(Q^*_k,P_k)|
	\\
	\le 2\norm{X_k}\norm{M}^2\norm{P_k}\norm{Q_k}
	\le 2\norm{X_k}\norm{M}^2pq_1.
\end{multline}

%	Now we prepare some useful estimates.
%	\begin{itemize}
	%		\item $\forall X\in\C^{n_u\times n_u}$: \quad
	%		$0\le G(X^*,X)=\norm{HT_kXMy}^2\le (\norm{H}\norm{T_k}\norm{M}\norm{X})^2 \norm{y}^2.$
	%		
	%		Since $\norm{Q}\le q_1$ and $\norm{Q}\le q_2|\sin\theta|$, we have
	%		$$G(Q^*,Q)\le (\norm{H}\norm{T_k}\norm{M}q_1)^2 \norm{y}^2
	%		\text{ and }
	%		G(Q^*,Q)\le (\norm{H}\norm{T_k}\norm{M}q_2\sin|\theta|)^2 \norm{y}^2.$$
	%		\item By Cauchy-Schwarz inequality we have
	%		$$|\Re(\lambda-1)+\gamma\Im(\lambda-1)|\le \sqrt{1+\gamma^2}|\lambda-1|; \quad
	%		|\gamma\Re(\lambda-1)-\Im(\lambda-1)|\le \sqrt{1+\gamma^2}|\lambda-1|.$$
	%		\item $\forall X,Y\in\C^{n_u\times n_u}$: \quad
	%		$|L(X,Y)|=|\scalar{M^*XX_kYMy,y}|\le\norm{X_k}\norm{M}^2\norm{X}\norm{Y}\norm{y}^2.$
	%		Hence
	%		$$\begin{array}{ll}
		%			|L_1|&=|L(P^*,P)-L(Q^*,Q)|
		%			\le |L(P^*,P)|+|L(Q^*,Q)|\\
		%			&\le \norm{X_k}\norm{M}^2(\norm{P}^2+\norm{Q}^2)\norm{y}^2
		%			\le \norm{X_k}\norm{M}^2(p^2+q_1^2)\norm{y}^2,
		%		\end{array}$$
	%		$$\begin{array}{ll}
		%			|L_2|&=|L(P^*,Q)+L(Q^*,P)|
		%			\le |L(P^*,Q)|+|L(Q^*,P)|\\
		%			&\le 2\norm{X_k}\norm{M}^2\norm{P}\norm{Q}\norm{y}^2
		%			\le 2\norm{X_k}\norm{M}^2pq_1\norm{y}^2,
		%		\end{array}$$
	%		and then
	%		$$\begin{array}{ll}
		%			&|[\Re(\lambda-1)+\gamma\Im(\lambda-1)]L_1+[\gamma\Re(\lambda-1)-\Im(\lambda-1)]L_2|\\
		%			\le&|\Re(\lambda-1)+\gamma\Im(\lambda-1)||L_1|+|\gamma\Re(\lambda-1)-\Im(\lambda-1)||L_2|\\
		%			\le&\sqrt{1+\gamma^2}|\lambda-1|\norm{X_k}\norm{M}^2(p^2+q_1^2+2pq_1)\norm{y}^2\\
		%			=&\sqrt{1+\gamma^2}|\lambda-1|\norm{X_k}\norm{M}^2(p+q_1)^2\norm{y}^2.
		%		\end{array}
	%		$$
	%	\end{itemize}

\begin{lemma}[Case 1]
	\label{lem:seim-k-shot:case1} Let $\rho(B)<1$ and let $|\lambda|\ge 1$ with $\Re(\lambda^2-\lambda)\ge 0$. Then equation \eqref{eq:seim-k-shot:eq-eigen} cannot hold if $\tau>0$ and
	\begin{multline*}
		\Big(4\norm{H}^2\norm{M}^2\norm{T_k}^2\norm{B^k}^2s(B^k)^4+\sqrt{2}\norm{M}^2\norm{X_k}(1+2\norm{B^k})^2s(B^k)^4
		+(\sqrt{2}-1)\alpha\Big)\tau
		\\< 1.
	\end{multline*}
	%			 $$\begin{array}{lr}
		%				\tau< &\left[4\norm{H}^2\norm{M}^2\norm{T_k}^2\norm{B^k}^2s(B^k)^4+\sqrt{2}\norm{M}^2\norm{X_k}(1+2\norm{B^k})^2s(B^k)^4\right.\\
		%				&\left.+(\sqrt{2}-1)\alpha\right]^{-1}
		%			\end{array}$$ 
	%			 $$\begin{array}{ll}
		%			 	\tau<1\Big/ &\left[4\norm{H}^2\norm{M}^2\norm{T_k}^2\norm{B^k}^2s(B^k)^4\right.\\
		%			 	&\left.+\sqrt{2}\norm{M}^2\norm{X_k}(1+2\norm{B^k})^2s(B^k)^4+(\sqrt{2}-1)\alpha\right]
		%			 \end{array}$$ 
	Moreover, if $\norm{B}<1$, the result is also true if $\tau>0$ and
	$$\left(\frac{\norm{H}^2\norm{M}^2}{(1-\norm{B})^2(1-\norm{B}^k)^2}\psi_1(k,\norm{B})+(\sqrt{2}-1)\alpha\right)\tau<1.$$
	%			$$\tau<\frac{1}{\frac{\norm{H}^2\norm{M}^2}{(1-\norm{B})^2(1-\norm{B}^k)^2}\psi_1(k,\norm{B})+(\sqrt{2}-1)\alpha}.$$
	where $\psi_1(k,b)\coloneqq4b^{2k}+\sqrt{2}(1-kb^{k-1}+(k-1)b^k)(1+b^k)$.
\end{lemma}
%	\begin{remark}
	%		$\psi(1,0)=0$.
	%	\end{remark}
\begin{proof}
	Define $$\gamma_1=\gamma_1(\lambda)\coloneqq\left\{\begin{array}{cc}
		1 &\text{if } \Im (\lambda^2-\lambda)\ge 0,\\
		-1 &\text{if } \Im (\lambda^2-\lambda)<0
	\end{array}\right.$$
	as in Lemma \ref{lem:gamma123}. Writing \eqref{eq:seim-k-shot:re+gamma*im} for $\gamma=\gamma_1$ as in Lemma \ref{lem:gamma123} and using $\gamma_1^2=1$ we obtain
	\begin{multline}\label{eq:seim-k-shot:case1}
		(1+\tau\alpha)[\Re(\lambda^2-\lambda)+\gamma_1\Im(\lambda^2-\lambda)]+\tau\alpha[\Re(\lambda-1)+\gamma_1\Im(\lambda-1)] +\tau\alpha
		\\
		+\tau G_k(P^*_k+\gamma_1 Q^*_k,P_k+\gamma_1 Q_k)-2\tau G_k(Q^*_k,Q_k)\\
		+\tau[\Re(\lambda-1)+\gamma_1\Im(\lambda-1)]L_{1,k}+\tau[\gamma_1\Re(\lambda-1)-\Im(\lambda-1)]L_{2,k}=0.
	\end{multline}
	Since $G_k(P^*_k+\gamma_1 Q^*_k,P_k+\gamma_1 Q_k)\ge 0$ by \eqref{ppty:Gk} and $\tau\alpha\ge 0$, the left-hand side of \eqref{eq:seim-k-shot:case1} is positive if $\tau$ satisfies
	\begin{multline*}
		(1+\tau\alpha)[\Re(\lambda^2-\lambda)+\gamma_1\Im(\lambda^2-\lambda)]-\tau\alpha[\Re(\lambda-1)+\gamma_1\Im(\lambda-1)]
		-2\tau G_k(Q^*_k,Q_k)
		\\
		-\tau[\Re(\lambda-1)+\gamma_1\Im(\lambda-1)]|L_{1,k}|-\tau[\gamma_1\Re(\lambda-1)-\Im(\lambda-1)]|L_{2,k}|>0,
	\end{multline*}
	or equivalently, 
	\begin{multline}\label{ineq:seim-k-shot:case1}
		1+\tau\alpha
		-\tau\alpha\frac{|\Re(\lambda-1)+\gamma_1\Im(\lambda-1)|}{\Re(\lambda^2-\lambda)+\gamma_1\Im(\lambda^2-\lambda)}
		-2\tau \frac{G_k(Q^*_k,Q_k)}{\Re(\lambda^2-\lambda)+\gamma_1\Im(\lambda^2-\lambda)}
		\\
		-\tau\frac{|\Re(\lambda-1)+\gamma_1\Im(\lambda-1)|}{\Re(\lambda^2-\lambda)+\gamma_1\Im(\lambda^2-\lambda)}|L_{1,k}|
		-\tau\frac{|\gamma_1\Re(\lambda-1)-\Im(\lambda-1)|}{\Re(\lambda^2-\lambda)+\gamma_1\Im(\lambda^2-\lambda)}|L_{2,k}| >0.
	\end{multline}
	Notice that the choice of $\gamma_1$ ensures $\Re(\lambda^2-\lambda)+\gamma_1\Im(\lambda^2-\lambda)>0$. 
	In the following we derive upper bounds independent of $\lambda$ for the terms appearing with the negative sign in \eqref{ineq:seim-k-shot:case1}.
	By Lemma \ref{lem:gamma123}  we have
	$$\frac{|\Re(\lambda-1)+\gamma_1\Im(\lambda-1)|}{\Re(\lambda^2-\lambda)+\gamma_1\Im(\lambda^2-\lambda)}
	\le\frac{\sqrt{1+\gamma_1^2}|\lambda-1|}{|\lambda(\lambda-1)|}
	=\frac{\sqrt{2}}{|\lambda|}\le\sqrt{2}$$
	and
	$$\frac{|\gamma_1\Re(\lambda-1)-\Im(\lambda-1)|}{\Re(\lambda^2-\lambda)+\gamma_1\Im(\lambda^2-\lambda)}\le\frac{\sqrt{1+\gamma_1^2}|\lambda-1|}{|\lambda(\lambda-1)|}
	=\frac{\sqrt{2}}{|\lambda|}\le\sqrt{2}.$$
	Using again Lemma \ref{lem:gamma123}  and \eqref{ppty:L2k}, we have
	\begin{multline*}
		\cfrac{G_k(Q^*_k,Q_k)}{\Re(\lambda^2-\lambda)+\gamma_1\Im(\lambda^2-\lambda)}
		\le\cfrac{(\norm{H}\norm{M}|\norm{T_k}\sin\theta|q_2)^2}{2|\sin(\theta/2)|}
		=2\left|\sin\frac{\theta}{2}\right|\cos^2\frac{\theta}{2}\norm{H}^2\norm{M}^2\norm{T_k}^2q_2^2
		\\
		\le2\norm{H}^2\norm{M}^2\norm{T_k}^2q_2^2.
	\end{multline*}
	Recall from \eqref{ppty:L1k} and \eqref{ppty:L2k} that
	$$|L_{1,k}|\le\norm{X_k}\norm{M}^2(p^2+q_1^2),\quad |L_{2,k}|\le2\norm{X_k}\norm{M}^2pq_1.$$
	Inserting these previous inequalities in \eqref{ineq:seim-k-shot:case1} gives the desired result thanks to expressions \eqref{eq:pBk}, \eqref{eq:qBk} and \eqref{eq:pqBk:rho(B)<1} of respectively $p,q_1$ and $q_2$, and thanks to Lemma \ref{lem:sTXk}.
	%	we have the first part of the conclusion using definitions \eqref{eq:pBk}, \eqref{eq:q1Bk}, \eqref{eq:q2Bk}  of $p, q_1, q_2$.
	%		Finally, the conclusion in the case $\norm{B}<1$ can be obtained by Lemma \ref{lem:sTXk}.
\end{proof}

%Since $G(P^*+\gamma_1 Q^*,P+\gamma_1 Q)\ge 0$, and by estimating
%\[
%G(Q^*,Q)\le(\norm{H}\norm{T_k}\norm{M}q_2\sin|\theta|)^2 \norm{y}^2,
%\]
%\begin{multline*}
%	[\Re(\lambda-1)+\gamma_1\Im(\lambda-1)]L_1+[\gamma_1\Re(\lambda-1)-\Im(\lambda-1)]L_2\\
%	\ge-\sqrt{2}|\lambda-1|\norm{X_k}\norm{M}^2(p+q_1)^2\norm{y}^2,
%\end{multline*}
%by Lemma \ref{lem:gamma123}  the left-hand side of \eqref{eq:seim-k-shot:case1} will be strictly positive if $\tau$ satisfies:
%$$\left(2\left(\norm{H}\norm{T_k}\norm{M}q_2\right)^2\frac{|\sin\theta|^2}{|\lambda-1|}+\sqrt{2}\norm{X_k}\norm{M}^2(p+q_1)^2\right)\tau<1.$$
%Since $\frac{|\sin\theta|^2}{|\lambda-1|}\le\frac{|\sin\theta|^2}{2|\sin(\theta/2)|}=2\left| \sin\frac{\theta}{2}\right|\cos^2\frac{\theta}{2}\le 2$, we have the first part of the conclusion using definitions \eqref{eq:pBk}, \eqref{eq:q1Bk}, \eqref{eq:q2Bk}  of $p, q_1, q_2$.
%Finally, the conclusion in the case $\norm{B}<1$ can be obtained by Lemma \ref{lem:sTXk}.

\begin{lemma}[Case 2]
	\label{lem:seim-k-shot:case2} Let $\rho(B)<1$ and let $|\lambda|\ge1$, $\Re(\lambda^2-\lambda)<0$, $\theta\in [\theta_0,\pi-\theta_0]\cup[-\pi+\theta_0,-\theta_0]$ for given $0<\theta_0\le\frac{\pi}{4}$. Then equation \eqref{eq:seim-k-shot:eq-eigen} cannot hold if $\tau>0$ and
	\begin{equation*}
		\Big(\Big(\frac{1}{2\sin\frac{\theta_0}{2}}\norm{H}^2\norm{M}^2\norm{T_k}^2+\sqrt{2}\norm{M}^2\norm{X_k}\Big)(1+2\norm{B^k})^2s(B^k)^4+(\sqrt{2}-1)\alpha\Big)\tau 
		<1.
	\end{equation*}
	%		 $$\begin{array}{ll}
		%				\tau< &\left[\left(\frac{1}{2\sin\frac{\theta_0}{2}}\norm{H}^2\norm{M}^2\norm{T_k}^2+\sqrt{2}\norm{M}^2\norm{X_k}\right)(1+2\norm{B^k})^2s(B^k)^4\right.\\
		%				&\left.+(\sqrt{2}-1)\alpha\right]^{-1}
		%			\end{array}$$ 
	%			 $$\begin{array}{ll}
		%				\tau<1\Big/ &\left[\left(\frac{1}{2\sin\frac{\theta_0}{2}}\norm{H}^2\norm{M}^2\norm{T_k}^2+\sqrt{2}\norm{M}^2\norm{X_k}\right)(1+2\norm{B^k})^2s(B^k)^4\right.\\
		%				&\left.+(\sqrt{2}-1)\alpha\right]
		%			\end{array}$$ 
	Moreover, if $\norm{B}<1$, the result is also true if
	$$\tau<\left(\frac{\norm{H}^2\norm{M}^2}{(1-\norm{B})^2(1-\norm{B}^k)^2}\psi_2(k,\norm{B})+(\sqrt{2}-1)\alpha\right)^{-1}$$
	where $\psi_2(k,b)\coloneqq\left( \frac{1}{2\sin\frac{\theta_0}{2}}(1-b^k)^2+\sqrt{2}(1-kb^{k-1}+(k-1)b^k)\right)(1+b^k)^2$.
\end{lemma}

%		$$\tau<\frac{1}{\left[ \frac{1}{2\sin\frac{\theta_0}{2}}(1-\norm{B}^k)^2+\sqrt{2}(1-k\norm{B}^{k-1}+(k-1)\norm{B}^k)\right](1+\norm{B}^k)^2+(\sqrt{2}-1)\alpha}.$$

\begin{proof}
	Define 
	$$\gamma_2=\gamma_2(\lambda)\coloneqq\left\{\begin{array}{cc}
		-1 &\text{if } \Im (\lambda^2-\lambda)\ge 0,\\
		1 &\text{if } \Im (\lambda^2-\lambda)<0
	\end{array}\right.$$ 
	as in Lemma \ref{lem:gamma123} (ii). 
	Writing \eqref{eq:seim-k-shot:re+gamma*im} for $\gamma=\gamma_2$ as in Lemma \ref{lem:gamma123} (ii) and using $\gamma_2^2=1$, we obtain
	\begin{multline}\label{eq:seim-k-shot:case2}
		(1+\tau\alpha)[\Re(\lambda^2-\lambda)+\gamma_2\Im(\lambda^2-\lambda)]+\tau\alpha[\Re(\lambda-1)+\gamma_2\Im(\lambda-1)]+\tau\alpha
		\\
		+\tau G_k(P^*_k+\gamma_2 Q^*_k,P_k+\gamma_2 Q_k)-2\tau G_k(Q^*_k,Q_k)
		\\
		+\tau[\Re(\lambda-1)+\gamma_2\Im(\lambda-1)]L_{1,k}+\tau[\gamma_2\Re(\lambda-1)-\Im(\lambda-1)]L_{2,k} =0.
	\end{multline}
	Since $G_k(Q^*_k,Q_k)\ge 0$ by \eqref{ppty:Gk}, the left-hand side of \eqref{eq:seim-k-shot:case2} is negative if $\tau$ satisfies
	\begin{multline*}
		(1+\tau\alpha)[\Re(\lambda^2-\lambda)+\gamma_2\Im(\lambda^2-\lambda)]+\tau\alpha|\Re(\lambda-1)+\gamma_2\Im(\lambda-1)|
		+\tau\alpha
		\\
		+\tau G_k(P^*_k+\gamma_2 Q^*_k,P_k+\gamma_2 Q_k)
		\\
		+\tau|\Re(\lambda-1)+\gamma_1\Im(\lambda-1)||L_{1,k}|+\tau|\gamma_1\Re(\lambda-1)-\Im(\lambda-1)||L_{2,k}| <0.
	\end{multline*}
	%			$$
	%			\begin{array}{rl}
		%			(1+\tau\alpha)[\Re(\lambda^2-\lambda)+\gamma_2\Im(\lambda^2-\lambda)]+\tau\alpha|\Re(\lambda-1)+\gamma_2\Im(\lambda-1)|
		%			+\tau\alpha&\\
		%			+\tau G(P^*+\gamma_2 Q^*,P+\gamma_2 Q)&\\
		%			+\tau|\Re(\lambda-1)+\gamma_1\Im(\lambda-1)||L_1|+\tau|\gamma_1\Re(\lambda-1)-\Im(\lambda-1)||L_2| &<0.
		%			\end{array}
	%			$$
	or equivalently, 
	\begin{multline}
		\label{ineq:seim-k-shot:case2}
		-1-\tau\alpha
		+\tau\alpha\frac{|\Re(\lambda-1)+\gamma_2\Im(\lambda-1)|}{|\Re(\lambda^2-\lambda)+\gamma_2\Im(\lambda^2-\lambda)|}
		+\frac{\tau\alpha}{|\Re(\lambda^2-\lambda)+\gamma_2\Im(\lambda^2-\lambda)|}
		\\
		+\tau\frac{G_k(P^*_k+\gamma_2 Q^*_k,P_k+\gamma_2 Q_k)}{|\Re(\lambda^2-\lambda)+\gamma_2\Im(\lambda^2-\lambda)|}
		+\tau\frac{|\Re(\lambda-1)+\gamma_2\Im(\lambda-1)|}{|\Re(\lambda^2-\lambda)+\gamma_2\Im(\lambda^2-\lambda)|}|L_{1,k}|
		\\
		+\tau\frac{|\gamma_2\Re(\lambda-1)-\Im(\lambda-1)|}{|\Re(\lambda^2-\lambda)+\gamma_2\Im(\lambda^2-\lambda)|}|L_{2,k}| <0.
	\end{multline}
	Notice that the choice of $\gamma_2$ ensures $\Re(\lambda^2-\lambda)+\gamma_2\Im(\lambda^2-\lambda)<0$. In the following we derive upper bounds independent of $\lambda$ for the terms appearing with the positive sign in \eqref{ineq:seim-k-shot:case2}. 
	By Lemma \ref{lem:gamma123} (ii) we have
	\begin{align*}
		\frac{|\Re(\lambda-1)+\gamma_2\Im(\lambda-1)|}{|\Re(\lambda^2-\lambda)+\gamma_2\Im(\lambda^2-\lambda)|}
		\le\frac{\sqrt{1+\gamma_2^2}|\lambda-1|}{|\lambda(\lambda-1)|}
		=\frac{\sqrt{2}}{|\lambda|}\le\sqrt{2},\\
		\frac{|\gamma_2\Re(\lambda-1)-\Im(\lambda-1)|}{|\Re(\lambda^2-\lambda)+\gamma_2\Im(\lambda^2-\lambda)|}\le\frac{\sqrt{1+\gamma_2^2}|\lambda-1|}{|\lambda(\lambda-1)|}
		=\frac{\sqrt{2}}{|\lambda|}\le\sqrt{2},
	\end{align*}
	and
	$$\frac{1}{|\Re(\lambda^2-\lambda)+\gamma_2\Im(\lambda^2-\lambda)|}\le\frac{1}{2\sin\frac{\theta_0}{2}}.$$
	Using again Lemma \ref{lem:gamma123} (ii) and  \eqref{ppty:Gk}, we have
	$$\frac{G_k(P^*_k+\gamma_2 Q^*_k,P_k+\gamma_2 Q_k)}{|\Re(\lambda^2-\lambda)+\gamma_2\Im(\lambda^2-\lambda)|}
	\le\frac{\norm{H}^2\norm{M}^2\norm{T_k}^2(p+q_1)^2}{2\sin\frac{\theta_0}{2}}.$$
	Recall from \eqref{ppty:L1k} and \eqref{ppty:L2k} that
	$$|L_{1,k}|\le\norm{X_k}\norm{M}^2(p^2+q_1^2),\quad |L_{2,k}|\le2\norm{X_k}\norm{M}^2pq_1.$$
	Inserting these previous inequalities in \eqref{ineq:seim-k-shot:case2} gives the desired result thanks to expressions \eqref{eq:pBk}, \eqref{eq:qBk} and \eqref{eq:pqBk:rho(B)<1} of respectively $p$ and $q_1$, and thanks to Lemma \ref{lem:sTXk}.
	%			we have the first part of the conclusion using definitions \eqref{eq:pBk}, \eqref{eq:q1Bk} of $p, q_1$.
	%			Finally, the conclusion in the case $\norm{B}<1$ can be obtained by Lemma \ref{lem:sTXk}.
\end{proof}	

%		Since $G(Q^*,Q)\ge 0$, and by estimating $\norm{P+\gamma_2 Q}\le \norm{P}+|\gamma_2|\norm{Q}=\norm{P}+\norm{Q}\le p+q_1$, so that 
%	\[
%	G(P^*+\gamma_2 Q^*,P+\gamma_2 Q)\le[\norm{H}\norm{T_k}\norm{M}(p+q_1)]^2 \norm{y}^2, 
%	\]
%	and 
%	\begin{multline*}
	%		[\Re(\lambda-1)+\gamma_2\Im(\lambda-1)]L_1+[\gamma_2\Re(\lambda-1)-\Im(\lambda-1)]L_2\\
	%		\le\sqrt{2}|\lambda-1|\norm{X_k}\norm{M}^2(p+q_1)^2\norm{y}^2,
	%	\end{multline*}
%	by Lemma \ref{lem:gamma123} (ii), the left-hand side of \eqref{eq:seim-k-shot:case2} will be strictly negative if $\tau$ satisfies:
%	$$\left(\left[\norm{H}\norm{T_k}\norm{M}(p+q_1)\right]^2\frac{1}{|\lambda-1|}+\sqrt{2}\norm{X_k}\norm{M}^2(p+q_1)^2\right)\tau<1.$$
%	Since $\frac{1}{|\lambda-1|}\le\frac{1}{2\sin\frac{\theta_0}{2}}$, we have the first part of the conclusion using definitions \eqref{eq:pBk}, \eqref{eq:q1Bk} of $p, q_1$.
%	Finally, the conclusion in the case $\norm{B}<1$ can be obtained by Lemma \ref{lem:sTXk}.

\begin{lemma}[Case 3]
	\label{lem:seim-k-shot:case3}
	Let $\rho(B)<1$ and let $|\lambda|\ge 1$, $\Re(\lambda^2-\lambda)<0$, $\theta \in (-\theta_0,\theta_0)$ for given $0<\theta_0\le \frac{\pi}{4}$. For any $\delta_0>0$, equation  \eqref{eq:seim-k-shot:eq-eigen} cannot hold if $\tau>0$ and
	\begin{multline*}
		\Big(\Big[\frac{2c\sin\frac{\theta_0}{2} }{\delta_0}\norm{H}^2\norm{M}^2\norm{T_k}^2\norm{B^k}^2+\frac{\sqrt{c}}{\delta_0}\norm{M}^2\norm{X_k}(1+2\norm{B^k}+2\norm{B^k}^2)
		\\
		+2\max\Big(\frac{\sqrt{c}}{\delta_0},\frac{\sqrt{c}}{\cos2\theta_0}\Big)\norm{M}^2\norm{X_k}(\norm{B^k}+\norm{B^k}^2)\Big]s(B^k)^2+(\sqrt{2}-1)\alpha\Big)\tau<1
	\end{multline*}
	%			$$
	%			\begin{array}{rl}
		%				\left(\left[\frac{2c\sin\frac{\theta_0}{2} }{\delta_0}\norm{H}^2\norm{M}^2\norm{T_k}^2\norm{B^k}^2+\frac{\sqrt{c}}{\delta_0}\norm{M}^2\norm{X_k}(1+2\norm{B^k}+2\norm{B^k}^2)\right.\right.&\\
		%				\left.\left.+2\max\left(\frac{\sqrt{c}}{\delta_0},\frac{\sqrt{c}}{\cos2\theta_0}\right)\norm{M}^2\norm{X_k}(\norm{B^k}+\norm{B^k}^2)\right]s(B^k)^2+(\sqrt{2}-1)\alpha\right)\tau&<1
		%			\end{array}
	%			$$
	%		$$
	%		\begin{array}{ll}
		%		\tau<&1\bigg/\left(\left[\frac{2c\sin\frac{\theta_0}{2} }{\delta_0}\norm{H}^2\norm{M}^2\norm{T_k}^2\norm{B^k}^2+\frac{\sqrt{c}}{\delta_0}\norm{M}^2\norm{X_k}(1+2\norm{B^k}+2\norm{B^k}^2)\right.\right.\\
		%		&\left.\left.+2\max\left(\frac{\sqrt{c}}{\delta_0},\frac{\sqrt{c}}{\cos2\theta_0}\right)\norm{M}^2\norm{X_k}(\norm{B^k}+\norm{B^k}^2)\right]s(B^k)^2+(\sqrt{2}-1)\alpha\right)
		%		\end{array}
	%		$$
	where $c=c(\theta_0,\delta_0)\coloneqq \left(1+2\delta_0\sin\frac{3\theta_0}{2}+\delta_0^2\right)/\cos^2\frac{3\theta_0}{2}$.
	Moreover, if $\norm{B}<1$, the result is also true if $\tau>0$ and
	$$\Big(\frac{\norm{H}^2\norm{M}^2}{(1-\norm{B})^2(1-\norm{B}^k)^2}\psi_3(k,\norm{B})+\Big(\frac{\sqrt{c}}{\delta_0}-1\Big)\alpha\Big)\tau<1$$
	where 
	\begin{equation*}
			\psi_3(k,b)\coloneqq\displaystyle\frac{2c\sin\frac{\theta_0}{2}}{\delta_0}b^{2k}
			+\Big(\frac{\sqrt{c}}{\delta_0}(1+b^{2k})+2\max\Big(\frac{\sqrt{c}}{\delta_0},\frac{\sqrt{c}}{\cos2\theta_0}\Big)b^k\Big)
			(1-kb^{k-1}+(k-1)b^k)b^k.
	\end{equation*}
\end{lemma}

%		Moreover, if $\norm{B}<1$, we can take
%	$$\begin{array}{ll}
	%		\tau<&\frac{(1-\norm{B})^2}{\norm{H}^2\norm{M}^2}(1-\norm{B}^k)^2\left[\frac{2c\sin\frac{\theta_0}{2}}{\delta_0}\norm{B}^{2k}+\frac{\sqrt{c}}{\delta_0}(1-k\norm{B}^{k-1}+(k-1)\norm{B}^k)(1+\norm{B}^{2k})\right.\\
	%		&\left.
	%		+2\max\left(\frac{\sqrt{c}}{\delta_0},\frac{\sqrt{c}}{\cos2\theta_0}\right)(1-k\norm{B}^{k-1}+(k-1)\norm{B}^k)\norm{B}^k\right]^{-1}.\\
	%	\end{array}
%	$$

\begin{proof}
	Define $$\gamma_3=\gamma_3(\mathrm{sign}(\theta))\coloneqq\left\{\begin{array}{cc}
		\left(\delta_0+\sin\frac{3\theta_0}{2}\right)/\cos\frac{3\theta_0}{2} & \text{if }\theta>0,\\
		-\left(\delta_0+\sin\frac{3\theta_0}{2}\right)/\cos\frac{3\theta_0}{2} & \text{if }\theta<0
	\end{array}\right.$$ 
	as in Lemma \ref{lem:gamma123} (iii).
	Writing \eqref{eq:seim-k-shot:re+gamma*im} for $\gamma=\gamma_3$ we obtain
	\begin{multline}\label{eq:seim-k-shot:case3}
		(1+\tau\alpha)[\Re(\lambda^2-\lambda)+\gamma_3\Im(\lambda^2-\lambda)]+\tau\alpha[\Re(\lambda-1)+\gamma_3\Im(\lambda-1)]
		+\tau\alpha
		\\
		+\tau G_k(P^*_k+\gamma_3 Q^*_k,P_k+\gamma_3 Q_k)-(1+\gamma_3^2)\tau G_k(Q^*_k,Q_k)
		\\
		+\tau\left([\Re(\lambda-1)+\gamma_3\Im(\lambda-1)]L_{1,k}+[\gamma_3\Re(\lambda-1)-\Im(\lambda-1)]L_{2,k}\right)  =0.
	\end{multline}
	Since $G(P^*+\gamma_3 Q^*,P+\gamma_3 Q)\ge 0$ and $\tau\alpha\ge 0$, the left-hand side of \eqref{eq:seim-k-shot:case3} is positive if $\tau$ satisfies
	\begin{multline*}
		(1+\tau\alpha)[\Re(\lambda^2-\lambda)+\gamma_3\Im(\lambda^2-\lambda)]-\tau\alpha|\Re(\lambda-1)+\gamma_3\Im(\lambda-1)|-(1+\gamma_3^2)\tau G_k(Q^*_k,Q_k)
		\\
		-\tau|\Re(\lambda-1)+\gamma_3\Im(\lambda-1)||L_{1,k}|+|\gamma_3\Re(\lambda-1)-\Im(\lambda-1)||L_{2,k}|>0,
	\end{multline*}
	%			$$
	%			\begin{array}{rl}
		%				(1+\tau\alpha)[\Re(\lambda^2-\lambda)+\gamma_3\Im(\lambda^2-\lambda)]-\tau\alpha|\Re(\lambda-1)+\gamma_3\Im(\lambda-1)|-(1+\gamma_3^2)\tau G_k(Q^*_k,Q_k)&\\
		%				-\tau|\Re(\lambda-1)+\gamma_3\Im(\lambda-1)||L_{1,k}|+|\gamma_3\Re(\lambda-1)-\Im(\lambda-1)||L_{2,k}| &>0,
		%			\end{array}
	%			$$
	or equivalently, 
	\begin{multline}
		\label{ineq:seim-k-shot:case3}
		1+\tau\alpha-\tau\alpha\frac{|\Re(\lambda-1)+\gamma_3\Im(\lambda-1)|}{\Re(\lambda^2-\lambda)+\gamma_3\Im(\lambda^2-\lambda)}
		-\tau(1+\gamma_3^2) \frac{G_k(Q^*_k,Q_k)}{\Re(\lambda^2-\lambda)+\gamma_3\Im(\lambda^2-\lambda)}
		\\
		-\tau\frac{|\Re(\lambda-1)+\gamma_3\Im(\lambda-1)|}{\Re(\lambda^2-\lambda)+\gamma_3\Im(\lambda^2-\lambda)}|L_{1,k}|
		-\tau\frac{|\gamma_3\Re(\lambda-1)-\Im(\lambda-1)|}{\Re(\lambda^2-\lambda)+\gamma_3\Im(\lambda^2-\lambda)}|L_{2,k}| >0.
	\end{multline}
	Notice that the choice of $\gamma_3$ ensures $\Re(\lambda^2-\lambda)+\gamma_3\Im(\lambda^2-\lambda)>0$, also 
	$$
	1+\gamma_3^2=1+\frac{\left(\delta_0+\sin\frac{3\theta_0}{2}\right)^2}{\cos^2\frac{3\theta_0}{2}}=\frac{1+2\delta_0\sin\frac{3\theta_0}{2}+\delta_0^2}{\cos^2\frac{3\theta_0}{2}}=:c
	$$ 
	%			$$\begin{array}{c}
		%				1+\gamma_3^2=1+\left(\delta_0+\sin\frac{3\theta_0}{2}\right)^2/\cos^2\frac{3\theta_0}{2}=\left(1+2\delta_0\sin\frac{3\theta_0}{2}+\delta_0^2\right)/\cos^2\frac{3\theta_0}{2}=:c
		%			\end{array}$$ 
	is a constant greater than $\delta_0^2$. 
	In the following we derive upper bounds independent of $\lambda$ for the terms appearing with the negative sign in \eqref{ineq:seim-k-shot:case3}.
	By Lemma \ref{lem:gamma123} (iii) we have
	$$\frac{|\Re(\lambda-1)+\gamma_3\Im(\lambda-1)|}{\Re(\lambda^2-\lambda)+\gamma_3\Im(\lambda^2-\lambda)}
	\le\frac{\sqrt{1+\gamma_3^2}}{\delta_0}=\frac{\sqrt{c}}{\delta_0}
	$$
	and
	$$\frac{|\gamma_3\Re(\lambda-1)-\Im(\lambda-1)|}{\Re(\lambda^2-\lambda)+\gamma_3\Im(\lambda^2-\lambda)}\le\max\left(\frac{\sqrt{1+\gamma_3^2}}{\delta_0},\frac{\sqrt{1+\gamma_3^2}}{\cos2\theta_0}\right)
	=
	\max\left(\frac{\sqrt{c}}{\delta_0},\frac{\sqrt{c}}{\cos2\theta_0}\right).$$
	Using again Lemma \ref{lem:gamma123} (iii) and \eqref{ppty:Gk}, we have
	\begin{multline*}
		\frac{G_k(Q^*_k,Q_k)}{\Re(\lambda^2-\lambda)+\gamma_3\Im(\lambda^2-\lambda)}
		\le\frac{(\norm{H}\norm{M}|\norm{T_k}\sin\theta|q_2)^2}{2\delta_0|\sin(\theta/2)|}
		=\frac{2}{\delta_0}\left|\sin\frac{\theta}{2}\right|\cos^2\frac{\theta}{2}\norm{H}^2\norm{M}^2\norm{T_k}^2q_2^2
		\\
		\le\frac{2}{\delta_0}\norm{H}^2\norm{M}^2\norm{T_k}^2q_2^2.
	\end{multline*}
	Recall from \eqref{ppty:L1k} and \eqref{ppty:L2k} that
	$$|L_{1,k}|\le\norm{X_k}\norm{M}^2(p^2+q_1^2),\quad |L_{2,k}|\le2\norm{X_k}\norm{M}^2pq_1.$$
	Inserting these previous inequalities in \eqref{ineq:seim-k-shot:case3} gives the desired result thanks to expressions \eqref{eq:pBk}, \eqref{eq:qBk} and \eqref{eq:pqBk:rho(B)<1} of respectively $p,q_1$ and $q_2$, and thanks to Lemma \ref{lem:sTXk}.
\end{proof}

%$$\begin{array}{ll}
	%	\tau<&\cfrac{1}{\norm{y}^2}\left[(1+\gamma_3^2)\cfrac{G(Q^*,Q)}{\Re(\lambda^2-\lambda)+\gamma_3\Im(\lambda^2-\lambda)}\right.\\
	%	&\left.+|L_1|\cfrac{|\Re(\lambda-1)+\gamma_3\Im(\lambda-1)|}{\Re(\lambda^2-\lambda)+\gamma_3\Im(\lambda^2-\lambda)}
	%	+|L_2|\cfrac{|\gamma_3\Re(\lambda-1)-\Im(\lambda-1)|}{\Re(\lambda^2-\lambda)+\gamma_3\Im(\lambda^2-\lambda)}\right]^{-1}.
	%\end{array}
	%$$
	%By estimating
	%\begin{itemize}
	%	\item $G(Q^*,Q)\le(\norm{H}\norm{T_k}\norm{M}q_2|\sin\theta|)^2 \norm{y}^2$
	%	\item $|L_1|\le \norm{X_k}\norm{M}^2(p^2+q_1^2)\norm{y}^2$;
	%	\item $|L_2|\le  2\norm{X_k}\norm{M}^2pq_1\norm{y}^2$
	%\end{itemize}
	%and using Lemma \ref{lem:gamma123} (iii), it suffices to choose
	%$$\begin{array}{ll}
		%	\left[(1+\gamma_3^2)\left(\norm{H}\norm{T_k}\norm{M}q_2\right)^2\frac{2|\sin\frac{\theta}{2}|\cos^2\frac{\theta}{2}}{\delta_0}+\norm{X_k}\norm{M}^2(p^2+q_1^2)\frac{\sqrt{1+\gamma_3^2}}{\delta_0}\right.&\\
		%	\left.
		%	+2\norm{X_k}\norm{M}^2pq_1\max\left(\frac{\sqrt{1+\gamma_3^2}}{\delta_0},\frac{\sqrt{1+\gamma_3^2}}{\cos2\theta_0}\right)\right]\tau&<1.
		%\end{array}
		%$$	
		%Noting that $c=1+\gamma_3^2$, the final result is obtained by definitions \eqref{eq:pBk}, \eqref{eq:q1Bk}, \eqref{eq:q2Bk} of $p,q_1,q_2$. 
		%Finally, the conclusion in the case $0<\norm{B}<1$ can be obtained by Lemma \ref{lem:sTXk}.	
		
		\subsection{Final result ($k\ge 1$)}
		Considering
		Proposition~\ref{prop:seim-k-shot:tau:real-lam} (for real eigenvalues) and taking the bound in Proposition~\ref{prop:seim-k-shot:tau:complex-lam} (for complex eigenvalues), we finally obtain a sufficient condition on the descent step $\tau$ to ensure convergence of the multi-step one-shot method. 
		
		\begin{theorem}[Convergence of semi-implicit $k$-step one-shot, $k\ge 1$]
			\label{th:seim-k-shot:tau:all}
			Under assumption \eqref{hypo}, the $k$-step one-shot method \eqref{alg:seim-k-shot}, $k\ge 1$, converges for sufficiently small $\tau$. In particular, for $\norm{B}<1$, there exists an explicit piecewise (at most a $(4k)$th order) polynomial function $\mathcal{P}_k$ and a pure constant $C>0$ independent of $k$ such that $\mathcal{P}_k>0$ on $[0,1)$ and it is enough to take $\tau>0$ and
			$$\left(\frac{\norm{H}^2\norm{M}^2}{(1-\norm{B})^2(1-\norm{B}^k)^2}\mathcal{P}_k(\norm{B})+C\alpha\right)\tau<1.$$
		\end{theorem}
		
		\noindent We emphasize that the bound of $\tau$ in Theorem \ref{th:seim-k-shot:tau:all} depends only on $\norm{B},\norm{M},\norm{H}$, the number of inner iterations $k$ and the regularization parameter $\alpha$. This bound in fact does not depend on the dimensions of $\sigma$, $u$ and $g$.
		Also, Theorem \ref{th:seim-k-shot:tau:all} includes the case $B=0$, but the bound for $\tau$ in this case is quite far from optimal. Note that the optimal bound for $\tau$ in the case $B=0$ can be found in Proposition \ref{prop:seim-1-shot:tau:B=0} (for $k=1$) and Remark \ref{rk:B=0,k>2} (for $k\ge2$).

		\section{Numerical experiments on a toy problem}
		\label{sec:num-exp}
		
		In order to compare the performance of classical gradient descent algorithm \eqref{alg:seimgd} with the $k$-step one-shot algorithms \eqref{alg:seim-k-shot}, we propose the following toy model related to the inverse conductivity problem in a cavity. 
		Given $\Omega\subset\R^2$ an open bounded regular domain, we consider a system of $M$ direct problems, each of which is the Helmholtz equation for the linearized scattered field $u \in\Hso^1(\Omega)$ given by
		\begin{equation}
			\label{u-toymod}
			\left\{
			\begin{array}{ll}
				\dive(\sigma_0\nabla u)+\omega^2u=\dive(\sigma\nabla u_0), &\text{in }\Omega, \\
				u=0, &\text{on }\partial\Omega,
			\end{array}
			\right.
		\end{equation}
		where the incident field $u_0\in\Hso^{1}(\Omega)$ satisfies
		\begin{equation}
			\label{u0-toymod}
			\left\{
			\begin{array}{ll}
				\dive(\sigma_0\nabla u_0)+\omega^2u_0=0, &\text{in }\Omega, \\
				u_0=f, &\text{on }\partial\Omega,
			\end{array}
			\right.
		\end{equation}
		with the boundary data $f=f_i\in\Hso^{1/2}(\partial\Omega)$, $1\le i\le M$.
		Here $\sigma$ and $\sigma_0\in\mathrm{L}^{\infty}(\Omega)$ are supposed to be positive functions such that the support  $\overline{\Omega}_0$ of $\sigma$ is strictly included inside $\Omega$. 
		Equation \eqref{u-toymod} is obtained by linearizing the scattered field problem using the Born approximation, and $\sigma$ is the \emph{conductivity contrast} with respect to $\sigma_0$.
		The variational formulation of \eqref{u-toymod} is: find $u\in\Hso^1_0(\Omega)$ such that
		%and $u_0$ are respectively
		\begin{equation}
			\label{vf:u-toymod}
			\int_{\Omega}\sigma_0\nabla u\cdot\nabla v \dx-\int_\Omega \omega^2uv \dx=\int_{\Omega_0}\sigma\nabla u_0\cdot\nabla v \dx, \quad\forall v\in\Hso^1_0(\Omega).
		\end{equation}
		%	\begin{equation}
			%		\label{vf:u-toymod0}
			%		\int_{\Omega}\sigma_0\nabla u_{0}\cdot\nabla v-\int_\Omega \omega^2uv=0, \quad\forall v\in\Hso^1_0(\Omega)\quad\text{and }u_0=f\text{ on }\partial\Omega.
			%	\end{equation}
		We consider the inverse problem of retrieving $\sigma$ from measurements $g\coloneqq \sigma_0\frac{\pa u}{\pa\nu}\big|_{\partial\Omega}$ (we indeed collect $M$ different measurements).
		%	To solve this inverse problem we use the method of least squares. Denoting by $\sigma^\mathrm{ex}$ the exact $\sigma$ and $g=\sigma_0\frac{\pa u(\sigma^\mathrm{ex})}{\pa\nu}\big|_{\partial\Omega}$ the corresponding measurements, we consider the cost functional $$J(\sigma)=\frac{1}{2}\norm{Hu(\sigma)-g}^2_{\mathrm{L}^2(\partial\Omega)}+\frac{\alpha}{2}\norm{\sigma}^2_{\mathrm{L}^2(\Omega_0)}.$$
		%	The Lagrangian technique allows us to identify the $\mathrm{L}^2$ gradient  of $J$ as 
		%	$$\nabla_\sigma J(\sigma)=-\nabla u_0\cdot\nabla p(\sigma)+\alpha\sigma
		%	,$$ where the adjoint state $p=p(\sigma)\in\Hso^1(\Omega)$ satisfies
		%	\begin{equation}
			%	\label{vf-p}
			%	\int_{\Omega}\sigma_0\nabla p\cdot\nabla v\dx-\int_\Omega\omega^2pv\dx=0, \quad\forall v\in\Hso^1_0(\Omega)\quad\text{and }p=\left(\sigma_0\frac{\pa u(\sigma)}{\pa\nu}\bigg|_{\partial\Omega}-g\right)\text{ on }\partial\Omega.
			%	\end{equation}
		
		\noindent By discretizing $u$ using $\mathbb{P}^1$-Lagrange finite elements on a mesh $\mathcal{T}_h(\Omega)$ of $\Omega$, and $\sigma$ by $\mathbb{P}^0$-Lagrange finite elements on a coarser mesh $\widetilde{\mathcal{T}}_{h'}(\Omega_0)$ of $\Omega_0$, the discretization of \eqref{vf:u-toymod} leads to a linear system of the form 
		\begin{equation}
			\label{eq:discret-toymod}
			A_1\vec u=A_2\vec\sigma,
		\end{equation}
		where $\vec u\in\R^{n_u}$, $\vec \sigma\in\R^{n_\sigma}$ with $n_u$ denoting the number of inner nodes of $\mathcal{T}_h(\Omega)$ and $n_\sigma$ denoting the number of triangles in $\widetilde{\mathcal{T}}_{h'}(\Omega_0)$. 
		In order to rewrite the linear system in the form \eqref{prb-direct} with a controllable norm for the matrix $B$, we choose to parameterize $\sigma_0$ as
		$\sigma_0=\tilde{\sigma}_0+\delta\sigmar$ where $\delta>0$ is a small parameter and $\sigmar\le1$ is a random function. With this choice of $\sigma_0$, we can write the matrix $A_1$ as 
		$A_1=A_{11}+\delta A_{12}$
		where $A_{11}$ corresponds to the discretization of the bilinear form $\int_{\Omega}(\tilde{\sigma}_0\nabla u\cdot\nabla v-\omega^2uv)\dx$ on $\Hso^1_0(\Omega)\times \Hso^1_0(\Omega)$ and $A_{12}$ corresponds to the discretization of the bilinear form $\int_{\Omega}\sigmar\nabla u\cdot\nabla v\dx$  on $\Hso^1_0(\Omega)\times \Hso^1_0(\Omega)$. For $\omega^2$ not a Dirichlet eigenvalue of $-\dive(\tilde{\sigma}_0\nabla u)$ in $\Omega$, the matrix $A_{11}$ is invertible and we can equivalently write \eqref{eq:discret-toymod} as
		%	 
		%	To rewrite the system in the form \eqref{prb-direct}, we consider the naive splitting $A_1=A_{11}+\delta A_{12}$, where $A_{11}$ and $A_{12}$ are respectively issued from the discretization of $\int_{\Omega}\tilde{\sigma}_0\nabla u\cdot\nabla v-\int_\Omega \omega^2uv$ and $\int_{\Omega}\sigmar\nabla u\cdot\nabla v$. 
		\[
		\vec{u}=A_{11}^{-1}(-\delta A_{12}\vec{u}+A_2\vec{\sigma})
		\]
		%\quad\text{ and } \vec{u}=0\text{ on }\partial\Omega.
		which is in the form \eqref{prb-direct} with $B=-\delta A_{11}^{-1}A_{12}$, $M=A_{11}^{-1}A_2$ and $F=0$.
		%	The discretized adjoint state $p$
		%	$$\vec{p}=A_{11}^{-1}\left(-\delta A_{12}\vec{p}\right)\quad\text{ and } \vec{p}=H\vec{u}-\vec{g}\text{ on }\partial\Omega$$ 
		%	where $H\in\R^{n_g\times n_u}$ is the discretization of the above operator $H$ by abuse of notation. Choosing $\delta$ such that $\delta\norm{A_{11}^{-1}A_{12}}_2<1$,  we consider \eqref{direct-and-inv-prb} with $B=-\delta A_{11}^{-1}A_{12}$, $M=A_{11}^{-1}A_2$, $F=0$. The application of $A_{11}^{-1}$, which has the same size as matrix $A_1$, is done by a direct solver; more practical fixed point iterations will be investigated in the future.  
		\noindent Indeed $\norm{B}<1$ for sufficiently small $\delta$. We employed the finite element library FreeFEM 	 \cite{FreeFEM} to generate the matrices $A_{11}, A_{12}, A_2$ and the measurement operator $H\in\R^{n_g\times n_u}$, which is the discretization of the normal trace operator $\sigma_0\frac{\pa u}{\pa\nu}\big|_{\partial\Omega}$, where $n_g$	denotes the number of nodes on $\partial\Omega$. 
		%FreeFEM v4.12 230404
		
		For the numerical tests below, we set $\omega=2\pi$, $\tilde{\sigma}_0=1$, $ \delta=0.01$ and the mesh size $h={\lambda}/{20}=0.05$ where $\lambda=\sqrt{\tilde{\sigma}_0}{2\pi}/{\omega}=1$. The domain $\Omega$ is the disk of radius $R=2\lambda$ and $\Omega_0$ is formed by three squares as in Figures \ref{fig:domain} and \ref{fig:meshnoise}. 
		To generate measurements, we use $M=6$ resulting in six different incident fields $u_0$ corresponding to imposing boundary data $f=f_i$, $1\le i\le 6$, where $f_i(x)= Y_0(\omega|x-y_i|)$ with $y_i$ located (outside $\Omega$) on the circle of radius $R+0.25\lambda$ (see Figure \ref{fig:domain}). The function $Y_0$ is the Bessel function of the second kind of order zero.
		The cost functional is then the sum of the cost functionals associated with each of the incident fields.
		We take as exact solution $\sigma^\mathrm{ex}=10$ in each square, and as initial guess $\sigma^0=12$ in each square.

		\subsection{The case of noise-free data}
		\label{sec:exp1}
		
		%	\begin{figure}[htbp]
			%	\centering
			%	\begin{subfigure}{0.33\textwidth}
				%	\includegraphics[width=\linewidth]{Thu_v1-eps-converted-to.pdf}
				%	\caption{Mesh for $u$.}
				%	\label{fig:Thu_v1}
				%	\end{subfigure}
			%%\hfil
			%	\begin{subfigure}{0.25\textwidth}
				%	\includegraphics[width=\linewidth]{Ths_v1-eps-converted-to.pdf}
				%	\caption{Mesh for $\sigma$.}
				%	\label{fig:Ths_v1}
				%	\end{subfigure}
			%%\hfil
			%	\begin{subfigure}{0.33\textwidth}
				%		\centering
				%		\resizebox{0.7\textwidth}{!}{%
					%	  \begin{tikzpicture}
						%	  	\def \n {125} 
						%	  	\def \radius {2cm}
						%	  	\node[circle,draw=gray,fill=white,minimum size=\radius] (a) at (0,0) {};
						%	  	%\filldraw [gray] (a.center) circle [radius=0.01cm];
						%	  	\foreach \X in {1,...,\n}
						%	  	{\def \s {-360/\n*\X};
							%	  		\draw[black] (a.\s) -- ++ (\s:0.04);}
						%	  \end{tikzpicture}
					%		}%
				%	\caption{Mesh for $g$.}
				%	\label{fig:Thumes_v1}
				%	\end{subfigure}
			%	\caption{Meshes used to generate the matrices $A_1$, $A_2$ and $H$ in the case of noise-free data.}
			%	\label{fig:domain}
			%	\end{figure}
		
		\begin{figure}
			\centering
			\subfloat[][Mesh for $u$ ($n_u=5849$). \label{fig:Thu_v1}]
			{\includegraphics[width=.32\textwidth]{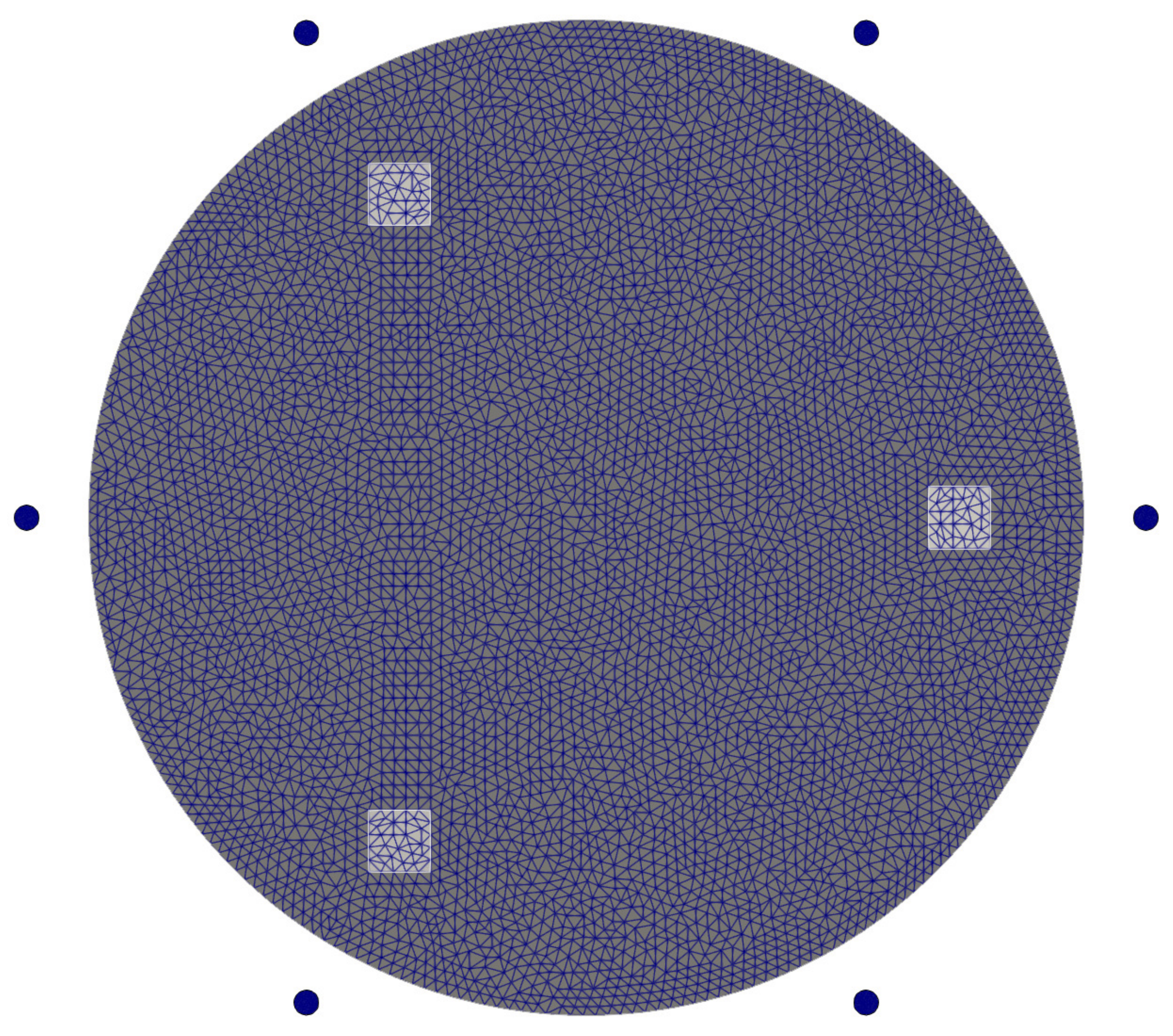}} 
			\hfil
			\subfloat[][Mesh for $\sigma$ ($n_\sigma=6$). \label{fig:Ths_v1}]
			{\includegraphics[width=.26\textwidth]{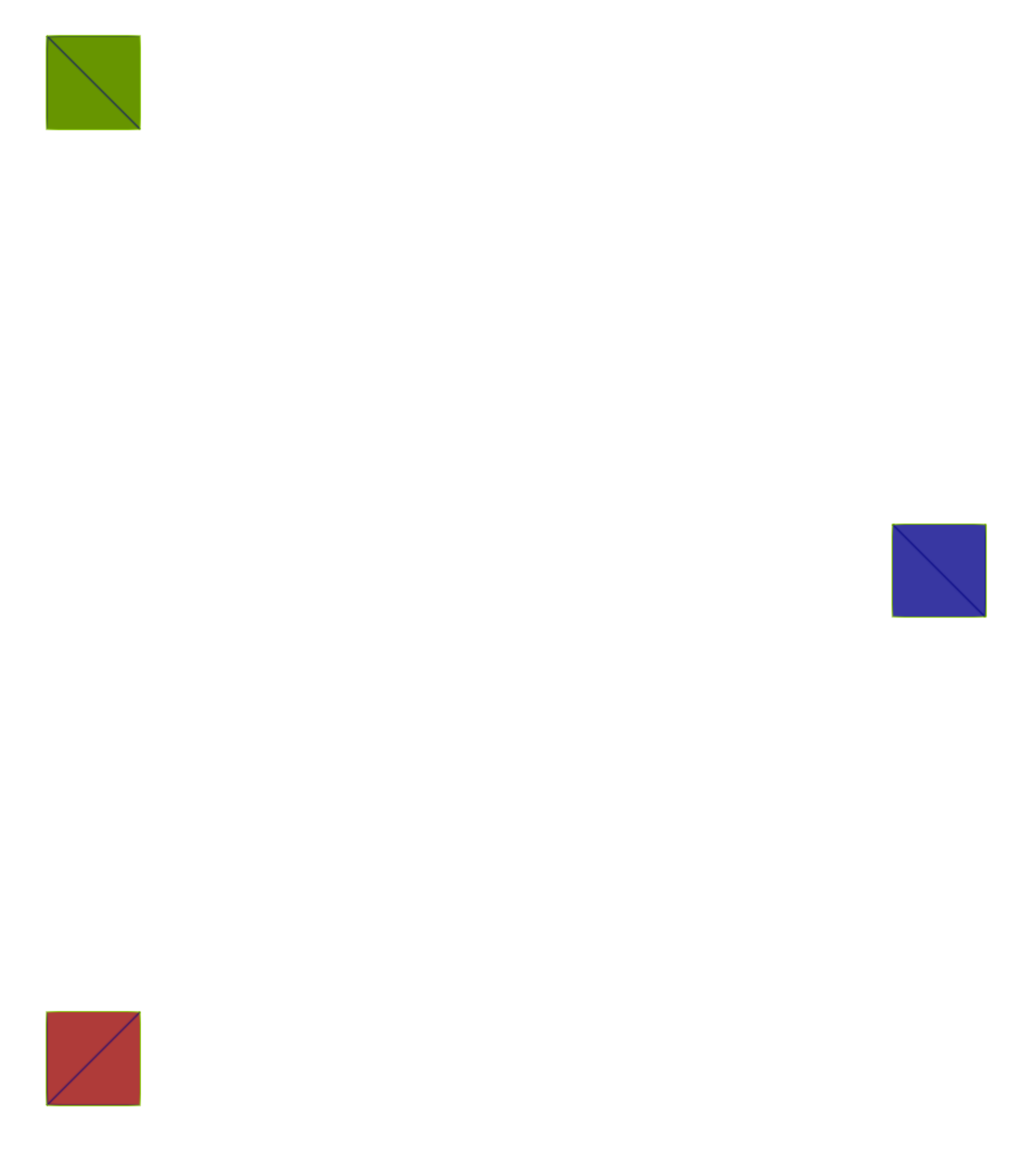}} 
			\hfil
			\subfloat[][Mesh for $g$ ($n_g=125$). \label{fig:Thumes_v1}]
			{\begin{tikzpicture}
					\def \n {125} 
					\def \radius {4cm}
					\node[circle,draw=gray,fill=white,minimum size=\radius] (a) at (0,0) {};
					%\filldraw [gray] (a.center) circle [radius=0.01cm];
					\foreach \X in {1,...,\n}
					{\def \s {-360/\n*\X};
						\draw[black] (a.\s) -- ++ (\s:0.04);}
			\end{tikzpicture}} 
			\caption{The configuration for the experiment in Section \ref{sec:exp1}. The small circles outside the mesh for $u$ indicate the source locations $y_i$, $1\le i\le 6$.}
			\label{fig:domain}
		\end{figure}
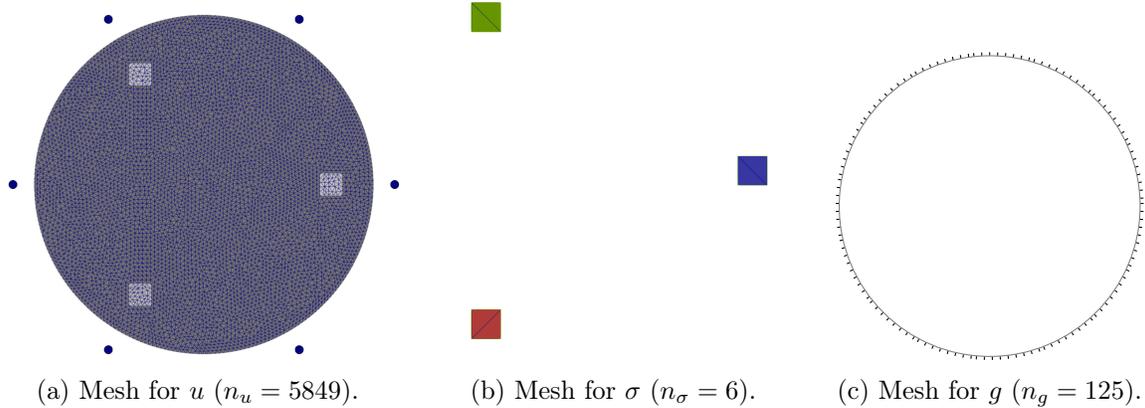 
		
		\noindent We first consider the case of noise-free data, without regularization ($\alpha=0$). 
		Here, the domain $\Omega_0$ is formed by three squares of size ${\lambda}/{4}$ distributed as shown in Figure \ref{fig:domain}. We set the mesh size $h'={\lambda}/{4}$ so that each square is divided into two triangles (see Figure \ref{fig:Ths_v1}), which gives $n_\sigma=6$. 
		The mesh used for $\Omega$ (see Figure \ref{fig:Thu_v1}) leads to $n_u=5849$. The boundary mesh used for generating the data $g$ (see Figure \ref{fig:Thumes_v1}) is two times coarser than the mesh for $u$ and this gives $n_g=125$.
		We compare the performances of $k$-step one-shot methods \eqref{alg:seim-k-shot} (which coincide with \eqref{alg:k-shot} in the present case $\alpha=0$).
		Recall that $k$ is the number of inner iterations on the direct and adjoint problems. 
		We consider two series of experiments. 
		
		%	 \begin{figure}[htb]
			%	 	\centering
			%	 	\begin{subfigure}{0.45\linewidth}
				%	 		\centering
				%	 		\includegraphics[width=\linewidth]{noise0_Jn_seim_fix_k1a-eps-converted-to.pdf}
				%	 		\caption{Gradient descent and $1$-step one-shot.}
				%	 		\label{fig:noise=0,k=1a}
				%	 	\end{subfigure}\hfil
			%	 	\begin{subfigure}{0.45\linewidth}
				%	 		\centering
				%	 		\includegraphics[width=\linewidth]{noise0_Jn_seim_fix_k1b-eps-converted-to.pdf}
				%	 		\caption{$1$-step one-shot.}
				%	 		\label{fig:noise=0:k=1b}
				%	 	\end{subfigure}\\[1em]
			%	 	\begin{subfigure}{0.45\linewidth}
				%	 		\centering
				%	 		\includegraphics[width=\linewidth]{noise0_Jn_seim_fix_k2a-eps-converted-to.pdf}
				%	 		\caption{Gradient descent and $2$-step one-shot.}
				%	 		\label{fig:noise=0:k=2a}
				%	 	\end{subfigure}\hfil
			%	 	\begin{subfigure}{0.45\linewidth}
				%	 		\centering
				%	 		\includegraphics[width=\linewidth]{noise0_Jn_seim_fix_k2b-eps-converted-to.pdf}
				%	 		\caption{$2$-step one-shot.}
				%	 		\label{fig:noise=0:k=2b}
				%	 	\end{subfigure}
			%	 	\caption{Convergence curves of gradient descent and $k$-step one-shot: dependence on the descent step $\tau$.}
			%	 	\label{fig:noise=0:k=1}
			%	 \end{figure}
		
		\begin{figure}[htb]
			\centering
			\subfloat[][Gradient descent and $1$-step one-shot. \label{fig:noise=0,k=1a}]
			{\includegraphics[width=.45\textwidth]{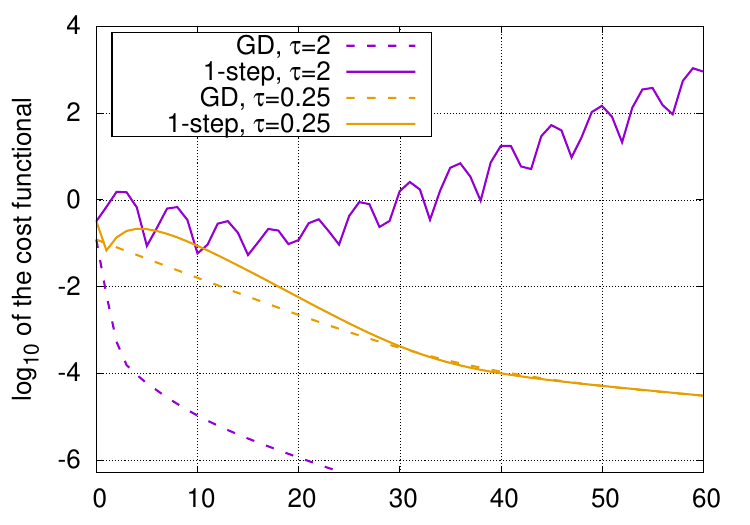}} 
			\hfil
			\subfloat[][$1$-step one-shot. \label{fig:noise=0,k=1b}]
			{\includegraphics[width=.45\textwidth]{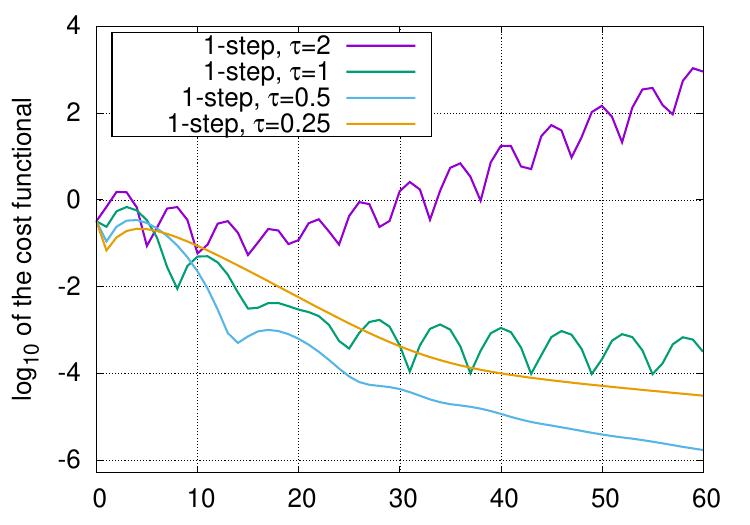}} 
			\\
			\subfloat[][Gradient descent and $2$-step one-shot. \label{fig:noise=0,k=2a}]
			{\includegraphics[width=.45\textwidth]{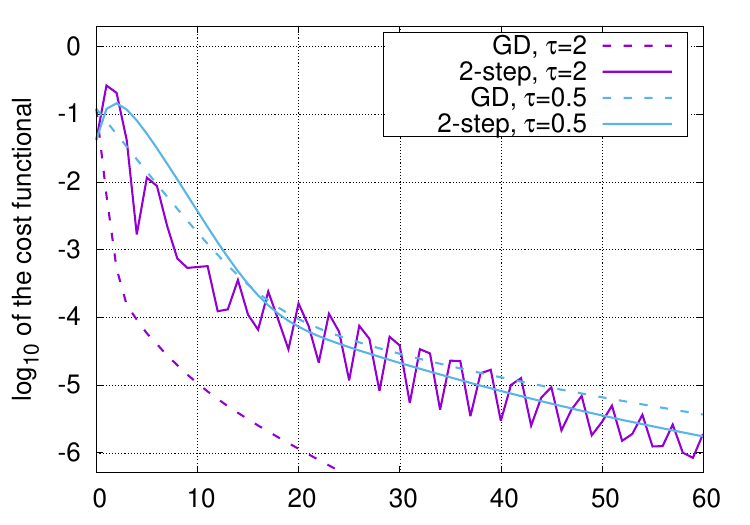}} 
			\hfil
			\subfloat[][$2$-step one-shot. 
			\label{fig:noise=0,k=2b}]
			{\includegraphics[width=.45\textwidth]{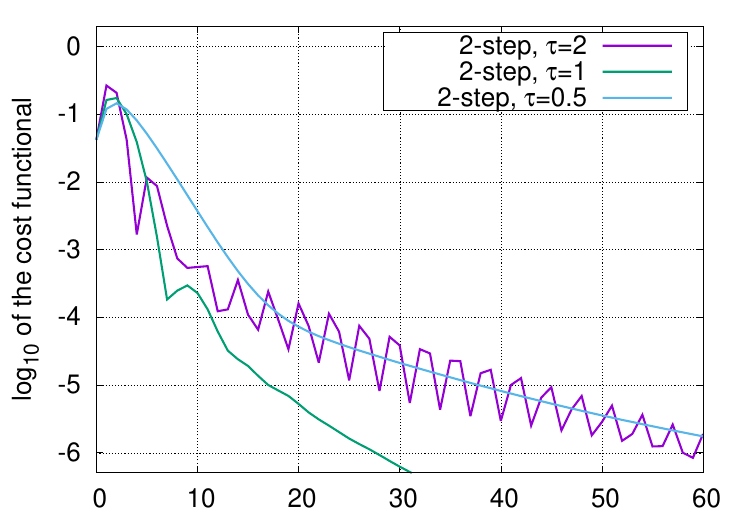}} 
			\caption{Convergence curves of gradient descent and $k$-step one-shot: dependence on the descent step $\tau$.}
			\label{fig:noise=0:k=1}
		\end{figure}
		In the first one, we study the dependence on the descent step $\tau$. In Figure \ref{fig:noise=0,k=1a}--\ref{fig:noise=0,k=1b} and \ref{fig:noise=0,k=2a}--\ref{fig:noise=0,k=2b} we respectively fix $k=1$ and $k=2$ and compare $k$-step one-shot methods with the gradient descent method. We plot in semi-log scale the value of the cost functional in terms of the (outer) iteration number $n$ in \eqref{alg:usualgd} and \eqref{alg:seim-k-shot}. We can verify that for sufficiently small $\tau$, the $k$-step one-shot methods  converge. This is not always the case for larger value of $\tau$. 
		In particular, for $\tau=2$, while gradient descent and $2$-step one-shot converge, $1$-step one-shot diverges. Oscillations may appear on the convergence curve for certain values of $\tau$, but they gradually disappear when $\tau$ gets smaller. For sufficiently small $\tau$, the convergence curves of one-shot methods are comparable to the one of gradient descent.
		
		%	  	 \begin{figure}[htb]
			%	 	\centering
			%	 	\begin{subfigure}{0.45\linewidth}
				%	 		\centering
				%	 		\includegraphics[width=\linewidth]{noise0_Jn_seim_fixtau2-eps-converted-to.pdf}
				%	 		\caption{Gradient descent and $k$-step one-shot with $\tau=2$.}
				%	 		\label{fig:noise=0,tau=2}
				%	 	\end{subfigure}\hfil
			%	 	\begin{subfigure}{0.45\linewidth}
				%	 		\centering
				%	 		\includegraphics[width=\linewidth]{noise0_acc_seim_fixtau2-eps-converted-to.pdf}
				%	 		\caption{$k$-step one-shot with $\tau=2$ in terms of the accumulated inner iteration number.}
				%	 		\label{fig:acc,noise=0,tau=2}
				%	 	\end{subfigure}\\[1ex]
			%	 	\begin{subfigure}{0.45\linewidth}
				%	 		\centering
				%	 		\includegraphics[width=\linewidth]{noise0_Jn_seim_fixtau2-5-eps-converted-to.pdf}
				%	 		\caption{Gradient descent and $k$-step one-shot with $\tau=2.5$.}
				%	 		\label{fig:noise=0,tau=2.5}
				%	 	\end{subfigure}\hfil
			%	 	\begin{subfigure}{0.45\linewidth}
				%	 		\centering
				%	 		\includegraphics[width=\linewidth]{noise0_acc_seim_fixtau2-5-eps-converted-to.pdf}
				%	 		\caption{$k$-step one-shot with $\tau=2.5$ in terms of the accumulated inner iteration number.}
				%	 		\label{fig:acc,noise=0,tau=2.5}
				%	 	\end{subfigure}
			%	 	\caption{Convergence curves of gradient descent and $k$-step one-shot: dependence on the number of inner iterations $k$.}
			%	 	\label{fig:noise=0,fix-tau}
			%	 \end{figure}
		\begin{figure}[htb]
			\centering
			\subfloat[][Gradient descent and $k$-step one-shot with $\tau=2$. 	\label{fig:noise=0,tau=2}]
			{\includegraphics[width=0.46\linewidth]{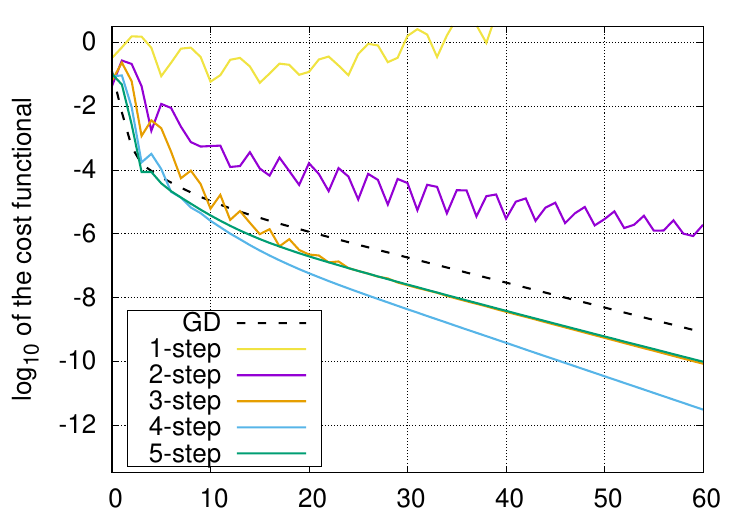}} 
			\hfil
			\subfloat[][$k$-step one-shot with $\tau=2$ in terms of the accumulated inner iteration number.\label{fig:acc,noise=0,tau=2}]
			{\includegraphics[width=0.46\linewidth]{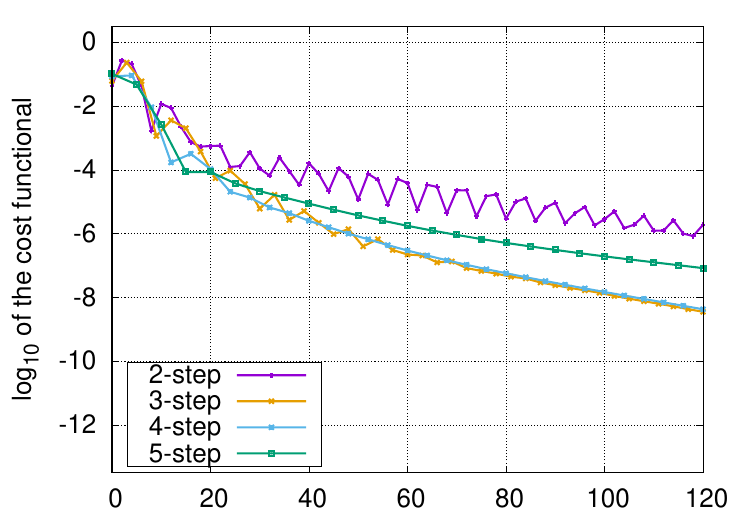}} 
			\\
			\subfloat[][Gradient descent and $k$-step one-shot with $\tau=2.5$.	\label{fig:noise=0,tau=2.5}]
			{\includegraphics[width=0.46\linewidth]{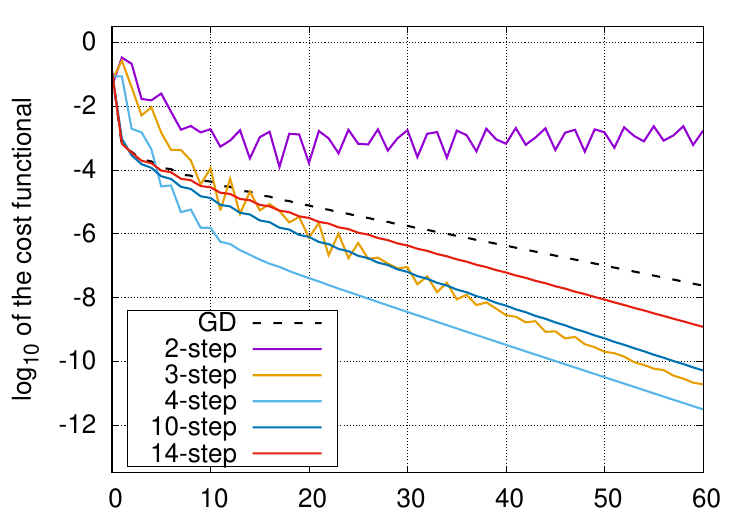}} 
			\hfil
			\subfloat[][$k$-step one-shot with $\tau=2.5$ in terms of the accumulated inner iteration number.	\label{fig:acc,noise=0,tau=2.5}]
			{\includegraphics[width=0.46\linewidth]{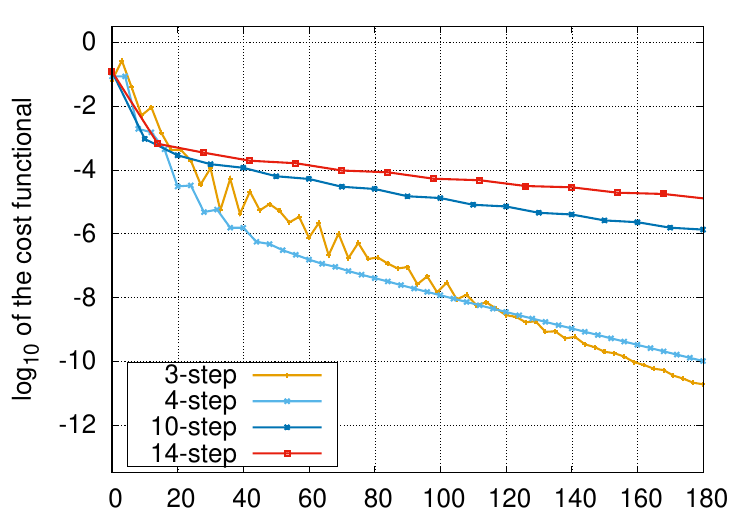}}
			\caption{Convergence curves of gradient descent and $k$-step one-shot: dependence on the number of inner iterations $k$.}
			\label{fig:noise=0,fix-tau}
		\end{figure}
		In the second series of experiments, we study the dependence on the number of inner iterations $k$, for fixed $\tau$. First (Figures~\ref{fig:noise=0,tau=2} and \ref{fig:noise=0,tau=2.5}), we investigate for which $k$ the convergence curve of $k$-step one-shot is comparable with the one of gradient descent, and, as in the previous figures, on the horizontal axis we indicate the (outer) iteration number $n$.
		For $\tau=2$ (see Figure \ref{fig:noise=0,tau=2}), we observe that for $k=3,4$ the convergence curves of $k$-step one-shot are close to the one of gradient descent. 
		Note that with $3$ inner iterations the $\mathrm{L}^2$ error between $u^n$ and the exact solution to the forward problem ranges between $59\cdot10^{-9}$ and $0.48334$ for different $n$ in \eqref{alg:seim-k-shot}; in fact, this error is rather significant at the beginning of the iteration, then it reduces as we get closer to the convergence for the parameter $\sigma$.
		Therefore, incomplete inner iterations on the forward problem are enough to have good precision on the solution of the inverse problem.
		In the particular case $\tau =2.5$ (see Figure \ref{fig:noise=0,tau=2.5}), we observe an interesting phenomenon: when $k=3,4,10$, with $k$-step one-shot the cost functional decreases even faster than with gradient descent.
		For larger values of $k$, for example $k=14$, the convergence curve of one-shot method gets closer to the one of gradient descent as one may expect. 
		Next, in Figures~\ref{fig:acc,noise=0,tau=2} and \ref{fig:acc,noise=0,tau=2.5}, we display the results of the same experiment as in Figures~\ref{fig:noise=0,tau=2} and \ref{fig:noise=0,tau=2.5}, but this time on the horizontal axis we indicate the accumulated inner iteration number, which is equal to $kn$ where $n$ is the number of outer iterations. This would allow us to compare the overall speed of convergence among $k$-step one-shot methods.
		For $\tau=2$ (respectively for $\tau=2.5$), $k=3$ and $k=4$ (respectively $k=3$) appear to provide the fastest rate of convergence. 
		This confirms the potential interest that this type of method would have in solving large-scale inverse problem since few inner iterations are capable of providing fast convergence.

			\subsection{The case of noisy data}
			\label{sec:exp2}
			
			%		\begin{figure}[htb]
				%		\centering
				%		\begin{subfigure}{0.4\linewidth}
					%			\centering
					%			\includegraphics[width=\linewidth]{Thu_v2-eps-converted-to.pdf}
					%			\caption{Mesh for $u$.}
					%			\label{fig:Thu_v2}
					%		\end{subfigure}
				%		\begin{subfigure}{0.3\linewidth}
					%			\includegraphics[width=\linewidth]{Ths_v2-eps-converted-to.pdf}
					%			\caption{Mesh for $\sigma$.}
					%			\label{fig:Ths_v2}
					%		\end{subfigure}
				%		\begin{subfigure}{0.25\linewidth}
					%			\centering
					%			\resizebox{0.9\textwidth}{!}{%
						%				\begin{tikzpicture}
							%					\def \n {12} 
							%					\def \radius {2cm}
							%					\node[circle,draw=gray,fill=white,minimum size=\radius] (a) at (0,0) {};
							%					%\filldraw [gray] (a.center) circle [radius=0.01cm];
							%					\foreach \X in {1,...,\n}
							%					{\def \s {-360/\n*\X};
								%						\draw[black] (a.\s) -- ++ (\s:0.04);}
							%				\end{tikzpicture}
						%			}%
					%			\caption{Mesh for $g^\varepsilon$.}
					%			\label{fig:Thumes_v2}
					%		\end{subfigure}
				%		\caption{Meshes used to generate the matrices $A_1$, $A_2$ and $H$ in the case of noisy data.}
				%		\label{fig:meshnoise}
				%	\end{figure}
			
			\begin{figure}
				\centering
				\subfloat[][Mesh for $u$ ($n_u=5582$). \label{fig:Thu_v2}]
				{\includegraphics[width=.32\textwidth]{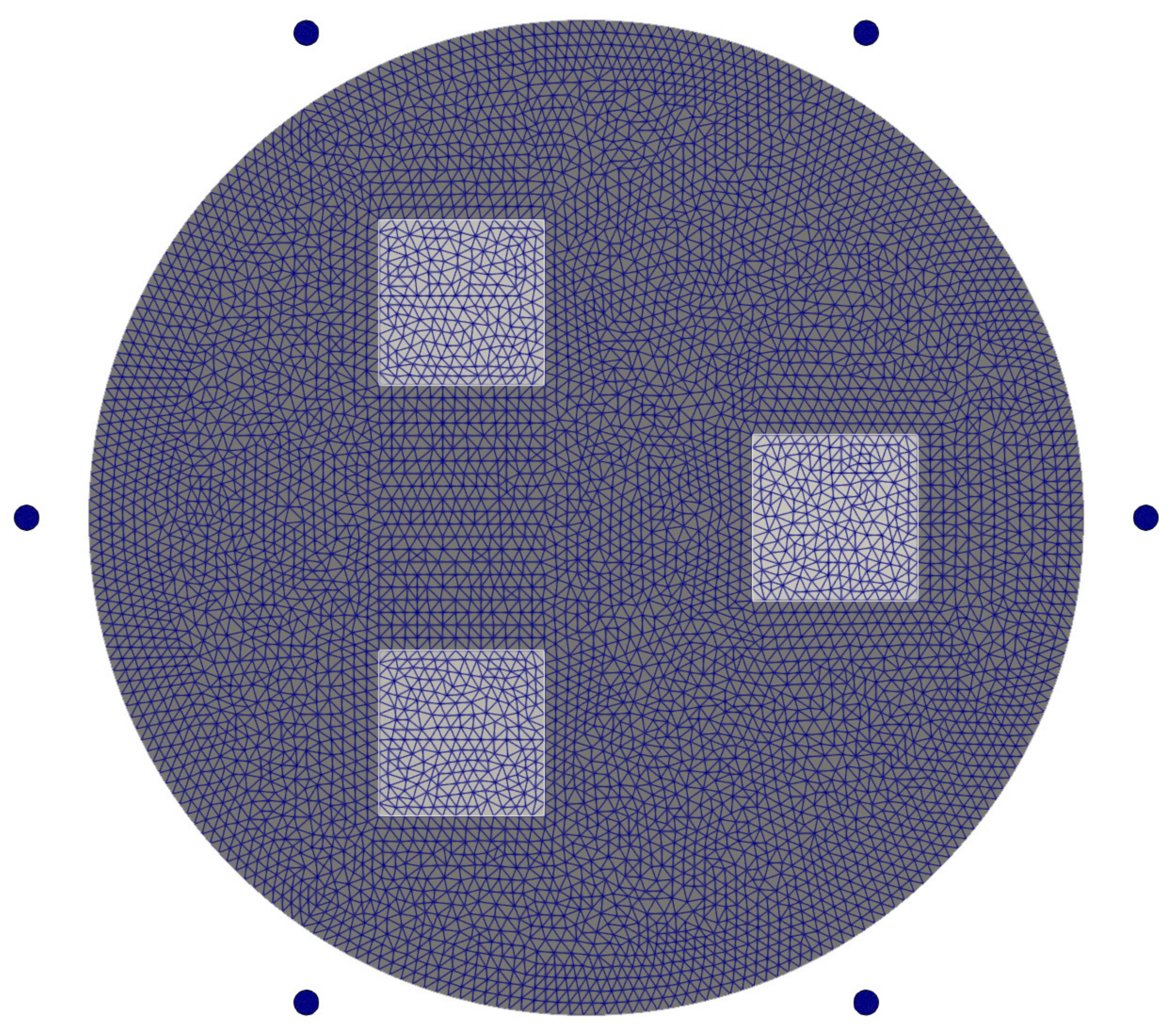}} 
				\hfil
				\subfloat[][Mesh for $\sigma$ ($n_\sigma=698$).\label{fig:Ths_v2}]
				{\includegraphics[width=.29\textwidth]{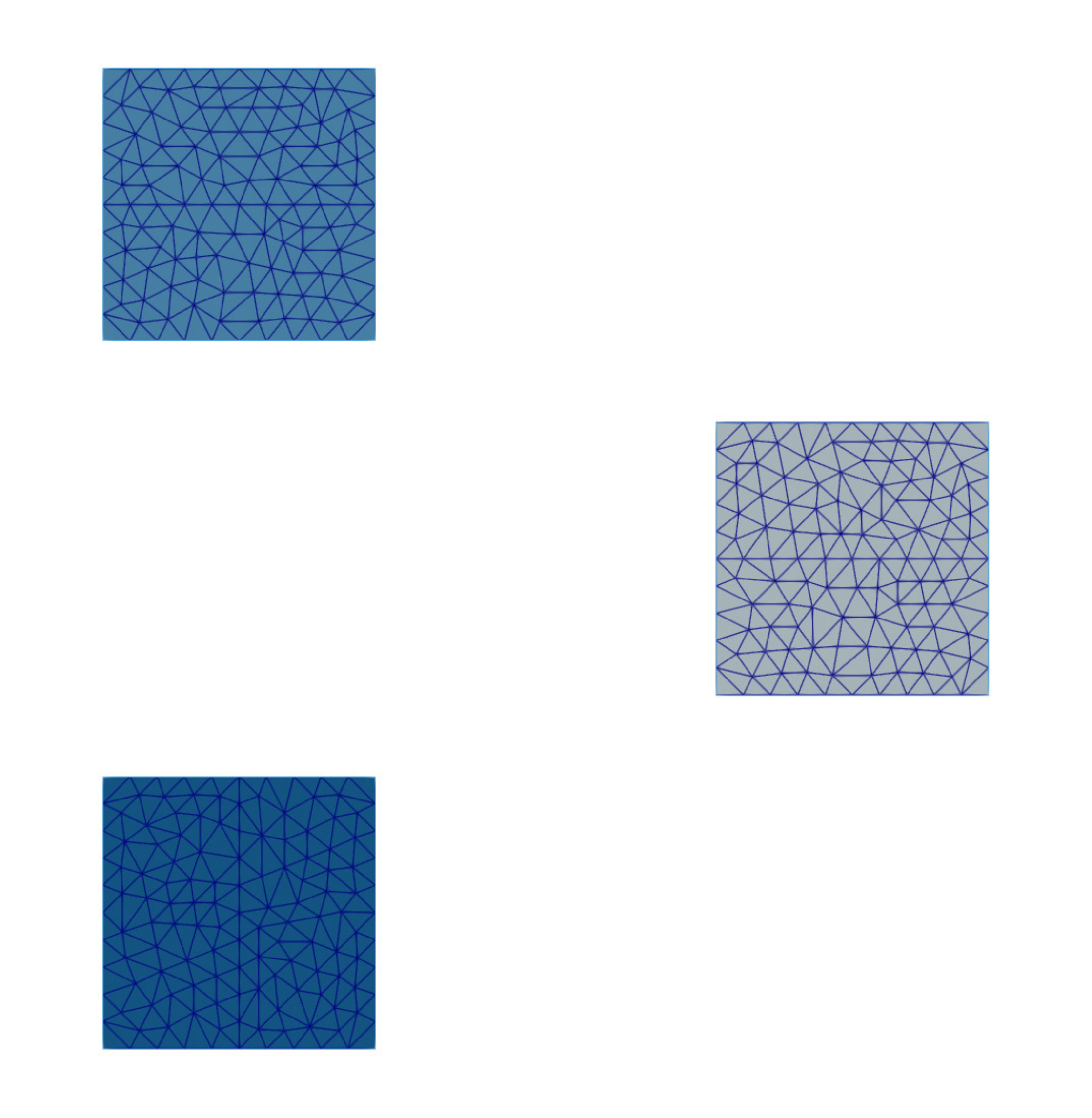}} 
				\hfil
				\subfloat[][Mesh for $g^\varepsilon$ ($n_g=12$). \label{fig:Thumes_v2}]
				{\begin{tikzpicture}
						\def \n {12} 
						\def \radius {4cm}
						\node[circle,draw=gray,fill=white,minimum size=\radius] (a) at (0,0) {};
						%\filldraw [gray] (a.center) circle [radius=0.01cm];
						\foreach \X in {1,...,\n}
						{\def \s {-360/\n*\X};
							\draw[black] (a.\s) -- ++ (\s:0.04);}
				\end{tikzpicture}} 
				\caption{The configuration for the experiment in Section \ref{sec:exp2}. The small circles outside the mesh for $u$ indicate the source locations $y_i$, $1\le i\le 6$.}
				\label{fig:meshnoise}
			\end{figure} 
			
			\noindent We consider now the case where the measurements $g$ are corrupted with noise. More specifically, we replace the vector $g=Hu(\sigma^\mathrm{ex})$ in the cost functional by the vector $g^\varepsilon\in\R^{n_g}$ with $g^\varepsilon_i\coloneqq g_i+\varepsilon_i g_i$, where the $\varepsilon_i$ are random numbers uniformly distributed between $-\varepsilon$ and $\varepsilon$ for a noise level $\varepsilon$.
			In order to hit the ill-posedness of the inverse problem, we artificially increase the size of the discretization space for the parameter $\sigma$ to $n_\sigma=698$ by enlarging the size of the squares (now the size of their edges equals ${2\lambda}/{3}$ and the distance of their center from the boundary equals $\lambda$,  see Figure \ref{fig:Ths_v2}).
			%n_u new =5582 (n_u changes since ractangles change, but the mesh size dose not change)
			We also show in Figure \ref{fig:Thu_v2} the mesh for $u$, which is about $20$ times finer than the boundary mesh used for generating the $6$ data $g^\varepsilon$ (see Figure \ref{fig:Thumes_v2}). In this new configuration, $n_u=5582$ and $n_g=12$. Notice also that the meshes of the squares in Figure \ref{fig:Thu_v2} and \ref{fig:Ths_v2} do not coincide.
			%	Note: New $n_\sigma=706$, size square changes from $2\cdot\frac{\lambda}{8}=\frac{\lambda}{4}$ to $2\cdot\frac{\lambda}{3}=\frac{2\lambda}{3}$, centered from $R-\frac{\lambda}{2}$ to $R-\lambda$ where $R=2\lambda$.
			
			We perform several numerical tests with different noise levels: $\varepsilon=1\%,3\%$ and $5\%$. We choose the regularization parameter $\alpha$ depending on the noise level $\varepsilon$ in order to optimize the accuracy of the reconstruction. 
			This choice, which is made by trial and error, does not affect much the convergence of the algorithms (see Figures \ref{fig:Jn:noise135} and \ref{fig:RelaEn:noise135}) and mainly reduces the size of the oscillations for the reconstructed $\sigma$.
			% Though we do not show the complete trial and error here, this fact can be revised from what we are about to present in Figures \ref{fig:Jn:noise135} and \ref{fig:RelaEn:noise135}.
			The convergence curves displayed in Figure \ref{fig:Jn:noise135} for different noise levels show that the semi-implicit $k$-step one-shot methods ($k=3,4$) require a similar number of outer iterations as the semi-implicit gradient descent algorithm to achieve the same precision. 
			We also see from Figure \ref{fig:Jn:noise135} that the convergence curves with $\alpha\neq 0$ look less steep than those with $\alpha=0$.
			Since in the case of noisy data the convergence for the cost functional does not imply in general the accuracy of the reconstructed parameter $\sigma$, we also plot convergence curves for the relative error on $\sigma$ in Figure \ref{fig:RelaEn:noise135} to check the quality of the reconstruction. Also in these plots we see that the semi-implicit $k$-step one-shot methods ($k=3,4$) require a similar number of outer iterations as the semi-implicit gradient descent algorithm to achieve the same accuracy. 
			Indeed, this proves the potential of these methods since only a few inner iterations are used. 
			Moreover, the chosen values of $\alpha\neq 0$ adapted to the noise level give fairly better relative error curves.
			Finally, to better show the effect of the regularization, Figures \ref{fig:sigman:seim-3-shot,noise1,alp0}--\ref{fig:sigman:seim-3-shot,noise5,alp2e-4} display the final reconstructions for $\sigma$ by semi-implicit $3$-step one-shot, whose relative error curves were presented in Figures \ref{fig:RelaEn:noise1:alp0}--\ref{fig:RelaEn:noise5:alp2e-4}. 
			On the left of the color scale we also indicate the actual maximum and minimum values attained for each reconstruction; note that this range gets wider for higher noise level. 
			These plots confirm the benefit of regularization for treating noisy data.
			
			%\begin{figure}[htbp]
			%	\centering
			%	\begin{subfigure}{0.33\columnwidth}
				%		\includegraphics[width=\linewidth]{Jn_noise1_alp0_tau4.7-eps-converted-to.pdf}
				%		\caption{$\varepsilon=1\%$, $\alpha=0$.}
				%		\label{fig:Jn:noise1:alp0}
				%	\end{subfigure}\hfill
			%	\begin{subfigure}{0.33\columnwidth}
				%		\includegraphics[width=\linewidth]{Jn_noise3_alp0_tau4.7-eps-converted-to.pdf}
				%		\caption{$\varepsilon=3\%$, $\alpha=0$.}
				%		\label{fig:Jn:noise3:alp0}
				%	\end{subfigure} \hfill
			%	\begin{subfigure}{0.33\columnwidth}
				%		\includegraphics[width=\linewidth]{Jn_noise5_alp0_tau4.7-eps-converted-to.pdf}
				%		\caption{$\varepsilon=5\%$, $\alpha=0$.}
				%		\label{fig:Jn:noise5:alp0}
				%	\end{subfigure} \hfill
			%	\begin{subfigure}{0.33\columnwidth}
				%		\includegraphics[width=\linewidth]{Jn_noise1_alp0.0001_tau4.7-eps-converted-to.pdf}
				%		\caption{$\varepsilon=1\%$, $\alpha=10^{-4}$.}
				%		\label{fig:Jn:noise1:alp1e-4}
				%	\end{subfigure}\hfill
			%	\begin{subfigure}{0.33\columnwidth}
				%		\includegraphics[width=\linewidth]{Jn_noise3_alp0.0002_tau4.7-eps-converted-to.pdf}
				%		\caption{$\varepsilon=3\%$, $\alpha=2\cdot10^{-4}$.}
				%		\label{fig:Jn:noise3:alp2e-4}
				%	\end{subfigure}\hfill
			%	\begin{subfigure}{0.33\columnwidth}
				%		\includegraphics[width=\linewidth]{Jn_noise5_alp0.0002_tau4.7-eps-converted-to.pdf}
				%		\caption{$\varepsilon=5\%$, $\alpha=2\cdot10^{-4}$.}
				%		\label{fig:Jn:noise5:alp2e-4}
				%	\end{subfigure}\hfill
			%	\caption{Convergence curves of semi-implicit gradient descent and $k$-step one-shot with different noise levels $\varepsilon$ and regularization parameters $\alpha$.}
			%	\label{fig:Jn:noise135}
			%\end{figure}
			
			\begin{figure}[htbp]
				\centering
				\subfloat[][$\varepsilon=1\%$, $\alpha=0$. 
				\label{fig:Jn:noise1:alp0}]
				{\includegraphics[width=0.33\textwidth]{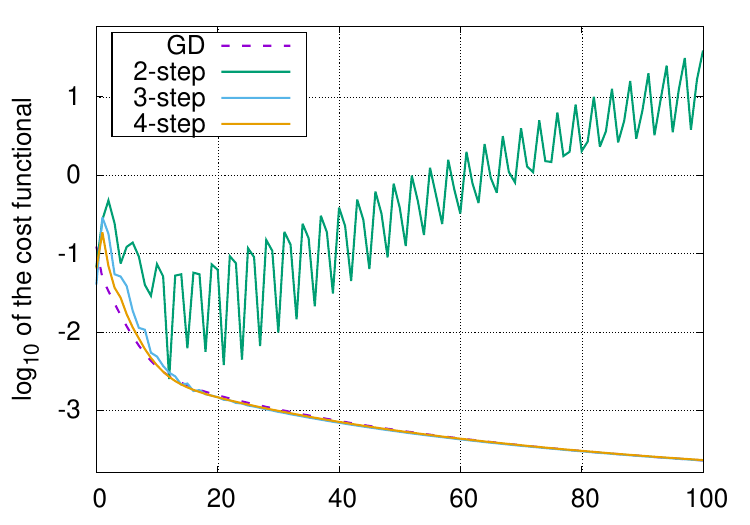}}
				\hfil
				\subfloat[][$\varepsilon=1\%$, $\alpha=0$. 
				\label{fig:Jn:noise3:alp0}]
				{\includegraphics[width=0.33\textwidth]{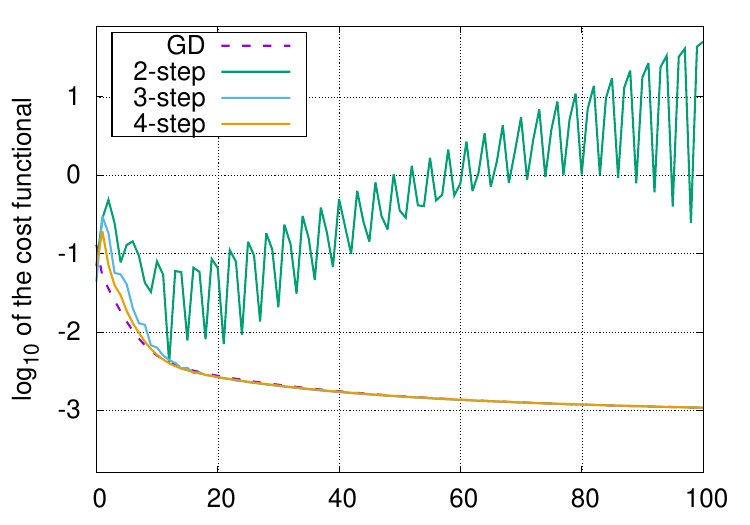}}
				\hfil
				\subfloat[][$\varepsilon=1\%$, $\alpha=0$. 
				\label{fig:Jn:noise5:alp0}]
				{\includegraphics[width=0.33\textwidth]{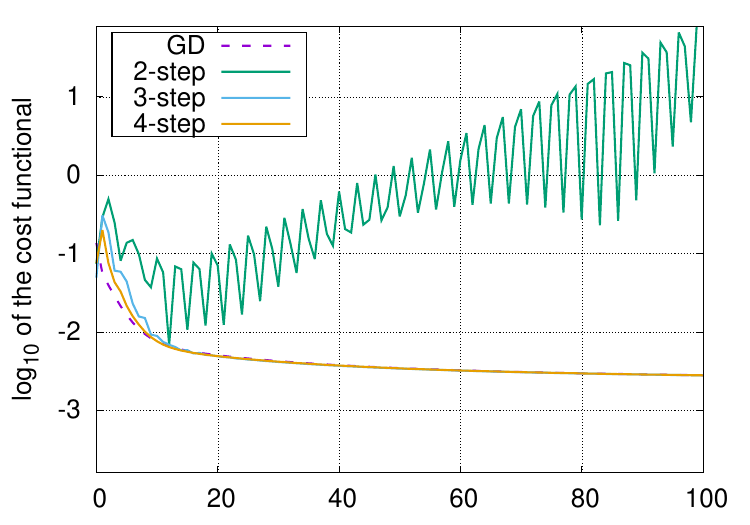}}
				\\[1ex]
				\subfloat[][$\varepsilon=1\%$, $\alpha=10^{-4}$.
				\label{fig:Jn:noise1:alp1e-4}]
				{\includegraphics[width=0.33\textwidth]{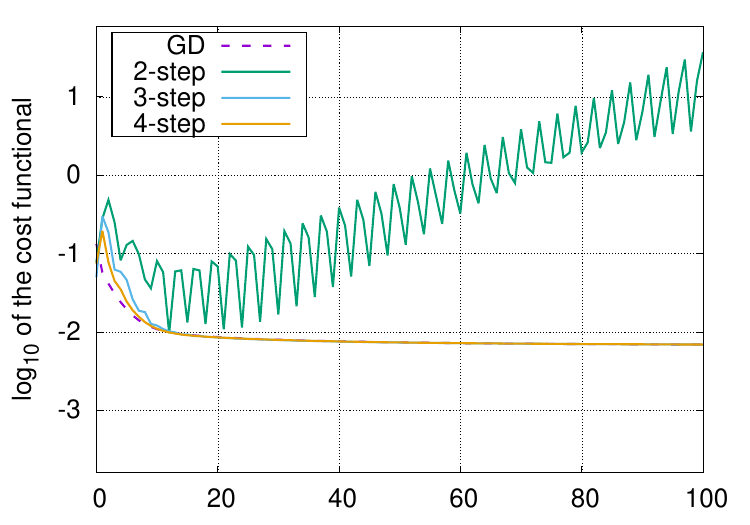}} 
				\hfil
				\subfloat[][$\varepsilon=3\%$, $\alpha=2\cdot10^{-4}$.
				\label{fig:Jn:noise3:alp2e-4}]
				{\includegraphics[width=0.33\textwidth]{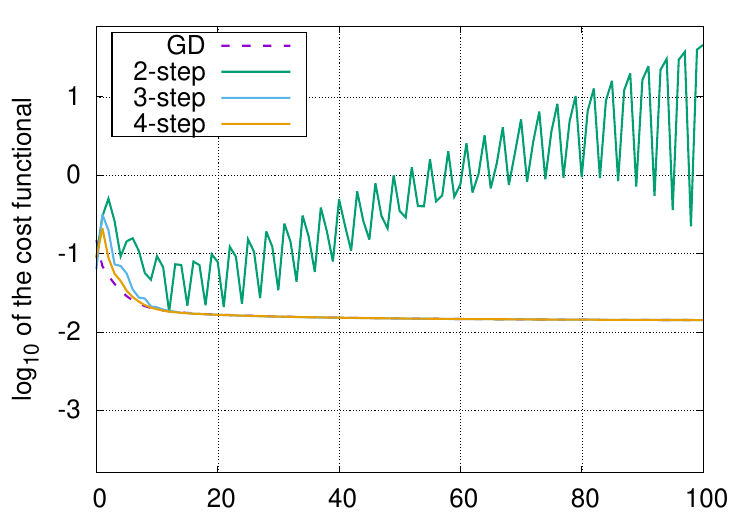}} 
				\hfil
				\subfloat[][$\varepsilon=5\%$, $\alpha=2\cdot10^{-4}$.
				\label{fig:Jn:noise5:alp2e-4}]
				{\includegraphics[width=0.33\textwidth]{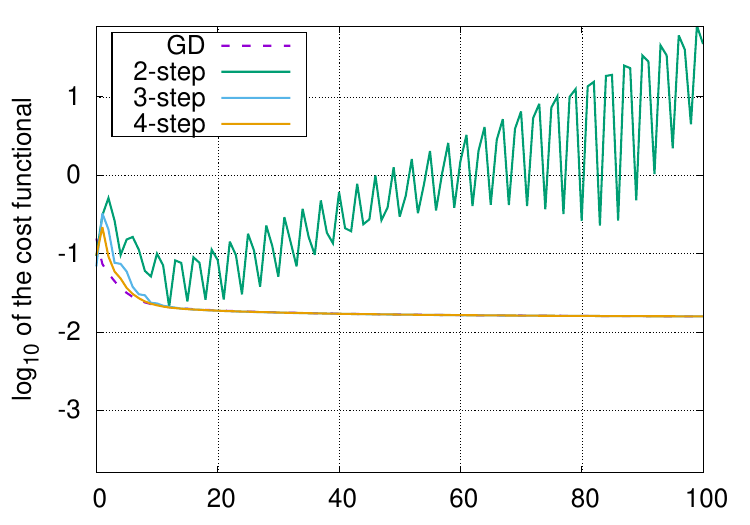}} 
				\caption{Convergence curves of semi-implicit gradient descent and $k$-step one-shot with different noise levels $\varepsilon$ and regularization parameters $\alpha$. The descent step is $\tau=4.7$.}
				\label{fig:Jn:noise135}
			\end{figure}

			\begin{figure}[htbp]
				\centering
				\subfloat[][$\varepsilon=1\%$, $\alpha=0$.
				\label{fig:RelaEn:noise1:alp0}]
				{\includegraphics[width=.33\textwidth]{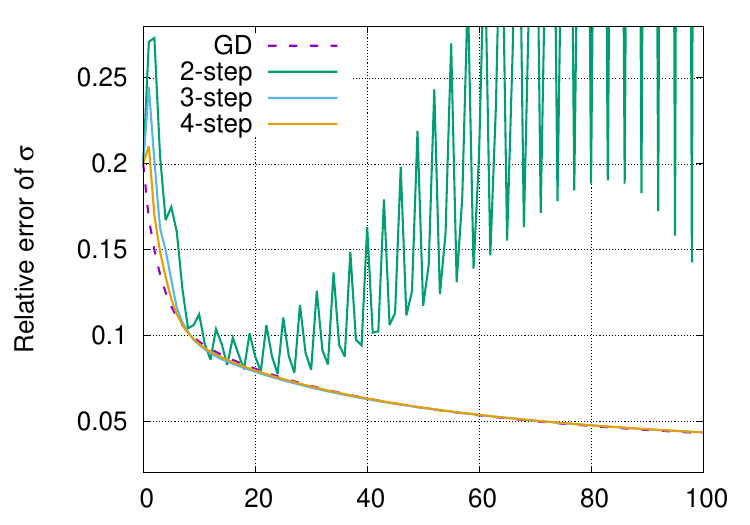}} 
				\hfil
				\subfloat[][$\varepsilon=3\%$, $\alpha=0$.
				\label{fig:RelaEn:noise3:alp0}]
				{\includegraphics[width=.33\textwidth]{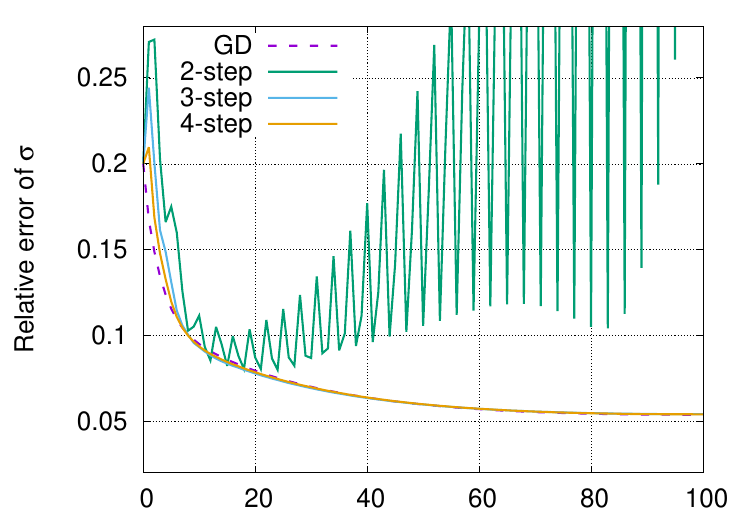}} 
				\hfil
				\subfloat[][$\varepsilon=5\%$, $\alpha=0$.
				\label{fig:RelaEn:noise5:alp0}]
				{\includegraphics[width=.33\textwidth]{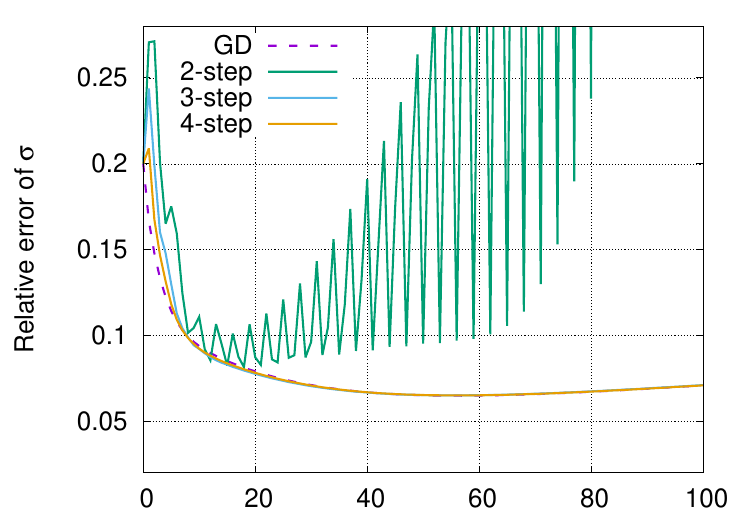}} 
				\\[1ex]
				\subfloat[][$\varepsilon=1\%$, $\alpha=10^{-4}$.
				\label{fig:RelaEn:noise1:alp1e-4}]
				{\includegraphics[width=.33\textwidth]{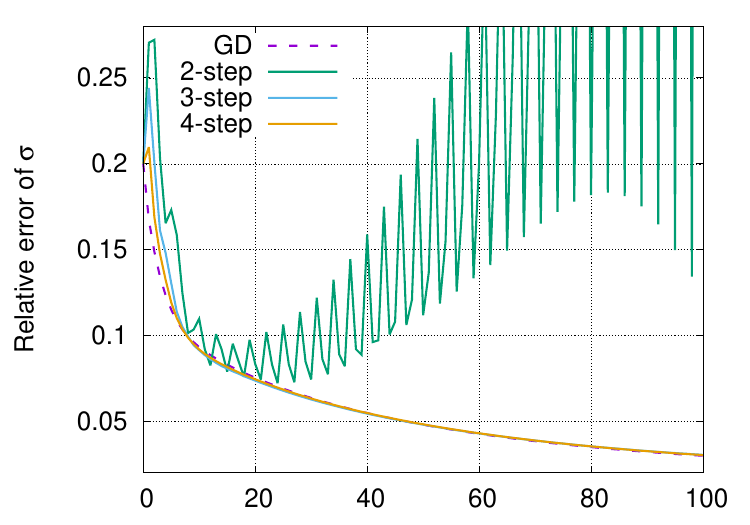}} 
				\hfil
				\subfloat[][$\varepsilon=3\%$, $\alpha=2\cdot10^{-4}$.
				\label{fig:RelaEn:noise3:alp2e-4}]
				{\includegraphics[width=.33\textwidth]{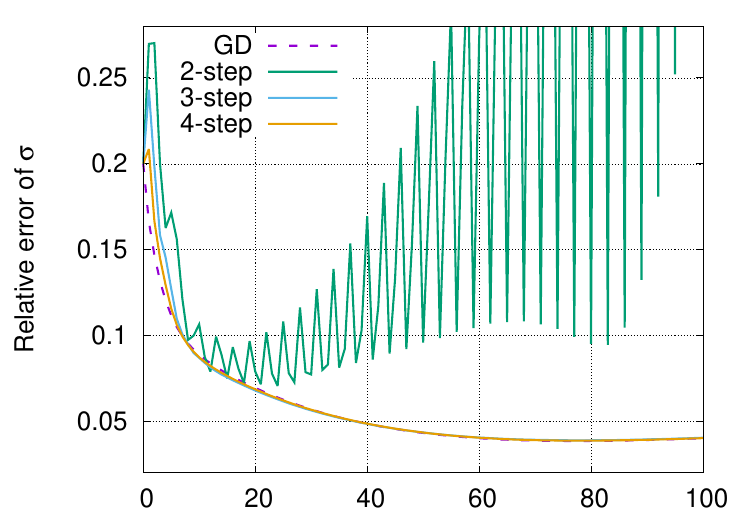}} 
				\hfil
				\subfloat[][$\varepsilon=5\%$, $\alpha=2\cdot10^{-4}$.
				\label{fig:RelaEn:noise5:alp2e-4}]
				{\includegraphics[width=.33\textwidth]{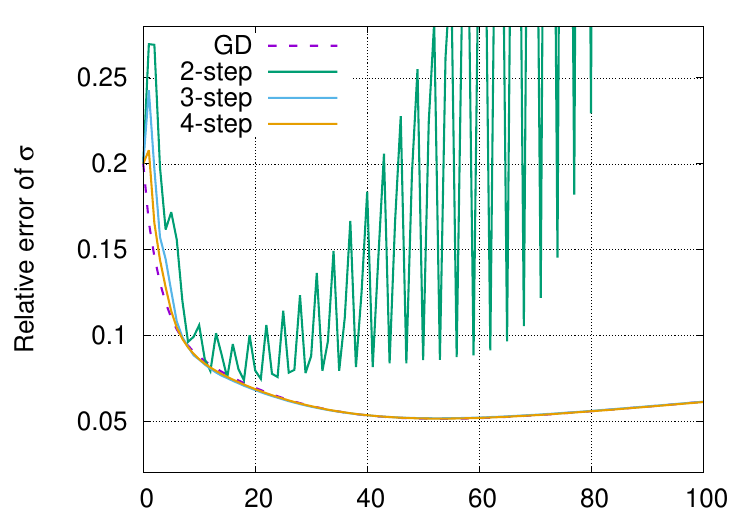}} 
				\caption{Convergence curves for the relative error on the parameter $\sigma$ of semi-implicit gradient descent and $k$-step one-shot with different noise levels $\varepsilon$ and regularization parameters $\alpha$. The descent step is $\tau=4.7$.}
				\label{fig:RelaEn:noise135}
			\end{figure}

			\begin{figure}[htbp]
				\centering
				\subfloat[][$\varepsilon=1\%$, $\alpha=0$.
				\label{fig:sigman:seim-3-shot,noise1,alp0}]
				{\includegraphics[width=.33\textwidth]{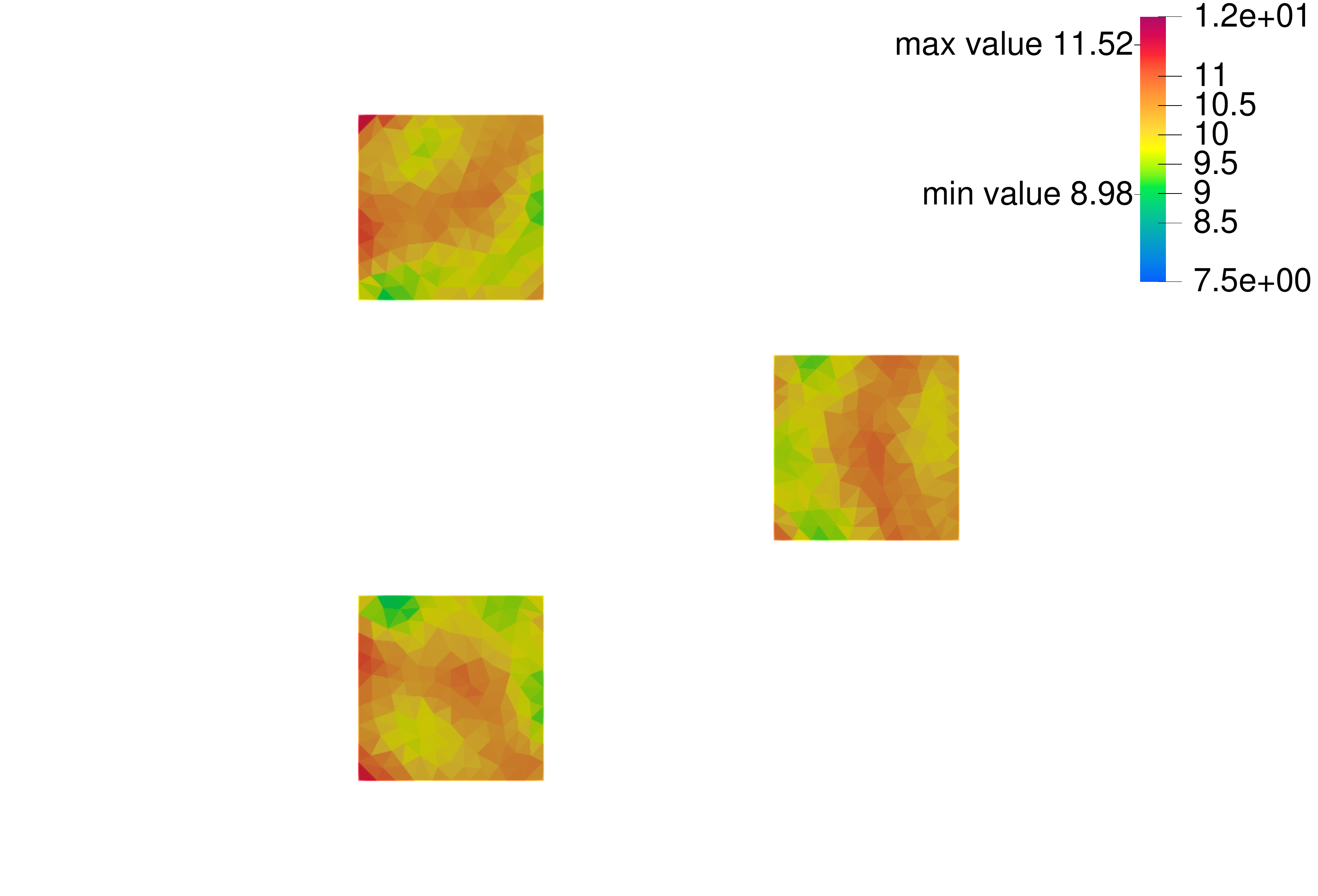}}
				\hfil
				\subfloat[][$\varepsilon=3\%$, $\alpha=0$.
				\label{fig:sigman:seim-3-shot,noise3,alp0}]
				{\includegraphics[width=.33\textwidth]{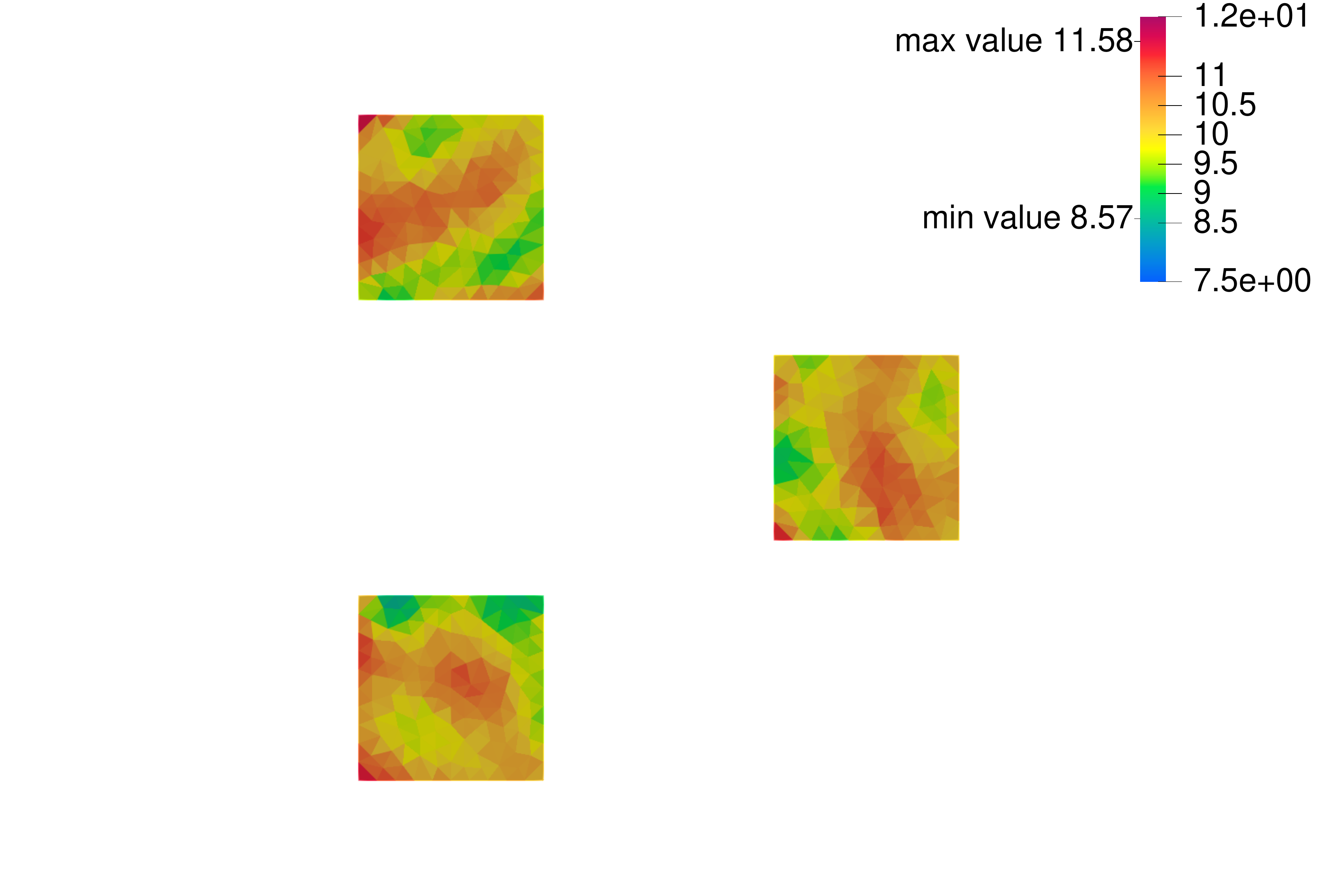}}
				\hfil
				\subfloat[][$\varepsilon=5\%$, $\alpha=0$.
				\label{fig:sigman:seim-3-shot,noise5,alp0}]
				{\includegraphics[width=.33\textwidth]{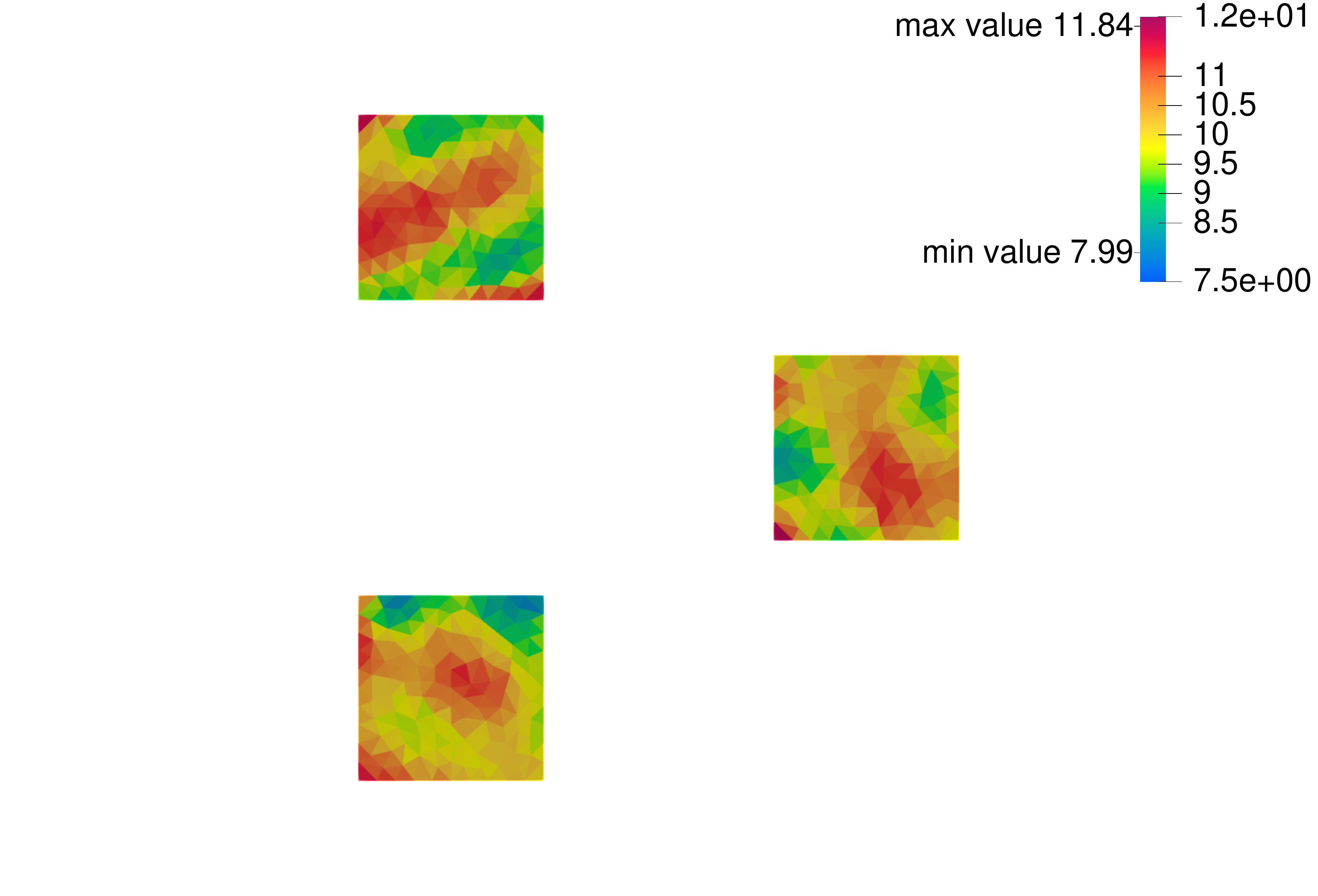}}
				\\[1ex]
				\subfloat[][$\varepsilon=1\%$, $\alpha=10^{-4}$.
				\label{fig:sigman:seim-3-shot,noise1,alp1e-4}]
				{\includegraphics[width=.33\textwidth]{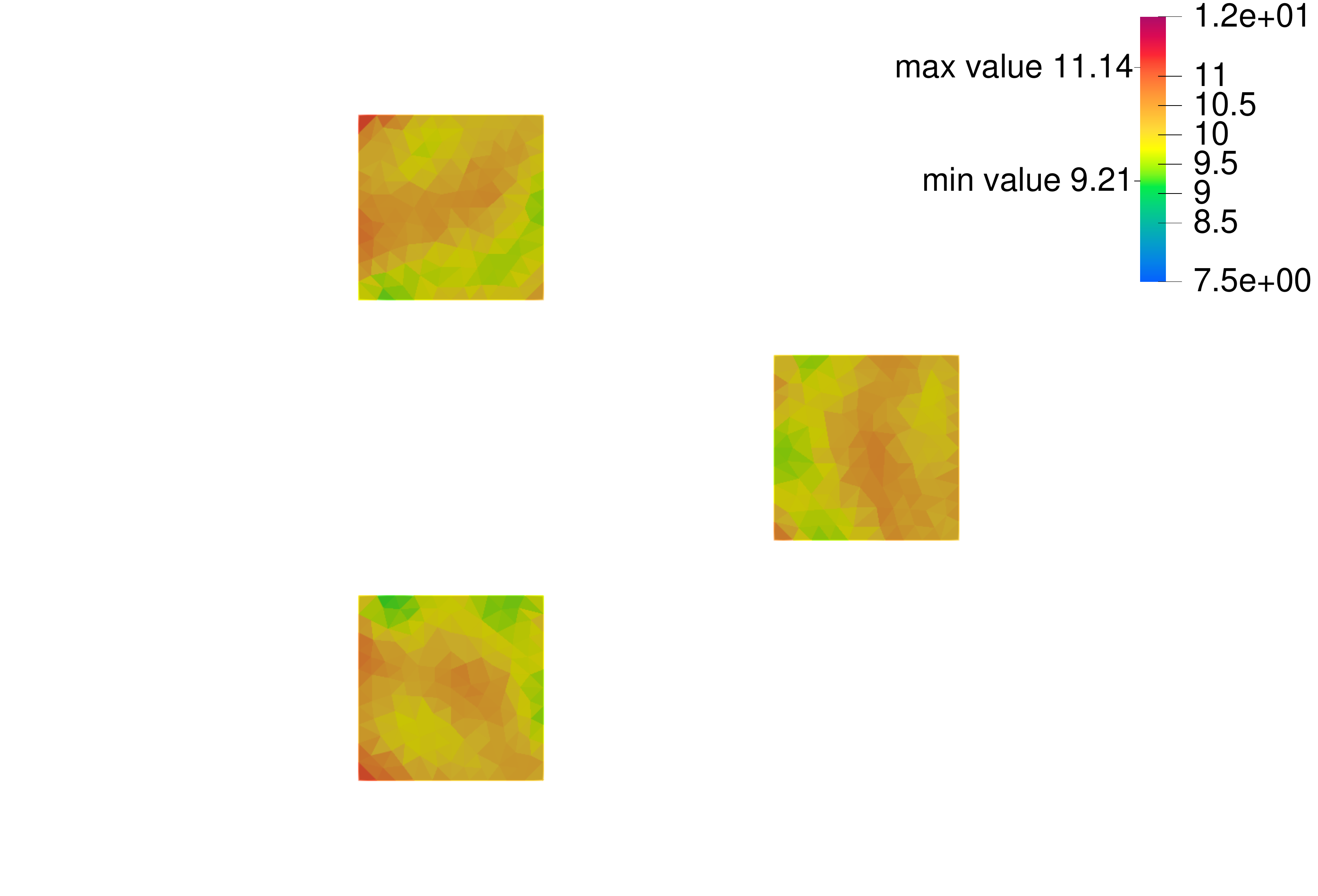}}
				\hfil
				\subfloat[][$\varepsilon=3\%$, $\alpha=2\cdot10^{-4}$.
				\label{fig:sigman:seim-3-shot,noise3,alp2e-4}]
				{\includegraphics[width=.33\textwidth]{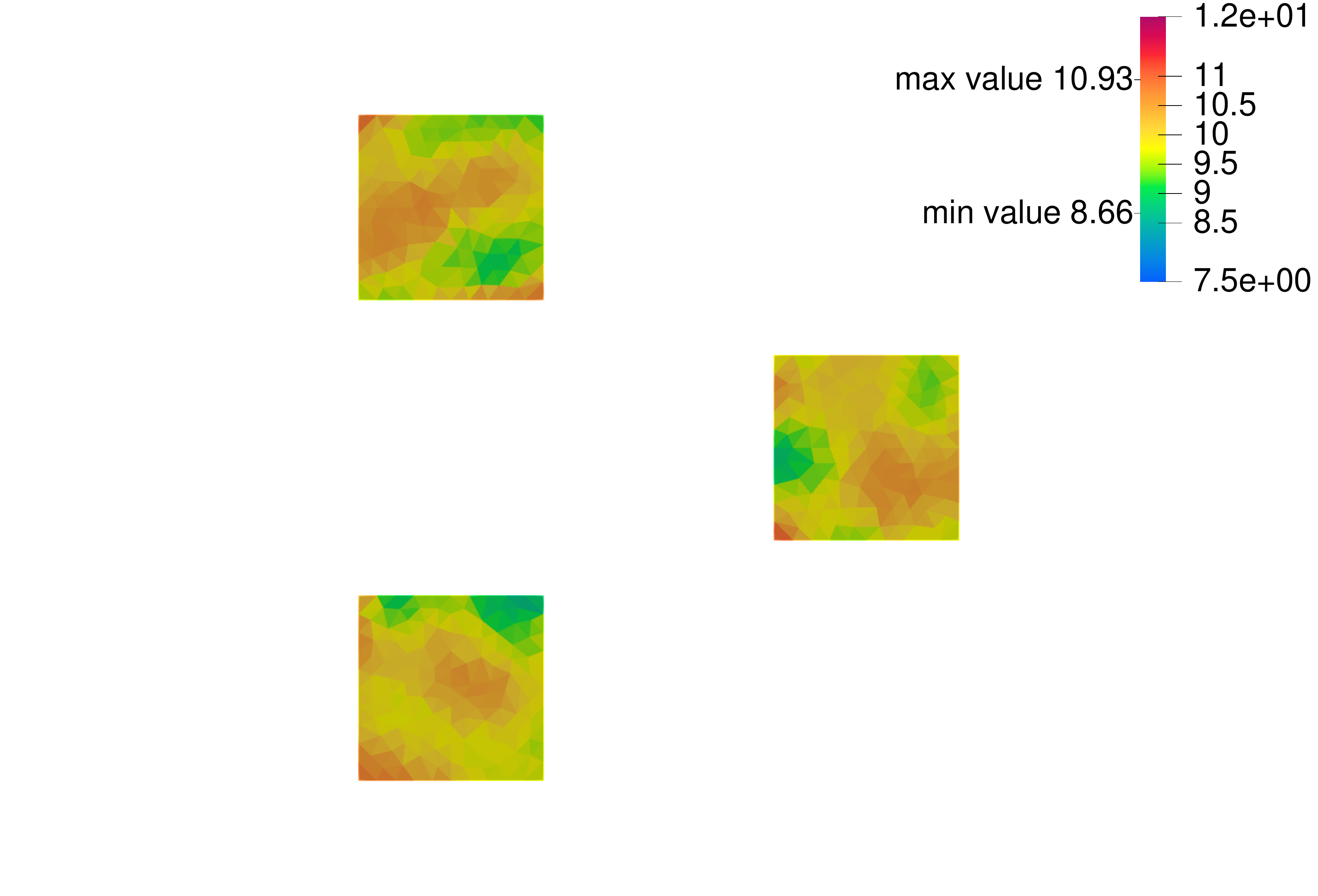}}
				\hfil
				\subfloat[][$\varepsilon=5\%$, $\alpha=2\cdot10^{-4}$.
				\label{fig:sigman:seim-3-shot,noise5,alp2e-4}]
				{\includegraphics[width=.33\textwidth]{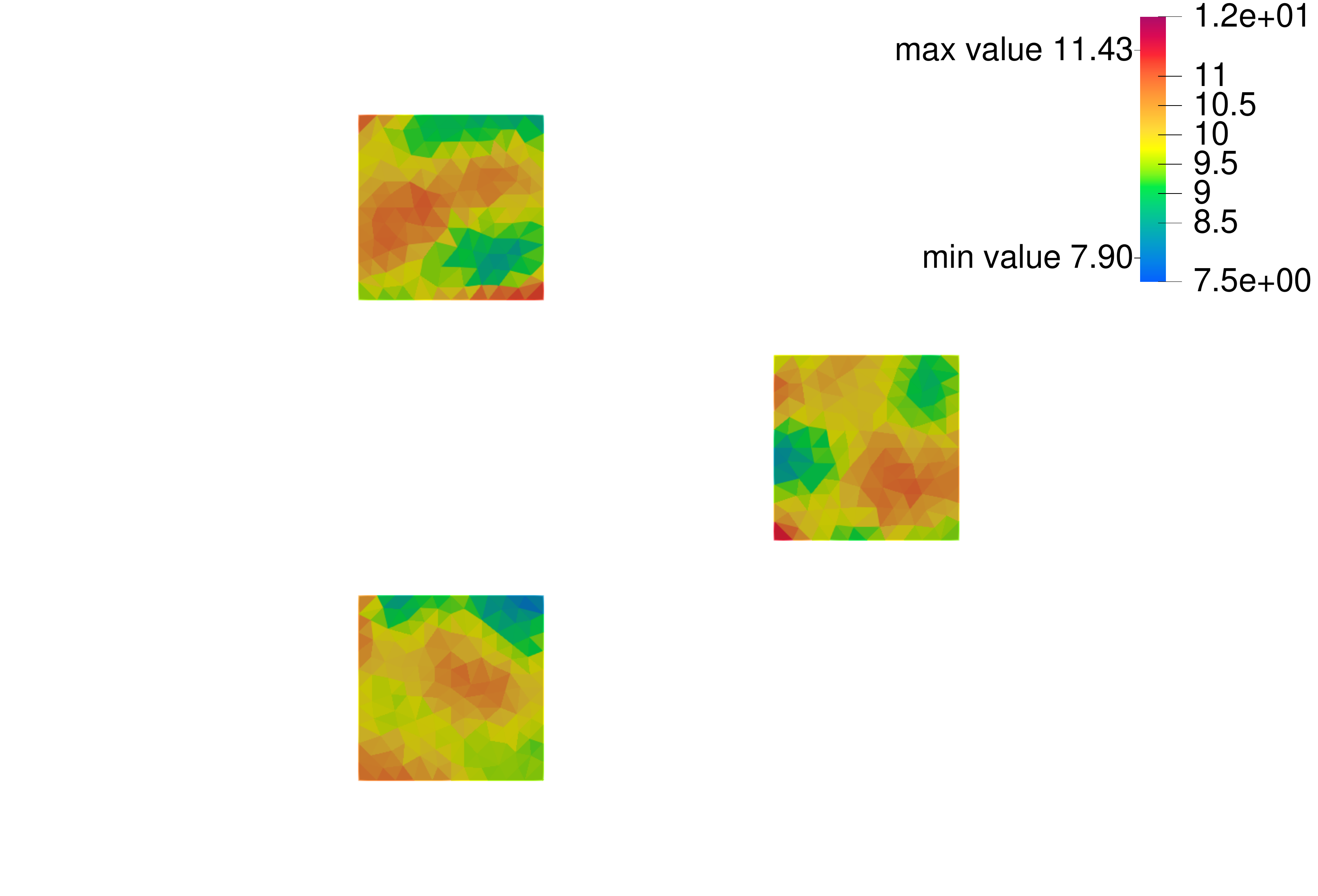}}
				\caption{Reconstructed $\sigma$ by semi-implicit $3$-step one-shot with different noise levels $\varepsilon$. The exact $\sigma$ is a constant function in each square with the value $\sigma^\mathrm{ex}=10$.}
				\label{fig:sigman,seim-3-shot,noise135}
			\end{figure}

						\subsection{Robustness with respect to the size of the discretized problem} 
							\label{sec:exp3}
							In the following experiment, we would like to confirm that the convergence rate of the $k$-step one-shot methods is asymptotically independent of the size of the discretized problem similarly to what is known for the gradient descent algorithms. 
							In order to modify the number of discretization points without modifying the structure of the background media, we replace the random contribution $\sigma_\mathrm{r}$ with the constant function $\sigma_\mathrm{r}=1$. 
							This means that we choose $\sigma_0=\tilde{\sigma}_0+\delta$. 
							(we fix $\delta=0.01$).
							We keep the same configuration as in the experiment presented in Figure \ref{fig:meshnoise} in terms of number of sources, number of measurements, and values of the exact conductivity contrast and of the initial guess (i.e.~$\sigma^\mathrm{ex}=10$ and $\sigma^0=12$ in each square), but we change the configuration of the support of the unknown conductivity contrast to investigate another test case.
							The exact data $g_i$, $1\le i\le 6$ are generated using a fine mesh corresponding to a mesh size $h=\lambda/60$ and a number of degrees of freedom $157844$. 
							The data are then corrupted with $\varepsilon=3\%$ relative random noise. For the inversion, we use two different meshes respectively corresponding to mesh sizes $h=\lambda/40$ ($n_u=70093$) and $h=\lambda/55$ ($n_u=128490$). 
							This experiment is illustrated by Figure \ref{fig:mesh_new_exp}, where the mesh with $h=\lambda/40$ is displayed. 
							We also show the mesh for the unknown $\sigma$, that results in $n_\sigma=10$ unknowns.
							For the inversion, we choose to show the result for the semi-implicit $2$-step one-shot method with fixed descent step $\tau=2$ and regularization parameter $\alpha=2\cdot10^{-4}$. 
							As shown in Figure \ref{fig:tune_h}, the number of outer iterations is not much affected by the size of the discrete system. 
							We observe the same behavior for other semi-implicit $k$-step one-shot methods with $k\ge3$.
						
						\subsection{Dependence of the number of outer iterations on the norm of $B$}
							\label{sec:exp4} 
							The purpose of this experiment is to illustrate how the convergence speed would depend on the norm of $B$. Although we did not analyze the convergence rate, one indeed expects that the smaller the norm of $B$, the faster the convergence.
							This is what is observed in Figure \ref{fig:tune_delta}, where we repeat the same experiment as in Figure \ref{fig:tune_h} for $h=\lambda/40$, but we vary the values of $\delta$ ($\delta=0.01$ and $\delta=0.02$). 
							For a given accuracy, the number of outer iterations for $\delta=0.01$ is less than the one for $\delta=0.02$.
						
						% New experiment settings: $\sigma=1+\delta$. The conductivity contrast is supported in $5$ squares of size $\frac{\lambda}{4}$ (the centers of the squares lie on the circle of radius $R-\frac{\lambda}{4}$ instead of $R-\frac{\lambda}{2}$ for the case of $3$ squares). $M=6$ data. $n_\sigma=10$, $n_g=12$ are fixed. The size of the mesh for generating data is $h=\frac{\lambda}{60}$, for which the number of degrees of freedom is $157844$. For $h=\frac{\lambda}{40}$, $n_u=70093$. For $h=\frac{\lambda}{55}$, $n_u=128490$. The errors of the measurements are less than 3\% so we add $3\%$ of noise. By trial and error, $\alpha=2\cdot10^{-4}$ and $\tau=2$.
						
						\begin{figure}
							\centering
							\subfloat[][Mesh for $u$ ($h=\lambda/40$). \label{fig:Thu_new_exp}]
							{\includegraphics[width=.33\textwidth]{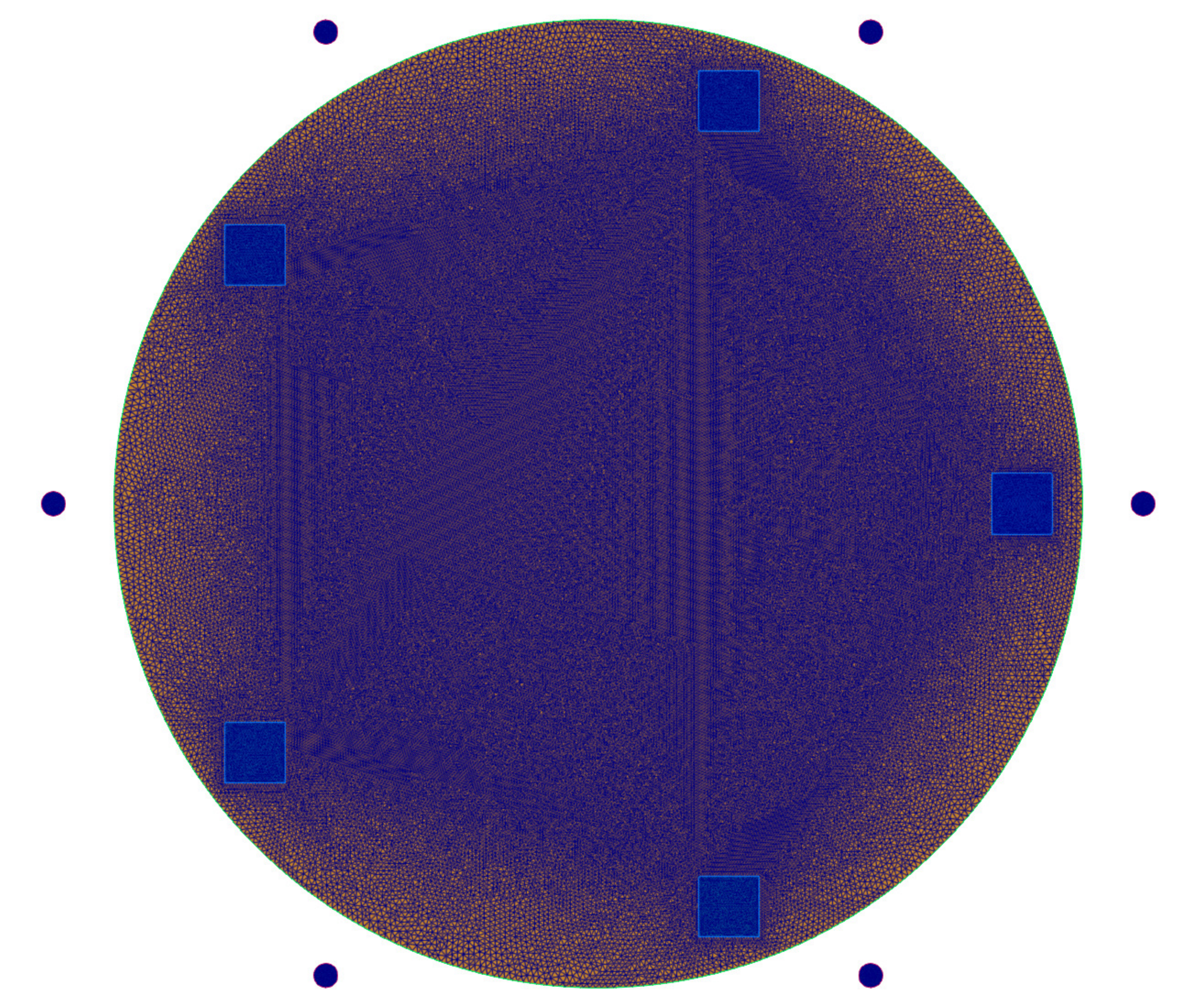}} 
							\hfil
							\subfloat[][Mesh for $\sigma$ ($n_\sigma=10$).\label{fig:Ths_new_exp}]
							{\includegraphics[width=.31\textwidth]{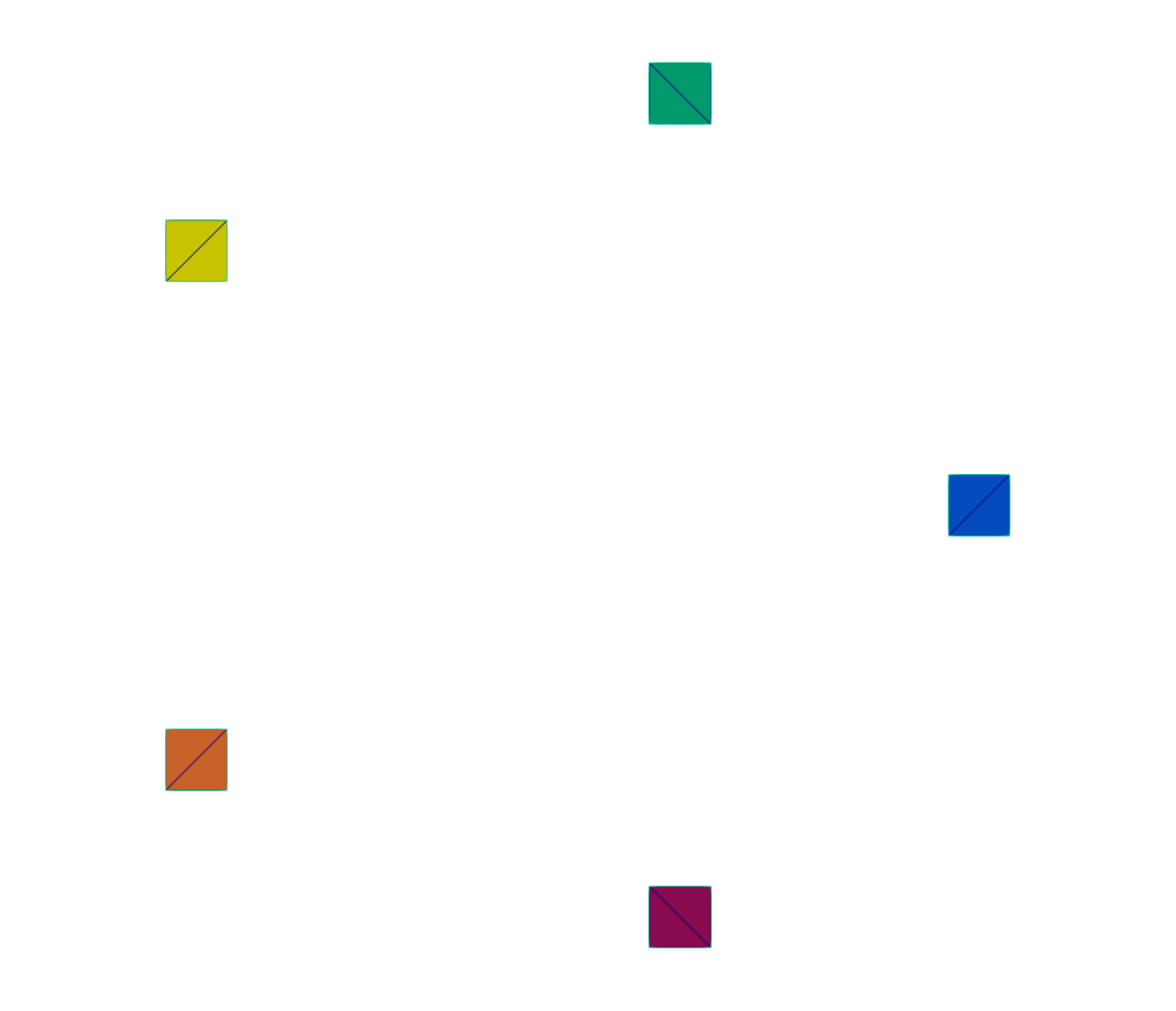}} 
							\hfil
							\subfloat[][Mesh for $g^\varepsilon$ ($n_g=12)$. \label{fig:Thumes_new_exp}]
							{\begin{tikzpicture}
									\def \n {12} 
									\def \radius {4cm}
									\node[circle,draw=gray,fill=white,minimum size=\radius] (a) at (0,0) {};
									%\filldraw [gray] (a.center) circle [radius=0.01cm];
									\foreach \X in {1,...,\n}
									{\def \s {-360/\n*\X};
										\draw[black] (a.\s) -- ++ (\s:0.04);}
							\end{tikzpicture}} 
							\caption{The configuration for the experiment of Sections \ref{sec:exp3} and \ref{sec:exp4}. The values of the exact conductivity contrast is $\sigma^\mathrm{ex}=10$ in each square.
							The small circles outside the mesh for $u$ indicate the source locations $y_i$, $1\le i\le 6$.}
							\label{fig:mesh_new_exp}
						\end{figure}
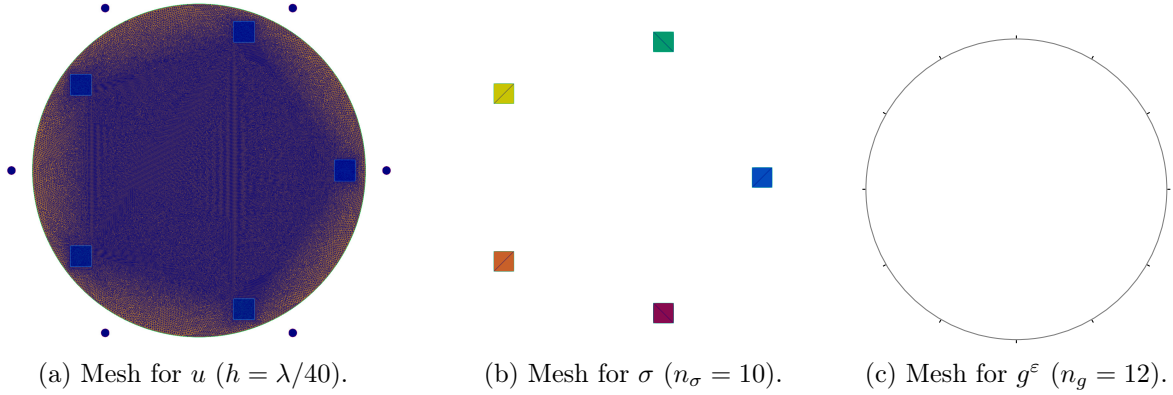 
						
						\begin{figure}[htbp]
							\centering
							\subfloat[][$\log_{10}$ of the cost functional. 
							\label{fig:Jn_noise0_tau2_htune}]
							{\includegraphics[width=0.43\textwidth]{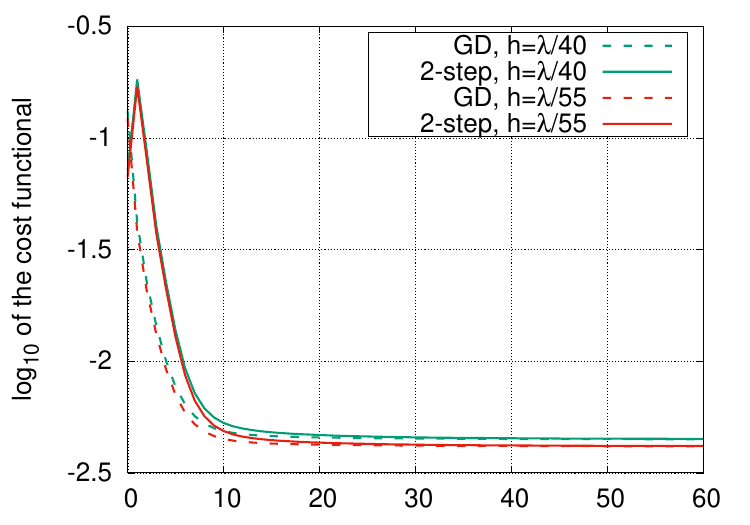}}
							\hfil
							\subfloat[][Relative error on $\sigma$. 
							\label{fig:RelaEn_noise0_tau2_htune}]
							{\includegraphics[width=0.43\textwidth]{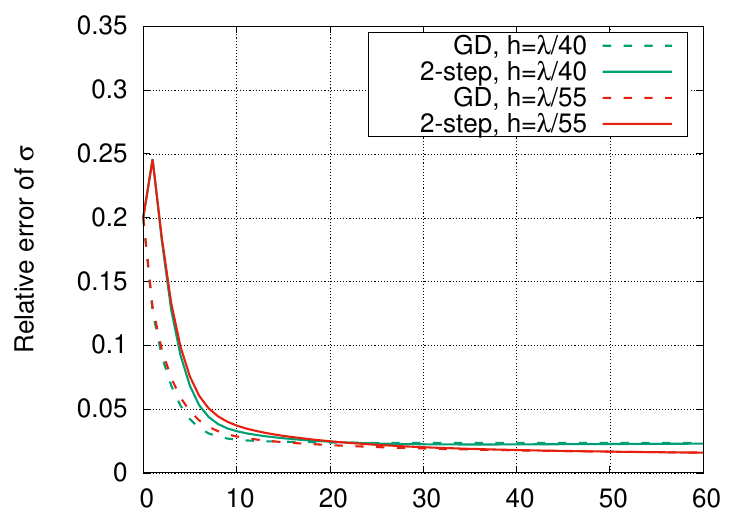}}
							\caption{Result for the experiment in Section \ref{sec:exp3} displaying the convergence rate for semi-implicit $2$-step one-shot and gradient descent with different mesh sizes for the state equations. The descent step is $\tau=2$, the noise level is $\varepsilon=3\%$ and the regularization parameter is $\alpha=2\cdot10^{-4}$.}
							\label{fig:tune_h}
						\end{figure}
						
						\begin{figure}[htbp]
							\centering
							\subfloat[][$\log_{10}$ of the cost functional. 
							\label{fig:Jn_noise0_tau2_delta0.01-0.02}]
							{\includegraphics[width=0.43\textwidth]{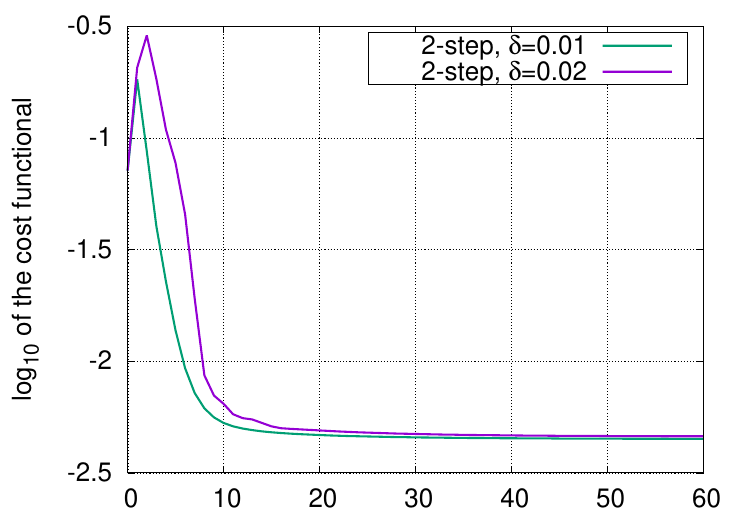}}
							\hfil
							\subfloat[][Relative error on $\sigma$. 
							\label{fig:RelaEn_noise0_tau2_delta0.01-0.02}]
							{\includegraphics[width=0.43\textwidth]{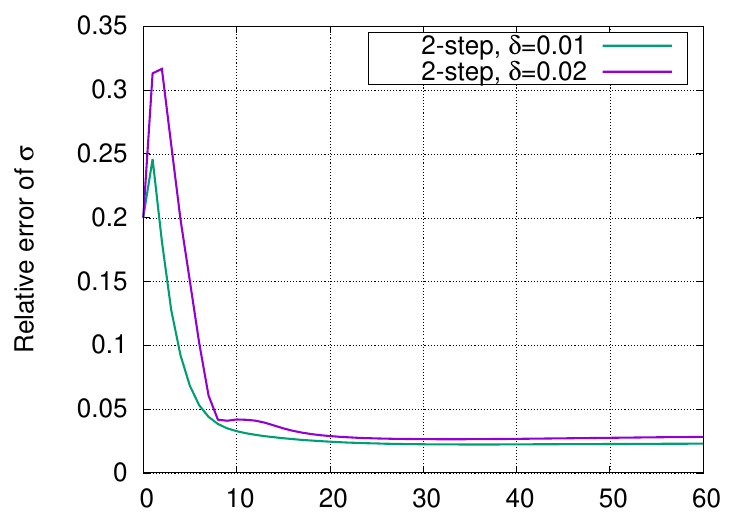}}
							\caption{Result for the experiment in Section \ref{sec:exp4} displaying the dependence of the convergence rate for semi-implicit $2$-step one-shot with respect to $\delta$. The descent step is $\tau=2$, the noise level is $\varepsilon=3\%$ and the regularization parameter is $\alpha=2\cdot10^{-4}$.}
							\label{fig:tune_delta}
						\end{figure}
						
						\section{Conclusion}
						We have proved sufficient conditions on the descent step for the convergence of semi-implicit multi-step one-shot methods, with a regularization parameter $\alpha \ge 0$. This complements the results obtained in our research report \cite{bon22} for explicit schemes with no regularization. Although these bounds on the descent step are not optimal, to our knowledge no other bounds, explicit in the number of inner iterations, are available in the literature for multi-step one-shot methods. Furthermore, we have shown in the numerical experiments that very few inner iterations on the forward and adjoint problems may be enough to guarantee similar results as for the classical gradient descent algorithm.
						
						These preliminary numerical results are just proof of concept
						since the size of the direct problem is not very large. In our future work, we shall carry out more realistic numerical investigation where the iterative solvers are based on domain decomposition methods (see e.g.~\cite{dolean-ddm15}), which are well adapted to large-scale problems. In addition, the inner fixed point iterations could be replaced by more efficient Krylov subspace methods, such as conjugate gradient or GMRES, and one could use L-BFGS instead of gradient descent as optimization algorithm.    
						Another interesting issue is how to adapt the number of inner iterations in the course of the outer iterations. Moreover, based on this linear inverse problem study, we plan to tackle the challenging case of non-linear inverse problems, as studied in the seminal paper \cite{hanke95}.
						
						\bibliographystyle{plain} 
						\bibliography{k-shot.bib}

\begin{thebibliography}{10}

\bibitem{audibert2023}
L.~Audibert, H.~Girardon, H.~Haddar, and P.~Jolivet.
\newblock Inversion of eddy-current signals using a level-set method and block
  {K}rylov solvers.
\newblock {\em SIAM Journal on Scientific Computing}, 45(3):B366--B389, 2023.

\bibitem{bon22}
M.~Bonazzoli, H.~Haddar, and T.A. Vu.
\newblock {Convergence analysis of multi-step one-shot methods for linear
  inverse problems}.
\newblock Research Report RR-9477, {Inria Saclay; ENSTA ParisTech}, 2022.

\bibitem{burger02}
M.~Burger and W.~M{\"u}hlhuber.
\newblock Iterative regularization of parameter identification problems by
  sequential quadratic programming methods.
\newblock {\em Inverse Problems}, 18:943--969, 2002.

\bibitem{dolean-ddm15}
V.~Dolean, P.~Jolivet, and F.~Nataf.
\newblock {\em An {Introduction} to {Domain} {Decomposition} {Methods}:
  {Algorithms}, {Theory}, and {Parallel} {Implementation}}.
\newblock Society for Industrial and Applied Mathematics, Philadelphia, PA,
  2015.

\bibitem{gauger12}
N.~Gauger, A.~Griewank, A.~Hamdi, C.~Kratzenstein, E.{\"O}zkaya, and T.~Slawig.
\newblock Automated extension of fixed point {PDE} solvers for optimal design
  with bounded retardation.
\newblock In {\em Constrained Optimization and Optimal Control for Partial
  Differential Equations, International Series of Numerical Mathematics}, pages
  99--122. Springer Basel, 2012.

\bibitem{greenbaum97}
A.~Greenbaum.
\newblock {\em Iterative {Methods} for {Solving} {Linear} {Systems}}.
\newblock Number~17 in Frontiers in {Applied} {Mathematics}. Soc. for
  Industrial and Applied Math, Philadelphia, 1997.

\bibitem{griewank06}
A.~Griewank.
\newblock Projected {Hessians} for {Preconditioning} in {One}-{Step}
  {One}-{Shot} {Design} {Optimization}.
\newblock In {\em Large-{Scale} {Nonlinear} {Optimization}}, volume~83, pages
  151--171. Springer US, Boston, MA, 2006.
\newblock Series Title: Nonconvex Optimization and Its Applications.

\bibitem{guenther16}
S.~Günther, N.~R. Gauger, and Q.~Wang.
\newblock Simultaneous single-step one-shot optimization with unsteady {PDEs}.
\newblock {\em Journal of Computational and Applied Mathematics}, 294:12--22,
  2016.

\bibitem{haber01}
E.~Haber and U.~M. Ascher.
\newblock Preconditioned all-at-once methods for large, sparse parameter
  estimation problems.
\newblock {\em Inverse Problems}, 17(6):1847--1864, 2001.

\bibitem{hamdi09}
A.~Hamdi and A.~Griewank.
\newblock Reduced quasi-{N}ewton method for simultaneous design and
  optimization.
\newblock {\em Computational Optimization and Applications}, 49(3):521--548,
  2009.

\bibitem{hamdi10}
A.~Hamdi and A.~Griewank.
\newblock Properties of an augmented {L}agrangian for design optimization.
\newblock {\em Optimization Methods and Software}, 25(4):645--664, 2010.

\bibitem{hanke95}
Martin Hanke, Andreas Neubauer, and Otmar Scherzer.
\newblock A convergence analysis of the {{Landweber}} iteration for nonlinear
  ill-posed problems.
\newblock {\em Numerische Mathematik}, 72(1):21--37, 1995.

\bibitem{hazra05}
S.B. Hazra, V.~Schulz, J.~Brezillon, and N.R. Gauger.
\newblock Aerodynamic shape optimization using simultaneous
  pseudo-timestepping.
\newblock {\em Journal of Computational Physics}, 204(1):46--64, 2005.

\bibitem{FreeFEM}
F.~Hecht.
\newblock New development in {F}ree{F}em++.
\newblock {\em J. Numer. Math.}, 20(3-4):251--265, 2012.

\bibitem{kaltenbacher16}
B.~Kaltenbacher.
\newblock Regularization based on all-at-once formulations for inverse
  problems.
\newblock {\em SIAM Journal on Numerical Analysis}, 54(4):2594--2618, 2016.

\bibitem{kaltenbacher17}
B~Kaltenbacher.
\newblock All-at-once versus reduced iterative methods for time dependent
  inverse problems.
\newblock {\em Inverse Problems}, 33(6):064002, 2017.

\bibitem{kaltenbacher14}
B.~Kaltenbacher, A.~Kirchner, and B.~Vexler.
\newblock {Goal oriented adaptivity in the IRGNM for parameter identification
  in PDEs II: all-at-once formulations}.
\newblock {\em Inverse Problems}, 30:045002, 2014.

\bibitem{marden66}
M.~Marden.
\newblock {\em Geometry of {Polynomials}}.
\newblock Number~3 in Mathematical {Surveys} and {Monographs}. American Math.
  Soc, Providence, RI, 2nd edition, 1966.

\bibitem{tram2019}
Tram Thi~Ngoc Nguyen.
\newblock Landweber–{{Kaczmarz}} for parameter identification in
  time-dependent inverse problems: All-at-once versus reduced version.
\newblock {\em Inverse Problems}, 35(3):035009, 2019.

\bibitem{tram2024}
Tram Thi~Ngoc Nguyen.
\newblock Bi-level iterative regularization for inverse problems in nonlinear
  {{PDEs}}.
\newblock {\em Inverse Problems}, 40(4):045020, 2024.

\bibitem{gauger09}
E.~{\"O}zkaya and N.~R. Gauger.
\newblock Single-step {One}-shot {Aerodynamic} {Shape} {Optimization}.
\newblock In {\em Optimal {Control} of {Coupled} {Systems} of {Partial}
  {Differential} {Equations}}, volume 158, pages 191--204. Birkh{\"a}user
  Basel, Basel, 2009.
\newblock Series Title: International Series of Numerical Mathematics.

\bibitem{schulz09}
V.~Schulz and I.~Gherman.
\newblock One-{Shot} {Methods} for {Aerodynamic} {Shape} {Optimization}.
\newblock In {\em {MEGADESIGN} and {MegaOpt} - {German} {Initiatives} for
  {Aerodynamic} {Simulation} and {Optimization} in {Aircraft} {Design}}, volume
  107, pages 207--220. Springer Berlin Heidelberg, Berlin, Heidelberg, 2009.
\newblock Series Title: Notes on Numerical Fluid Mechanics and
  Multidisciplinary Design.

\bibitem{shenoy97}
A.~Shenoy, M.~Heinkenschloss, and E.~M. Cliff.
\newblock Airfoil design by an all-at-once method.
\newblock {\em International Journal of Computational Fluid Dynamics},
  11(1-2):3--25, 1997.

\bibitem{taasan91}
S.~Ta'asan.
\newblock "{One} {Shot}" {Methods} for {Optimal} {Control} of {Distributed}
  {Parameter} {Systems} {I}: {Finite} {Dimensional} {Control}.
\newblock Technical Report 91-2, ICASE, Hampton, 1991.

\bibitem{taasan92}
S.~Ta'asan, G.~Kuruvila, and M.~Salas.
\newblock Aerodynamic design and optimization in one shot.
\newblock In {\em 30th {Aerospace} {Sciences} {Meeting} and {Exhibit}}, Reno,
  NV, U.S.A., 1992. American Institute of Aeronautics and Astronautics.

\bibitem{tarantola82}
A.~Tarantola and B.~Valette.
\newblock Generalized nonlinear inverse problems solved using the least squares
  criterion.
\newblock {\em Reviews of Geophysics}, 20(2):219--232, 1982.

\bibitem{TOURNIER2019}
P.-H. Tournier, I.~Aliferis, M.~Bonazzoli, M.~de~Buhan, M.~Darbas, V.~Dolean,
  F.~Hecht, P.~Jolivet, I.~{El Kanfoud}, C.~Migliaccio, F.~Nataf, C.~Pichot,
  and S.~Semenov.
\newblock Microwave tomographic imaging of cerebrovascular accidents by using
  high-performance computing.
\newblock {\em Parallel Computing}, 85:88--97, 2019.

\bibitem{leeuwen13}
T.~van Leeuwen and F.~J. Herrmann.
\newblock {Mitigating local minima in full-waveform inversion by expanding the
  search space}.
\newblock {\em Geophysical Journal International}, 195(1):661--667, 2013.

\bibitem{leeuwen15}
T.~van Leeuwen and F.~J. Herrmann.
\newblock A penalty method for {PDE}-constrained optimization in inverse
  problems.
\newblock {\em Inverse Problems}, 32(1):015007, 2015.

\end{thebibliography}

						\appendix 

						\section{Some useful lemmas}
						\label{app:lems}
						We state auxiliary results about matrices like those appearing in the eigenvalue equations \eqref{eq:seim-1-shot:eq-eigen} and  \eqref{eq:seim-k-shot:eq-eigen}.
						\begin{lemma}\label{lem:inv(I-T)}
							Let $(\C^{n\times n},\norm{\cdot})$ be a normed space and $T\in\C^{n\times n}$. If $\rho(T)<1$, then 
							$$\sum_{k=0}^{\infty}T^k \text{ converges and } \sum_{k=0}^{\infty}T^k=(I-T)^{-1}.$$
							Moreover, if $\norm{T}<1$, $\norm{(I-T)^{-1}}\le \frac{1}{1-\norm{T}}$.	
						\end{lemma}
						
						\begin{lemma}\label{lem:inv(I-T/z)}
							Let $T\in\C^{n\times n}$ such that $\rho(T)<1$. Set
							\begin{equation}
								\label{lem:supbound}
								s(T) \coloneqq \sup_{z\in\C, |z|\ge 1}\norm{\left(I-T/z\right)^{-1}}
							\end{equation}
							then $0<\norm{(I-T)^{-1}}\le s(T)=s(T^*)<+\infty$. Moreover, if $\norm{T}<1$, $0<s(T)\le\cfrac{1}{1-\norm{T}}$.
						\end{lemma}
						\begin{proof}
							The existence of $s(T)$ (and also $s(T^*)$) is deduced from the fact that the functional $z\mapsto \norm{\left(I-T/z\right)^{-1}}$, with $z\in\C, |z|\ge 1$, is well-defined and continuous. For every $z\in\C, |z|\ge 1$ we have $$\norm{\left(I-T/z\right)^{-1}}=\norm{\left((I-T/z)^{-1}\right)^*}=\norm{(I-T^*/z^*)^{-1}}\le s(T^*)$$
							and
							$$\norm{\left(I-T^*/z\right)^{-1}}=\norm{\left((I-T^*/z)^{-1}\right)^*}=\norm{(I-T/z^*)^{-1}}\le s(T),$$
							thus $s(T)=s(T^*)$. The second part of conclusion is obtained by Lemma~\ref{lem:inv(I-T)}.
						\end{proof}
						
						The following lemma says that, for $T\in\C^{n\times n}$ and $\lambda\in\C, |\lambda|\ge 1$, we can decompose $\bigl(I-\frac{T}{\lambda}\bigr)^{-1}=P(\lambda)+\ic Q(\lambda)$ and $\bigl(I-\frac{T^*}{\lambda}\bigr)^{-1}=P(\lambda)^*+\ic Q(\lambda)^*$, and gives bounds for $P(\lambda)$ and $Q(\lambda)$. 
						%and moreover, $\norm{Q(\lambda)}\lesssim |\sin\phi|$.
						\begin{lemma}\label{lem:decomPQ}
							Let $T\in\C^{n\times n}$ such that $\rho(T)<1$ and $\lambda\in\C, |\lambda|\ge 1$. Write $\frac{1}{\lambda}=r(\cos\phi+\ic\sin\phi)$ in polar form, where $0<r\le 1$ and $\phi\in[-\pi,\pi]$. Then
							$$\left(I-\frac{T}{\lambda}\right)^{-1}=P(\lambda)+\ic Q(\lambda)\quad\text{and}\quad\left(I-\frac{T^*}{\lambda}\right)^{-1}=P(\lambda)^*+\ic Q(\lambda)^*$$
							where
							\begin{equation*}
								P(\lambda)=(I-r\cos\phi \, T)(I-2r\cos\phi \, T +r^2 T^2)^{-1} 	\end{equation*}
							and
							\begin{equation*}
								Q(\lambda)=r\sin\phi \, T(I-2r\cos\phi \, T +r^2 T^2)^{-1}
							\end{equation*}
							are $\C^{n\times n}$-valued functions. We also have the following properties:
							
							\begin{enumerate}[label=(\roman*)]
								\item $\norm{P(\lambda)}\le \left(1+\norm{T}\right)s(T)^2$ and $\norm{Q(\lambda)}\le|\sin\phi|\norm{T}s(T)^2\le\norm{T}s(T)^2.$
								%$$\norm{P(\lambda)}\le \left(1+r|\cos\phi|\norm{T}\right)s(T)^2\le \left(1+\norm{T}\right)s(T)^2,$$
								%$$\norm{Q(\lambda)}\le r|\sin\phi|\norm{T}s(T)^2\le|\sin\phi|\norm{T}s(T)^2\le\norm{T}s(T)^2$$ 
								%			and if $\norm{T}<1$ then
								%			$$\norm{P(\lambda)}\le \cfrac{1+r|\cos\phi|\norm{T}}{(1-r\norm{T})^2}, \quad \norm{Q(\lambda)}\le\cfrac{|\sin\phi|\norm{T}}{(1-\norm{T})^2};$$	
								%				$$\norm{P(\lambda)}\le \cfrac{1+r|\cos\phi|\norm{T}}{(1-r\norm{T})^2}, \quad \norm{Q(\lambda)}\le \cfrac{r|\sin\phi|\norm{T}}{(1-r\norm{T})^2}\le\cfrac{|\sin\phi|\norm{T}}{(1-\norm{T})^2};$$	
								\item Moreover if $\norm{T}<1$ then		
								$$\norm{P(\lambda)}\le\cfrac{1}{1-\norm{T}}\quad\mbox{and}\quad \norm{Q(\lambda)} \le\cfrac{\norm{T}}{1-\norm{T}}.$$	
								%			$$\norm{P(\lambda)}\le \cfrac{1}{1-r\norm{T}}\le\cfrac{1}{1-\norm{T}}, \quad \norm{Q(\lambda)}\le \cfrac{r\norm{T}}{1-r\norm{T}}\le\cfrac{\norm{T}}{1-\norm{T}}.$$	
							\end{enumerate} 
						\end{lemma}
						\begin{proof} The first part of the lemma is verified by direct computation, using
							\begin{equation*}
								\begin{array}{rl}
									\left(I-T/\lambda\right)^{-1}=&\left(I-T/\lambda^*\right)\left[\left(I-T/\lambda\right)\left(I-T/\lambda^*\right)\right]^{-1},\quad\\ \left(I-T^*/\lambda\right)^{-1}=&\left[\left(I-T^*/\lambda^*\right)\left(I-T^*/\lambda\right)\right]^{-1}\left(I-T^*/\lambda^*\right)
									\\
									\text{and }
									\left(I-T/\lambda\right)\left(I-T/\lambda^*\right)=&I-2r\cos\phi \, T +r^2 T^2.
								\end{array}
							\end{equation*}
							After that, with the help of Lemma \ref{lem:inv(I-T/z)}, it is not difficult to show the inequalities in (i). To prove (ii), first observe that the two series 
							$$\sum_{k=0}^\infty r^k\cos(k\phi) T^k\quad\mbox{and}\quad\sum_{k=1}^\infty r^k\sin(k\phi) T^k$$
							converge. Then, by expanding and simplifying the left-hand sides, we can show that 
							\begin{equation*}
								\bigl(\displaystyle\sum_{k=0}^\infty r^k\cos(k\phi) T^k\bigr)(I-2r\cos\phi \, T +r^2 T^2)=I-r\cos\phi \, T
							\end{equation*}
							and
							\begin{equation*}
								\bigl(\displaystyle\sum_{k=1}^\infty r^k\sin(k\phi) T^k\bigr)(I-2r\cos\phi \, T +r^2 T^2)= r\sin\phi \, T,
							\end{equation*}
							so $P(\lambda)$ and $Q(\lambda)$ can be expressed as the series above, and the inequalities in (ii) follow.
						\end{proof}
						% \begin{cmt}
							% 	The original idea of Lemma \ref{lem:decomPQ} comes from $T\in\R^{n\times n}$, in that case  $P(\lambda)=\Re[\left(I-T/\lambda\right)^{-1}]$ and $ Q(\lambda)=\Im[\left(I-T/\lambda\right)^{-1}]$.
							% \end{cmt}
						%	\begin{cmt}
							%		Note that (for $\norm{T}>0$)
							%		$$\max\left\{\cfrac{1+r|\cos\phi|\norm{T}}{(1-r\norm{T})^2}: r\le 1,\phi\in[-\pi,\pi]\right\}= \cfrac{1+\norm{T}}{(1-\norm{T})^2}\quad\text{but}\quad\cfrac{1+\norm{T}}{(1-\norm{T})^2}>\cfrac{1}{1-\norm{T}}$$
							%		and
							%		$$\max\left\{\cfrac{r|\sin\phi|\norm{T}}{(1-r\norm{T})^2}:r\le 1,\phi\in[-\pi,\pi]\right\}=\cfrac{\norm{T}}{(1-\norm{T})^2}\quad\text{but}\quad\cfrac{\norm{T}}{(1-\norm{T})^2}>\cfrac{\norm{T}}{1-\norm{T}}.$$
							%		This means that (when $0<\norm{T}<1$) the bounds in (ii) depending only on $\norm{T}$ are better than those that can be obtained by ; on the other hand, the bound for $Q$ in  gives a better estimation when $\phi\to 0$. 
							%	\end{cmt}
						
						In Sections \ref{subsec:seim-1-shot:loca-eigen} and \ref{subsec:seim-k-shot:complex-eigen} we identify different cases of $\lambda\in\C$ and we need corresponding estimations, given in the  following lemma.
						
						\begin{lemma}
							\label{lem:gamma123}
							For $\lambda\in\C\setminus{\R}, |\lambda|\ge 1$ we write $\lambda=R(\cos\theta+\ic\sin\theta)$ in polar form where $R\ge 1$, $\theta\in(-\pi,\pi)$, $\theta \ne 0$.  
							\begin{enumerate}[label=(\roman*)]
								\item For $\lambda$ satisfying $\Re(\lambda^2-\lambda)\ge 0$, let $\gamma_1=\gamma_1(\lambda)\coloneqq\left\{\begin{array}{cc}
									1 &\text{if } \Im (\lambda^2-\lambda)\ge 0,\\
									-1 &\text{if } \Im (\lambda^2-\lambda)<0
								\end{array}\right.$ then
								$$\Re(\lambda^2-\lambda)+\gamma_1\Im(\lambda^2-\lambda)\ge |\lambda(\lambda-1)|\ge 2|\sin(\theta/2)|.$$
								\item Let $0<\theta_0\le\frac{\pi}{4}$. For $\lambda$ satisfying $\Re(\lambda^2-\lambda)<0$ and $\theta\in[\theta_0,\pi-\theta_0]\cup[-\pi+\theta_0,-\theta_0]$,  let $\gamma_2=\gamma_2(\lambda)\coloneqq\left\{\begin{array}{cc}
									-1 &\text{if } \Im (\lambda^2-\lambda)\ge 0,\\
									1 &\text{if } \Im (\lambda^2-\lambda)<0
								\end{array}\right.$ then 
								$$|\Re(\lambda^2-\lambda)+\gamma_2\Im(\lambda^2-\lambda)| \ge |\lambda(\lambda-1)| \ge 2\sin({\theta_0}/{2}).$$
								\item Let $0<\theta_0\le\frac{\pi}{4}$ and $\delta_0>0$ . 
								For $\lambda$ satisfying $\Re(\lambda^2-\lambda)<0$ and $\theta\in(-\theta_0,\theta_0)\backslash\{0\}$, let $\gamma_3=\gamma_3(\mathrm{sign}(\theta))\coloneqq\left\{\begin{array}{cc}
									\left(\delta_0+\sin\frac{3\theta_0}{2}\right)/\cos\frac{3\theta_0}{2} & \text{if }\theta>0,\\
									-\left(\delta_0+\sin\frac{3\theta_0}{2}\right)/\cos\frac{3\theta_0}{2} & \text{if }\theta<0
								\end{array}\right.$ then
								\begin{equation*}
									\Re(\lambda^2-\lambda)+\gamma_3\Im(\lambda^2-\lambda)\ge 2\delta_0|\sin(\theta/2)|
								\end{equation*}
								and
								\begin{equation*}
									\frac{|\Re(\lambda-1)+\gamma_3\Im(\lambda-1)|}{\Re(\lambda^2-\lambda)+\gamma_3\Im(\lambda^2-\lambda)}\le\frac{\sqrt{1+\gamma_3^2}}{\delta_0}.
								\end{equation*}
								Moreover, if $0<\theta_0<\frac{\pi}{4}$ then
								\begin{equation*}
								\frac{|\gamma_3\Re(\lambda-1)-\Im(\lambda-1)|}{\Re(\lambda^2-\lambda)+\gamma_3\Im(\lambda^2-\lambda)}\le\max\left(\frac{\sqrt{1+\gamma_3^2}}{\delta_0},\frac{\sqrt{1+\gamma_3^2}}{\cos2\theta_0}\right).
								\end{equation*}
								\item Let $0<\theta_0\le\frac{\pi}{4}$. 
								There exists no $\lambda$ satisfying $\Re(\lambda^2-\lambda)<0$ and $\theta \in (\pi-\theta_0,\pi)\cup(-\pi,-\pi+\theta_0)$.
							\end{enumerate}
						\end{lemma}
						\begin{proof}
							Notice that $\gamma_1^2=1$, $\gamma_1\Im(\lambda^2-\lambda)\ge 0$. We have
							\begin{multline*}
								\left[\Re(\lambda^2-\lambda)+\gamma_1\Im(\lambda^2-\lambda)\right]^2
								=\left[\Re(\lambda^2-\lambda)\right]^2+\left[\Im(\lambda^2-\lambda)\right]^2 +2\gamma_1 \Re(\lambda^2-\lambda)\Im(\lambda^2-\lambda)\\
								\ge \left[\Re(\lambda^2-\lambda)\right]^2+\left[\Im(\lambda^2-\lambda)\right]^2
								=|\lambda^2-\lambda|^2,\\
							\end{multline*}
							which yields $\Re(\lambda^2-\lambda)+\gamma_1\Im(\lambda^2-\lambda)\ge |\lambda(\lambda-1)|$. Finally,
							$$|\lambda-1|=|R\cos\theta-1+\ic R\sin\theta|=\sqrt{R^2+1-2R\cos\theta}\ge \sqrt{2-2\cos\theta} = 2\left|\sin\frac{\theta}{2}\right|$$
							since the function $R\mapsto R^2+1-2R\cos\theta$, for $R\ge 1$, is increasing.
							
							\noindent (ii) In this case we have $\frac{\theta}{2}\in \left[\frac{\theta_0}{2},\frac{\pi}{2}-\frac{\theta_0}{2}\right]\cup\left[-\frac{\pi}{2}+\frac{\theta_0}{2},-\frac{\theta_0}{2}\right]$ so $\left|\sin\frac{\theta}{2}\right|\ge\sin\frac{\theta_0}{2}$. We also notice that $\gamma_2^2=1$ and $\gamma_2\Im(\lambda^2-\lambda)\le 0$. Similar to (i), we have
							$|\Re(\lambda^2-\lambda)+\gamma_2\Im(\lambda^2-\lambda)|=-\Re(\lambda^2-\lambda)-\gamma_2\Im(\lambda^2-\lambda)\ge |\lambda(\lambda-1)|\ge 2|\sin(\theta/2)|$, that implies the conclusion.
							
							\noindent (iii) Note that $\cos2\theta>0, -\frac{\pi}{2}<2\theta<\frac{\pi}{2}$, and $\sin 2\theta$ has the same sign as $\theta$ and $\gamma_3$, so we have 
							$$\begin{array}{ll}
								\Re(\lambda^2-\lambda)+\gamma_3\Im(\lambda^2-\lambda)&=R(R\cos 2\theta-\cos\theta+\gamma_3 R\sin 2\theta-\gamma_3\sin \theta)\\
								&\ge \cos 2\theta-\cos \theta+\gamma_3\sin 2\theta-\gamma_3\sin\theta\\
								&= -2\sin\frac{3\theta}{2}\sin\frac{\theta}{2}+2\gamma_3\cos\frac{3\theta}{2}\sin\frac{\theta}{2} \\
								&= 2\sin\frac{\theta}{2}\left(\gamma_3\cos\frac{3\theta}{2}-\sin\frac{3\theta}{2}\right).
							\end{array}
							$$
							Then we consider two cases: if $0<\theta<\theta_0$ then $\gamma_3>0$, $\left|\sin\frac{\theta}{2}\right|=\sin\frac{\theta}{2}>0$, 
							$0<\frac{3\theta}{2}<\frac{3\theta_0}{2}<\frac{\pi}{2}$ and $\gamma_3\cos\frac{3\theta}{2}-\sin\frac{3\theta}{2}>\gamma_3\cos\frac{3\theta_0}{2}-\sin\frac{3\theta_0}{2}=\delta_0$;	if $-\theta_0<\theta<0$ then $-\gamma_3>0$, $\left|\sin\frac{\theta}{2}\right|=-\sin\frac{\theta}{2}>0$, 
							$-\frac{\pi}{2}<-\frac{3\theta_0}{2}<\frac{3\theta}{2}<0$ and $-\gamma_3\cos\frac{3\theta}{2}+\sin\frac{3\theta}{2}>-\gamma_3\cos\frac{3\theta_0}{2}-\sin\frac{3\theta_0}{2}=\delta_0$.
							
							Next, we will show that $\frac{|\Re(\lambda-1)+\gamma_3\Im(\lambda-1)|}{\Re(\lambda^2-\lambda)+\gamma_3\Im(\lambda^2-\lambda)}$ is bounded. First,
							\begin{multline*}
								\frac{|\Re(\lambda-1)+\gamma_3\Im(\lambda-1)|}{\Re(\lambda^2-\lambda)+\gamma_3\Im(\lambda^2-\lambda)}=\frac{|(\cos\theta+\gamma_3\sin\theta)R-1|}{R[(\cos 2\theta+\gamma_3\sin 2\theta)R-(\cos\theta+\gamma_3\sin\theta)]}
								\\
								\le \frac{|(\cos\theta+\gamma_3\sin\theta)R-1|}{(\cos 2\theta+\gamma_3\sin2\theta)R-(\cos\theta+\gamma_3\sin\theta)}.
							\end{multline*}
							Since $\gamma_3$ does not depend on $R$, let us study $f_1(R)\coloneqq\left(\frac{aR-1}{bR-a}\right)^2$ where
							$a\coloneqq\cos\theta+\gamma_3\sin\theta$ and $b\coloneqq\cos 2\theta+\gamma_3\sin 2\theta$. We observe that:
							\begin{itemize}
								\item $a>0$ and $b>0$. Indeed, $\cos\theta>0$, $\cos 2\theta>0$, and $\theta$ and $\gamma_3$ have the same sign.
								\item $bR-a>0$ since $\Re(\lambda^2-\lambda)+\gamma_3\Im(\lambda^2-\lambda)>0$, thus  $R>\frac{a}{b}$.
								\item $a^2>b$ (equivalently $\frac{a}{b}>\frac{1}{a}$), since $a^2=\cos^2\theta+\gamma_3^2\sin^2\theta+\gamma_3\sin 2\theta>\cos^2\theta-\sin^2\theta+\gamma_3\sin 2\theta=b$. 
							\end{itemize}
							Now, $f_1'(R)=2\cdot\frac{aR-1}{bR-a}\cdot\frac{b-a^2}{(bR-a)^2}<0$
							for $R>\frac{a}{b}>\frac{1}{a}$ and we would like to have $\frac{a}{b}<1$ so that $f_1(R)\le f_1(1), \forall R\ge 1$. Indeed $\frac{a}{b}< 1$ is equivalent to
							$$\cos\theta+\gamma_3\sin\theta < \cos2\theta+\gamma_3\sin2\theta  \Leftrightarrow|\gamma_3|>\frac{\left|\sin\frac{3\theta}{2}\right|}{\cos\frac{3\theta}{2}},$$
							which is true since
							$$|\gamma_3|=\frac{\delta_0+\sin\frac{3\theta_0}{2}}{\cos\frac{3\theta_0}{2}}>\frac{\left|\sin\frac{3\theta}{2}\right|}{\cos\frac{3\theta}{2}}+\varepsilon_0 \quad\mbox{where}\quad\varepsilon_0=\frac{\delta_0}{\cos\frac{3\theta_0}{2}}.$$
							Then we study
							$$f_1(1)=\left( \frac{\cos\theta-1+\gamma_3\sin\theta}{\cos2\theta-\cos\theta+\gamma_3(\sin2\theta-\sin\theta)}\right)^2=\left(\cfrac{-\sin\frac{\theta}{2}+\gamma_3\cos\frac{\theta}{2}}{-\gamma_3\sin\frac{3\theta}{2}+\gamma_3^2\cos\frac{3\theta}{2}}\right)^2\gamma_3^2.
							$$
							We have:
							\begin{itemize}
								\item $(-\sin\frac{\theta}{2}+\gamma_3\cos\frac{\theta}{2})^2\le 1+\gamma_3^2$ by Cauchy-Schwartz inequality;
								\item $\gamma_3^2=|\gamma_3|^2>\frac{\gamma_3\sin\frac{3\theta}{2}}{\cos\frac{3\theta}{2}}+\varepsilon_0|\gamma_3|$ that leads to
								\begin{equation*}
									-\gamma_3\sin\frac{3\theta}{2}+\gamma_3^2\cos\frac{3\theta}{2}>\varepsilon_0\cos\frac{3\theta}{2}|\gamma_3|=\delta_0|\gamma_3|;
								\end{equation*}
							\end{itemize}
							hence $f_1(1)\le\frac{1+\gamma_3^2}{\delta_0^2}$ and this yields $\frac{|\Re(\lambda-1)+\gamma_3\Im(\lambda-1)|}{\Re(\lambda^2-\lambda)+\gamma_3\Im(\lambda^2-\lambda)}\le\frac{\sqrt{1+\gamma_3^2}}{\delta_0}$. 
							
							Finally, we show that if $0<\theta_0<\frac{\pi}{4}$ then $\frac{|\gamma_3\Re(\lambda-1)-\Im(\lambda-1)|}{\Re(\lambda^2-\lambda)+\gamma_3\Im(\lambda^2-\lambda)}$ is bounded. Indeed, we have
							\begin{multline*}
								\cfrac{|\gamma_3\Re(\lambda-1)-\Im(\lambda-1)|}{\Re(\lambda^2-\lambda)+\gamma_3\Im(\lambda^2-\lambda)}=\cfrac{|(\gamma_3\cos\theta-\sin\theta)R-\gamma_3|}{R[(\cos 2\theta+\gamma_3\sin 2\theta)R-(\cos\theta+\gamma_3\sin\theta)]}
								\\
								\le \cfrac{|(\gamma_3\cos\theta-\sin\theta)R-\gamma_3|}{(\cos 2\theta+\gamma_3\sin2\theta)R-(\cos\theta+\gamma_3\sin\theta)}.
							\end{multline*}
							Since $\gamma_3$ does not depend on $R$, let us study $f_2(R)\coloneqq\left(\frac{cR-\gamma_3}{bR-a}\right)^2$
							where $c\coloneqq\gamma_3\cos\theta-\sin\theta$ and $a, b$ as above. We observe that:
							\begin{itemize}
								\item $\gamma_3b-ca$ and $\theta$ have the same sign. Indeed, $\gamma_3b-ca=(\gamma_3^2+1)\sin\theta\cos\theta$. Consequently, we always have $(\gamma_3b-ca)\gamma_3>0$.
								
								\item We always have $\frac{\gamma_3}{c}>1$. 	Indeed, if $\theta>0$ then $c>0$ since $\gamma_3=\frac{\delta_0+\sin\frac{3\theta_0}{2}}{\cos\frac{3\theta_0}{2}}>\frac{\sin\theta}{\cos\theta}$, also 
								$\frac{\gamma_3}{c}=\frac{\gamma_3}{\gamma_3\cos\theta-\sin\theta}>1$; if $\theta<0$ then $c<0$ since 
								$-\gamma_3=\frac{\delta_0+\sin\frac{3\theta_0}{2}}{\cos\frac{3\theta_0}{2}}>-\frac{\sin\theta}{\cos\theta}$,	also $\frac{\gamma_3}{c}=\frac{-\gamma_3}{-\gamma_3\cos\theta+\sin\theta}>1$.
							\end{itemize}
							Now, $f_2'(R)=2\cdot\frac{\frac{c}{\gamma_3}R-1}{bR-a}\cdot\frac{(\gamma_3b-ca)\gamma_3}{(bR-a)^2}$, so, thanks to the above results, $f_2(R)$ decreases for $1\le R < \frac{\gamma_3}{c}$ and increases for $R > \frac{\gamma_3}{c}$.  
							Moreover, like for $f_1(1)$, we can estimate
							$$f_2(1)=\left(\cfrac{-\cos\frac{\theta}{2}-\gamma_3\sin\frac{\theta}{2}}{-\gamma_3\sin\frac{3\theta}{2}+\gamma_3^2\cos\frac{3\theta}{2}}\right)^2\gamma_3^2\le\cfrac{1+\gamma_3^2}{\delta_0^2}
							$$
							and $\lim_{R\to+\infty}f_2(R)=\left(\frac{\gamma_3\cos\theta-\sin\theta}{\cos2\theta+\gamma_3\sin2\theta}\right)^2\le \frac{1+\gamma_3^2}{\cos2\theta_0}$. Therefore
							$$\frac{|\gamma_3\Re(\lambda-1)-\Im(\lambda-1)|}{\Re(\lambda^2-\lambda)+\gamma_3\Im(\lambda^2-\lambda)}\le\max\left(\frac{\sqrt{1+\gamma_3^2}}{\delta_0},\frac{\sqrt{1+\gamma_3^2}}{\cos2\theta_0}\right).$$
							
							\noindent (iv)  For $\theta \in (\pi-\theta_0,\pi)\cup(-\pi,-\pi+\theta_0)$, we have $\cos 2\theta>0$ since $2\theta \in\left(\frac{3\pi}{2},2\pi\right)\cup\left(-2\pi,-\frac{3\pi}{2}\right)$, while $\cos \theta<0$. 
							Hence $\Re(\lambda^2-\lambda)=R(R\cos 2\theta-\cos\theta)>0$.
						\end{proof}
						% \begin{remark}
							% 	There exists non-trivial $\lambda$ satisfying conditions mentioned in each of cases , (ii) and (iii)-(iv) in Lemma \ref{lem:gamma123}. For instance, take $\theta_0=\frac{\pi}{4}-\epsilon$ for small $\epsilon>0$ then take $\lambda=2+\frac{1}{2}\ic$ for , $\lambda=\ic$ for (ii) and $\lambda=\frac{3}{2}\e^{\frac{\ic\pi}{6}}$ for (iii)-(iv).
							% \end{remark}

					\end{document}